\documentclass[11pt, reqno]{amsart}

\usepackage{amsmath,amsfonts,amssymb,amsthm,color,hyperref}

\usepackage{mathrsfs}
\usepackage{soul}

\usepackage[T1]{fontenc}

 \usepackage{graphicx}
\usepackage{subfigure} 

\makeatletter
     \def\section{\@startsection{section}{1}%
     \z@{.7\linespacing\@plus\linespacing}{.5\linespacing}%
     {\bfseries
     \centering
     }}
     \def\@secnumfont{\bfseries}
     \makeatother

\def\R{\mathbb{R}}
\def\P{\mathbb{P}}
\def\C{\mathbb{C}}

\def\wt{\widetilde}
\def\wh{\widehat}
\def\N{\mathbb{N}}

\def\H{\mathfrak{H}}
\def\F{\mathscr{F} }

 \def\E{\mathbb{E}}

\def\i{\mathbf{i}}
\def\to{\rightarrow}

\def\lmto{\longmapsto}

\def\bv{\big\vert}

 \def\e{\varepsilon}

\def\H{\mathfrak{H}}

\def\Var{{\rm Var}}

\newtheorem{theorem}{Theorem}[section]
\newtheorem{lemma}[theorem]{Lemma}
\newtheorem{proposition}[theorem]{Proposition}
\newtheorem{corollary}[theorem]{Corollary}
\newtheorem{definition}[theorem]{Definition}

\theoremstyle{remark}
\newtheorem{example}[theorem]{Example}

\theoremstyle{remark}
\newtheorem{remark}[theorem]{Remark}

\numberwithin{equation}{section}

\begin{document}

 \title[Averaging Gaussian functionals]{Averaging Gaussian functionals}


\author{David Nualart \and Guangqu Zheng}

\address{David Nualart: University of Kansas, Mathematics department, Snow Hall, 1460 Jayhawk blvd, Lawrence, KS 66045-7594, United States} \email{nualart@ku.edu} \thanks{D. Nualart is supported by NSF Grant DMS 1811181.}

\address{Guangqu Zheng: University of Kansas, Mathematics department, Snow Hall, 1460 Jayhawk blvd, Lawrence, KS 66045-7594, United States} \email{zhengguangqu@gmail.com}



\maketitle

\vspace*{-1cm}

\begin{abstract}    This paper consists of two parts. In the first part, we focus on the average of a functional over shifted Gaussian homogeneous noise  and as the averaging domain covers the whole space,
we establish a Breuer-Major type Gaussian fluctuation based on   various assumptions on the covariance kernel and/or the spectral measure.  Our methodology for the first part begins with the application of  Malliavin calculus  around Nualart-Peccati's Fourth Moment Theorem, and in addition we apply the Fourier techniques as well as a soft approximation argument based on  Bessel functions of first kind.

   The same methodology leads us to investigate a closely related problem in the second part. We study the spatial average of a linear stochastic heat equation driven by space-time Gaussian colored noise.  The temporal  covariance kernel $\gamma_0$ is  assumed to be {\it locally integrable} in this paper.   If the spatial covariance kernel is  {\it nonnegative and integrable on the whole space}, then  the spatial average admits Gaussian fluctuation; with some extra mild integrability condition on $\gamma_0$, we are able to provide a functional central limit theorem.  These results complement recent studies on the spatial average for SPDEs.  Our analysis also allows us to consider the case where the spatial covariance kernel is not integrable: For example, in the case of  the Riesz kernel, the first chaotic component of the spatial average is dominant so that the Gaussian fluctuation also holds true.  
    \end{abstract}

\allowdisplaybreaks

\tableofcontents

\section{Introduction}

Motivated by the  Breuer-Major central limit theorem (CLT) \cite{BM83} and recent studies on the  spatial averages of SPDEs \cite{HNV18, HNVZ19, DNZ18},  we devote this paper to seeking   general conditions that lead to  the Gaussian fluctuations of averages of Gaussian functionals.

 Let us briefly introduce our framework. Let $W$ be a  $d$-dimensional homogenous Gaussian noise  with covariance kernel  $\gamma$,   that is, $W= \big\{ W(\phi) ,\phi\in C_c^\infty(\R^d) \big\}$  is a centered Gaussian family of real random variables,
defined on a probability space $(\Omega, \mathcal{F}, \P)$,  with covariance structure given by 
\begin{align}\label{nncov}
 \E\big[ W(\phi) W(\varphi) \big] = \int_{\R^{2d}} \phi(x) \varphi(y) \gamma(x-y)\, dx \, dy\,,~ \forall \phi, \varphi\in C_c^\infty(\R^d),
\end{align}
 where $\gamma:\R^d\to\R\cup\{+\infty\}$ is symmetric with  $\gamma^{-1}(\{ \infty\})\subset \{0\}$ and  
$
   \gamma(x) = (\F\mu)(x)= \int_{\R^d} e^{-\i x \cdot \xi} \mu(d\xi)
$  
for some  nonnegative  tempered measure $\mu$ on $\R^d$.  These assumptions on $\gamma$ ensure that \eqref{nncov} defines a nonnegative definite covariance functional and $\mu$ is known as the spectral measure. Notice that    $\gamma(0)\in\R$ is equivalent to the finiteness of $\mu(\R^d)$.

It is clear that \eqref{nncov} defines an inner product, under which the space $ C_c^\infty(\R^d)$ can be extended into a real   Hilbert space $\H$. Furthermore, the mapping $\phi\in C_c^\infty(\R^d)\mapsto W(\phi)$ extends to a linear isometry between $\H$ and the Gaussian Hilbert space spanned by $W$. We write 
$
W(\phi) = \int_{\R^d} \phi(x) \, W(dx)
$
and 
$\E[ W(\phi) W(\varphi) \big] = \langle \phi, \varphi \rangle_\H$, for any $\phi, \varphi\in\H$.
This gives us an {\it isonormal Gaussian process}  over $\H$.

Now consider a real random variable $F\in L^2(\Omega)$ that is measurable with respect to $W$ and has the following Wiener chaos expansion:
\begin{align}\label{Fchaos}
F(W) = \E[F] + \sum_{p\geq 1} I_p^W(f_p)\,,
\end{align}
where $I_p^W(\cdot)$ denotes the $p$th multiple stochastic integral with respect to $W$ and $f_p$ belongs to the  symmetric subspace $\H^{\odot p}$ of the $p$th tensor product $\H^{\otimes p}$,   $\forall p\in\N$; see  \cite{Nualart06} for more details. Along the paper we will denote by $\Pi_p F$ the orthogonal projection of $F$ onto the $p$th Wiener chaos.

 In order to formulate our results, we need to introduce   the {\it spatial shifts} $\{U_x, x\in\R^d\}$. For each $x\in\R^d$ and $F$ given as in \eqref{Fchaos}, $U_xF$ is defined by 
 \begin{align}
U_xF :=  \E[F] + \sum_{p\geq 1} I_p^W(f_p^x), \label{Fchaosx}
 \end{align}
 with\footnote{  For   a generalized function $f\in\H$, we can define $f^x$ as follows.  Let $\{f_n, n\in\N\}\subset C^\infty_c(\R^d)$ be an approximating sequence of $f$ in $\H$, we can define $f_n^x$ for each $n\in\N$ and  $f^x$ to be the limit of the Cauchy sequence $\{ f_n^x,n\in\N\}$ in $\H$.  It is routine to verify that the definition of $f^x$ does not depend on the particular choice of the approximating sequence.   } $ f_p^x(y_1,\ldots, y_p) = f_p(y_1-x,\ldots, y_p-x)$ for any $x, y_1,\ldots, y_p\in\R^d$ and $p\in\N$. Here is another look at the above definition.
For any $x\in \R^d$  and any $\varphi\in C_c^\infty(\R^d)$, we write   
$ \varphi^x(y) =\varphi(y-x)$ and we  introduce $W_x$,  the shifted Gaussian field,  defined by $W_x(\phi) = W(\phi^x)$, for any $\phi\in C^\infty_c(\R^d)$, and by extension for any $\phi\in\H$.
The family $W_x$ has the same covariance structure as $W$ and  the associated multiple stochastic integrals satisfy
$
I^{W_x}_p(f)  = I_p^W(f^x)
$
for any    $f\in \H^{\odot p} $, so that $U_xF(W) = F(W_x)$ shall give us \eqref{Fchaosx}.

\medskip

Let $F$ be given as in \eqref{Fchaos}. We are interested in the spatial averages of $U_xF$  over $B_R=\{ x\in\R^d: \| x \| \leq R \}$,
with the particular aim at general conditions on the kernels $\{f_p,p\in\N\}$ and the covariance kernel $\gamma$ (and/or the associated spectral measure $\mu$) that imply
\begin{align}\label{Q0}
\frac{1}{\sigma(R)}\int_{B_R} U_xF \, dx \xrightarrow[R\to+\infty]{\rm law} N(0,1)\,,
\end{align}
where $\sigma(R)$ is a normalization constant and $N(m, v^2)$ stands for a real normal distribution with mean $m$ and variance $v^2$.

\smallskip

To illustrate how this spatial averaging is related to the aforementioned Breuer-Major theorem and to give a flavor of our results, we provide below a particular case (see Example \ref{ex12})  and refer to Section \ref{sec2} for more general results.   Let us first recall the continuous-time Breuer-Major theorem (in a slightly different form).

\begin{theorem}\label{thm0}   Suppose $g\in L^2(\R, e^{-x^2/2}dx)$ has the following orthogonal expansion in Hermite polynomials $\{ H_p = (-1)^p e^{x^2/2} \frac{d^p}{dx^p}e^{-x^2/2}, p\in\N  \}:$
\[
\qquad\qquad\qquad g = \sum_{p\geq m} c_p H_p \quad\text{with $c_m\neq 0$, $m\geq 1$ known as the Hermite rank of $g$}.
\]
Let $Y=\{Y_x,x\in\R^d\}$ be a centered Gaussian stationary process with covariance function $\E[Y_a Y_b ]= \rho(   a-b  )$ such that $\rho(0)=1$. Under the condition $\rho\in L^m(\R^d, dx)$, we have
\[
R^{-d/2} \int_{B_R} g(Y_x)\, dx \xrightarrow[R\to+\infty]{\rm law}N(0,\sigma^2) \,,
\]
 with $ \sigma^2:= \omega_d\sum_{q\geq m}c_q^2 q! \int_{\R^d} \rho(x)^m~dx\in[0,\infty)$,  $\omega_d$ being the volume of $B_1$; see also \cite{CNN18, NZ19}.   
\end{theorem}

\begin{example}\label{ex12} Now fix a unit vector $e\in\H$ and put $F = g\big(W(e)\big)$, then $U_x F = g\big(W_x(e)\big)=g\big(Y_x\big)$, with $Y_x = W(e^x)$. If $g\in L^2(\R, e^{-x^2/2}dx)$ has Hermite rank $m\geq 1$ and 
\[
\int_{\R^d} \left\vert \int_{\R^{2d}} e(a) e(b) \gamma(a-b-x) dadb \right\vert^m\, dx < +\infty \,,
\]
then Theorem \ref{thm0} produces an example of \eqref{Q0}. Note that in this example, the Gaussian functional $F = g\big(W(e)\big)$ depends only on one coordinate while our principal concern  is for Gaussian functionals that may depend on infinitely many coordinates.
\end{example}

\medskip

Recall the chaos expansions \eqref{Fchaos} and \eqref{Fchaosx}, and from now on, we consider the  case where  $F$ has {\it Hermite rank} $m\geq 1$, meaning that:
\begin{center}
$\E[F]=0$, $\{f_j, j=1,\ldots, m-1\}$ are zero vectors and    $f_m\in\H^{\odot m}$ is nonzero.
\end{center}
In this case, we write 
\[
\int_{B_R} U_xF ~ dx = \sum_{p\geq m} I_p^W\big( g_{p,R}\big) ~\text{with}~ g_{p,R}=\int_{B_R} f_p^x~ dx ~ \text{for each $p\geq m$.}
\]
In view of Hu and Nualart's chaotic central limit theorem \cite{HN05}, based on the Fourth Moment Theorems of Nualart, Peccati and Tudor \cite{FMT, PTudor},  it is enough to look for conditions that guarantee the central limit theorem on each fixed chaos, provided one has some uniform control of the variance of each chaotic component. More precisely,  we have the following general result. 

\begin{theorem}\label{thm00} Consider a sequence of centered square integrable random variables $(F_n, n\in\N)$ with Wiener chaos expansions $F_n = \sum_{q\geq 1} I^W_q(f_{q,n})$, where $f_{q,n}\in\H^{\odot q}$ for each $q,n\in\N$. Suppose that:
\begin{itemize}
\item[(i)]  $\forall q\geq 1$, $q! \| f_{q,n} \|^2_{\H^{\otimes q}}\to \sigma_q^2$, as $n\to+\infty$;

\item[(ii)] $\forall q\geq 2$ and  $\forall r\in\{1, \ldots, q-1\}$, $\| f_{q,n}  \otimes_r f_{q,n} \| _{\H^{\otimes (2q-2r)}} \to 0$, as $n\to+\infty$;

\item[(iii)]  $\lim_{N\to+\infty} \limsup_{n\to+\infty} \sum_{q\geq N} q! \| f_{q,n}\| _{\H^{\otimes q}}^2 = 0 \,.$

\end{itemize}
Then, as $n\to\infty$, $F_n$ converges in law to $N(0, \sigma^2)$, with $\sigma^2 = \sum_{q\geq 1} \sigma_q^2$. 
\end{theorem}

We refer to \cite{bluebook, Eulalia} for more details  on this result and to  Section \ref{sec2} for the definition of   the $r$-contraction $\otimes_r$.
 
Now let us look at the  central limit theorem  on each chaos.  We fix an integer $p\geq 2$ and put 
$$G_{p,R} =  I_p^W\big( g_{p,R}\big)$$ with
$\sigma_{p,R}^2: = \Var\big( G_{p,R} \big)$.  Assume  $\sigma_{p,R}>0$ for large $R$, then according to
     the Fourth Moment Theorem of Nualart and Peccati \cite{FMT}, we know that
\[
\frac {G_{p,R}} { \sigma_{p,R}}      \xrightarrow[R\to+\infty]{\rm law } N(0,1)
\]
 if and only if 
 \begin{align}  \label{contractions}
  \quad\lim_{R\to+\infty}  \frac{1}{\sigma^2_{p,R}}  \sum_{r=1}^{p-1} \left\| g_{p,R} \otimes_r g_{p,R} \right\|_{\H^{\otimes(2p-2r)}} = 0 \,.
\end{align}
 Moreover, we have the  following rate of convergence in the total variation distance, as a consequence of the Nourdin-Peccati bound  (see \cite[Chapter 5]{bluebook}):
\begin{equation} \label{dist}
d_{\rm TV}\left( \frac {G_{p,R}} { \sigma_{p,R}}, N(0,1)\right) \le \frac{C}{\sigma_{p,R}^2} \sum_{r=1} ^{p-1} \left\| g_{p,R} \otimes_r g_{p,R} \right\|_{\H^{\otimes(2p-2r)}} \,.
 \end{equation} Throughout this paper, we write $C$ for   immaterial constants that may vary from line to line. 
 
 In the first part of this paper (Section \ref{sec2}),  we will exploit the above ideas to derive sufficient conditions for \eqref{Q0} to hold, {\it with $\sigma(R)$ growing like $C R^{d/2}$}. Note that the order of $\sigma(R)$ matches the result in Theorem \ref{thm0}. Without introducing further notation,  we provide    another   example of \eqref{Q0}, which is a corollary of our main result (Theorem \ref{CCLT1});  see Remark \ref{rem66}.

\begin{theorem}\label{cor0}  Let the above notation prevail. Suppose $\gamma(0)\in(0,+\infty)$ and $\gamma\in L^m(\R^d, dx)$, where $m\geq 1$ is the Hermite rank of $F$. If we assume in addition that the kernels $f_p\in L^1(\R^{pd})\cap \H^{\odot p}$, $p\geq m$,  satisfy
\begin{align}\label{gamma0}
\sum_{p\geq m} p! \gamma(0)^{p} \| f_p\|_{L^1(\R^{pd})}^2 < +\infty\,,
\end{align}
then, 
${\displaystyle
R^{-d/2}\int_{B_R} U_x F \, dx  \xrightarrow[R\to+\infty]{\rm law} N(0, \sigma^2),
}$
with 
$$ \sigma^2= \omega_d \sum_{p\geq m} p!  \int_{\R^{2dp}} f_p(\pmb{s_p}) f_p(\pmb{t_p}) \left(\int_{\R^d} \prod_{j=1}^p \gamma(t_j-s_j+z)~dz\right) d\pmb{s_p} d\pmb{t_p}\in[0,\infty)$$
with  $\pmb{s_p}=(s_1, \dots, s_p)$, $d\pmb{t_p} =  d t_1 \cdots dt_p$ and $\omega_d$ being the volume of $B_1=\{ \| x\| \leq 1 \}$.
\end{theorem}

One may want to compare our Theorem \ref{cor0} with Theorem \ref{thm0} and Example \ref{ex12}. We refer the readers to Section \ref{sec2} for more results with this flavor and here we briefly give a literature overview:

\begin{enumerate}
\item To the best of our knowledge,  problem \eqref{Q0} first received   attention in the  1976 paper \cite{Maruyama76} by Maruyama, using  the method of moments.  Proofs and  extensions of Maruyama's CLT were published in his 1985 paper \cite{Maruyama85}.

\item In 1983, Breuer and  Major provided a  CLT  \cite{BM83}, motivated by  the non-central limit theorems of Dobrushin, Major, Rosenblatt and Taqqu during  1977-1981 (see \cite{DM79,Major81,Rosenblatt81,Taqqu77,Taqqu79}). Unlike these works,  Breuer and Major were interested at the asymptotic normality of nonlinear functionals over stationary  Gaussian  fields when the corresponding  correlation function decay fast enough.
Although Breuer-Major's theorem (see Theorem \ref{thm0}) takes a simpler form compared to Maruyama's CLT,  it has found a tremendous amount of applications in theory and practice.

\item Chambers and Slud  established further extensions to Maruyama's CLT in \cite{CS89} and   obtained  the Breuer-Major theorem as a corollary (when assuming the existence of spectral density). In both \cite{CS89} and Maruyama's work \cite{Maruyama76, Maruyama85} , the story always begins with  a real stationary Gaussian process with time-shifts $\{U_s, s\in\R\}$ and they formulated the chaos expansion based on the spectral (probability) measure.  

\item  In the present work, we  provide sufficient conditions for \eqref{Q0} in terms of the spectral measure. Comparing our assumptions based on the spectral measure with those in \cite{CS89}, both sets of assumptions  essentially cover our Theorem \ref{cor0} as a   particular case, while they are different in their full generality. Moreover, we also provide sufficient conditions for \eqref{Q0} in terms of the covariance kernel. 

\end{enumerate}

Our methodology from the first part can be applied to the study of  spatial averages of the  stochastic heat equation driven by Gaussian colored noise and this constitutes the   second part of our paper. More precisely, we consider the following   stochastic heat equation with a   multiplicative   Gaussian colored  noise on $\R_+\times\R^d$:
\begin{equation}  \label{pl1}
 \dfrac {\partial  u}{\partial t} = \dfrac{1}{2} \Delta u + u   \dot{W} 
 \end{equation}
where   the Laplacian $\Delta=\sum_{i=1}^d \partial_{x_i}^2$ concerns only spatial variables and the initial condition is fixed to be $u_{0,x}\equiv 1$.

The notation $\dot{W}$ stands for $\frac {\partial^{d+1} W}{ \partial t \partial x_1 \cdots  \partial x_d}$  and the noise $ W$ is formally defined
as a centered  Gaussian family   $\big\{ W(\phi), \phi \in C^\infty_c(\R_+\times\R^d)\big\}$,  with covariance structure
\begin{align}
\E [  W(\phi) W(\psi)  ]  & = \int_{\R_+^2}dsdt  \gamma_0(t-s)    
\langle \phi(s, \bullet), \gamma_1 * \psi(t,\bullet) \rangle_{L^2(\R^d)}
 \notag \\
&= \int_{\R_+^2}dsdt  \gamma_0(t-s)   \int_{\R^{d}} \mu_1(d\xi) \F \phi(s,\xi)  \F\psi(t,-\xi)  \,, \label{COVst}
\end{align}
for any $\phi,\psi\in C^\infty_c(\R_+\times\R^d)$,
where    $\F$ denotes the Fourier transform  with respect to the spatial variables and the following two conditions are satisfied:
\begin{enumerate}
\item $\gamma_0:\R\to[0,\infty]$ is  {\it locally  integrable} and nonnegative-definite,
\item  $\gamma_1$ is a measure, such that $\gamma_1 = \F \mu_1$ for some nonnegative tempered measure $\mu_1$, called the spectral measure,  satisfying   Dalang's condition (see \emph{e.g.} \cite{Dalang99})
\begin{equation} \label{Da}
\int_{\R^d} \frac{\mu_1(d\xi)}{1 + \| \xi \|^2} < +\infty.
\end{equation}
\end{enumerate}
If $\gamma_1$ is absolutely continuous with respect to the Lebesgue measure on $\R^d$, we still denote by $\gamma_1$ its density and then
\[
\langle \phi(s, \bullet), \gamma_1 * \psi(t,\bullet) \rangle_{L^2(\R^d)}
= \int_{R^{2d}} \phi(s,x) \gamma_1(x-y) \psi(t,y)dx dy.
\]
We will use this notation even if $\gamma_1$ is a measure.  The basic example is $d=1$ and $\gamma_1 =\delta_0$ and   in this case $\mu_1$ is $(2\pi)^{-1}$ times Lebesgue measure.

\smallskip
We   point out that \eqref{COVst} defines an inner product, under which $C^\infty_c(\R_+\times\R^d)$ can be extended into a Hilbert space $\mathscr{H}$. As we did before, we can build an isonormal process  $\{W(h), h\in\mathscr{H} \}$ from $\{W(h), h\in C^\infty_c(\R_+\times\R^d)  \}$.   We denote by $I_p^W(f)$  the  $p$th multiple integral of  a symmetric element $f\in\mathscr{H}^{\odot p}$.  For general $f\in\mathscr{H}^{\otimes p}$, we denote by $\wt{f}$ the canonical symmetrization of $f$, that is, 
 \begin{align*} 
\wt{f}(s_1, y_1, s_2, y_2, \ldots, s_p, y_p) = \frac{1}{p!} \sum_{\sigma\in \mathfrak{S}_p} f\big(s_{\sigma(1)}, y_{\sigma(1)}, \ldots, s_{\sigma(p)}, y_{\sigma(p)}\big),
\end{align*}
where the sum runs over the permutation group $ \mathfrak{S}_p$ over $\{1,  \dots,p\}$.  Quite often in this paper, we write $f(\pmb{s_p}, \pmb{y_p})$   for $f(s_1,y_1, \ldots, s_p, y_p)$, whenever it is convenient.

For each $t\geq 0$,  let  $\mathcal{F}_t$ be the $\sigma$-algebra generated by $\big\{W(\phi):\phi$ is continuous with support contained  in $[0,t]\times \R^d\big\}$. We say that a random field $u= \{ u_{t,x}, (t,x)\in\R_+\times\R^d\}$ is 
adapted if for each $(t,x)$, the random variable $u_{t,x}$ is $\mathcal{F}_t$-measurable.

\medskip

We interpret equation  \eqref{pl1} in the  Skorokhod  sense and recall the   definition of mild solution from  \cite[Definition 3.1]{HHNT15}.

\begin{definition}  An adapted random field $u = \big\{ u_{t,x},  t\geq 0, x\in\R^d\big\}$ such that $\E\big[ u_{t,x}^2 \big]<+\infty$ for all $(t,x)$ is said to be a {\it  mild solution} to  equation \eqref{pl1} with  initial conditoin $u_{0,\cdot}=1$, if for any $t\in\R_+$, $x\in\R^d$, the process $\{ G(t-s, x-y) u_{s,y}{\bf 1}_{[0,t]}(s): s\geq 0, y\in\R^d \}$ is Skorokhod integrable and 
\begin{align*}
u_{t,x} = 1 + \int_0^t  \int_{\R^d} G(t-s, x-y)u_{s,y}  W(ds,dy) \, ,
\end{align*}
where $G(t,x) = (2\pi t)^{-d/2} \exp\big( - \| x \| ^2/ (2t) \big)$ for $t>0$ and $x\in\R^d$.
\end{definition}

 The above stochastic heat equation  has a unique mild solution $u$ with explicit Wiener chaos expansion given by
 (see \cite[Theorem 3.2]{HHNT15})
 \begin{align*} 
 u_{t,x} = 1 + \sum_{n\geq 1} I_n^W(f_{t,x,n}),
 \end{align*}
 where
\begin{align}\label{form-f}
 f_{t,x,n}(\pmb{s_n}, \pmb{y_n}) = \frac{1  }{n!}  \prod_{i=0}^{n-1} G(s_{\sigma(i)}-s_{\sigma(i+1)},  y_{\sigma(i) }-y_{\sigma(i+1) }   ),
 \end{align}
 with $\sigma\in\mathfrak{S}_{n}$ being such that $t > s_{\sigma(1)} > \dots > s_{\sigma(n)} > 0$. 
 In the above expression we have used the convention $s_{\sigma(0)}=t$ and $y_{\sigma(0)} =x$.  We also refer interested readers to \cite{HHNT18,HLN} for more general noises.

  Notice that   $u_{t,x} -\E[ u_{t,x}]$ has Hermite rank $1$ and   it is known that  for any fixed $t\in\R_+$, $\{ u_{t,x} : x\in\R^d\}$ is strictly stationary meaning that  the finite-dimensional
distributions of the process $\{u_{t, x + y}, x\in\R^d\}$ do not depend on $y$.  So   the following integral 
\begin{align} \label{sa}
\int_{B_R} \big( u_{t,x} -1 \big)\, dx
\end{align}
 resembles the object in \eqref{Q0} and we are able to establish its Gaussian fluctuation under some mild assumptions. The spatial averages \eqref{sa}  have been studied in recent articles \cite{HNV18,HNVZ19,DNZ18}:
\begin{itemize}
 \item[(i)] Huang, Nualart and Viitasaari  \cite{HNV18} initiated their study by looking at the one-dimensional (nonlinear) stochastic heat equation driven by a space-time white noise.

  \item[(ii)] Huang, Nualart, Viitasaari and Zheng \cite{HNVZ19} continued to study  the $d$-dimensional stochastic heat equation driven by Gaussian noise that is white in time and colored in space, with the spatial covariance  described  by the Riesz kernel. 

 \item[(iii)] Delgado-Vences, Nualart and Zheng  \cite{DNZ18} carried out similar investigation for the one-dimensional stochastic wave equation.

\end{itemize}
  In the above references, the Gaussian noise is assumed to be white in time, which gives rise to a martingale structure. This is important for applying It\^o calculus (\emph{e.g.} Burkholder-Davis-Gundy  inequality and Clark-Ocone formula) to obtain  quantitative  central limit theorems for \eqref{sa}. 
  
   In the present paper, we consider a linear stochastic heat equation driven by space-time colored noise, so It\^o calculus can not be applied anymore; while due to the linearity, an explicit chaos expansion of the solution is available for us to apply the chaotic central limit theorem (Theorem \ref{thm00}).

  We define 
 \[
 A_t(R): = \int_{B_R} \big( u_{t,x} - 1 \big) \, dx
 \]
and let  $\Pi_pA_t(R)$  be the projection of $A_t(R)$ on the $p$th Wiener chaos, that is, 
$$\Pi_pA_t(R):=  I_p^W\left( \int_{B_R} f_{t,x,p}dx \right).$$

 Throughout this paper, we assume that $\gamma_0, \gamma_1$ are nontrivial, meaning that   
  \[
  \gamma_1(\R^d)> 0 \quad{\rm and}\quad   \int_0^t \int_0^t \gamma_0(r-v) drdv > 0
  \]
  for any $t>0$.     The following is our main result.
  
  \begin{theorem}\label{thmSHE}     Suppose $\gamma_0:\R\to\R_+\cup \{ +\infty\}$ is locally integrable,
  $\gamma_1$ satisfies Dalang's condition \eqref{Da} and  $\gamma_1(\R^d)< \infty $.   Then as $R\to+\infty$, $\{ R^{-d/2}A_t(R), \, t\ge 0\}$ converges to a    centered continuous  Gaussian process $ \{ \mathcal{G}_t, t\ge 0\}$ in finite-dimensional distributions. The covariance structure of   $\mathcal{G}$ is given by
\begin{align} \label{SIGst}
 \E [ \mathcal{G}_s \mathcal{G}_t]=:\Sigma_{s,t}   = \omega_d \int_{\R^d}\left( \E\left[ e^{\beta_{s,t}(z)}   \right] -1 \right)dz\in (0,\infty),
\end{align}
 where 
 \[
\beta_{s,t}(z):= \int_0^s \int_0^t \gamma_0(r-v)   \gamma_1(X_r^1 - X_v^2 + z)drdv
 \] 
 with $X^1, X^2$   two independent standard  Brownian motions on $\R^d$.
  
 If in addition, there exist some $t_0>0$ and some $\alpha\in(0,1/2)$ such that 
 \begin{align}\label{ADDc}
 \int_0^{t_0}\int_0^{t_0} \gamma_0(r-v) r^{-\alpha} v^{-\alpha} drdv < +\infty,
 \end{align}
 then as $R\to+\infty$, $\left\{ R^{-d/2}A_t(R), \, t\ge 0\right\}$ converges weakly to $ \{\mathcal{G}_t, t\ge 0\} $ in the space of continuous functions $C(\R_+)$.  
  \end{theorem}
  
  Notice that \eqref{ADDc} is satisfied when $\gamma_0=\delta_0$. In this case $\gamma_0$ is not a function but the result can be properly formulated.

One may ask what happens if  $\gamma_1(\R^d)$ is not finite, and this includes an important example,  the Riesz kernel $\gamma_1(z) = \|z\|^{-\beta}$ with  $\beta\in(0, 2\wedge d)$.    
 
  \begin{theorem}\label{Nchaotic}      Suppose $\gamma_0:\R\to\R_+\cup \{ \infty\}$ is locally integrable  and $\gamma_1(\R^d)= + \infty$.
  
{\rm  (1)} Assume that      $\mu_1$  admits a  density $\varphi_1$ that satisfies  
  \begin{align} \label{mDc2}
  \int_{\R^d} \frac{\varphi_1(\xi) + \varphi_1(\xi)^2 }{1+ \| \xi\|^2} d\xi < +\infty.
  \end{align} 
  Then,
  $R^{-d} \Var\big( \Pi_1A_t(R) \big) $ diverges to infinity as $R\to+\infty$
  and 
  \begin{align*}
& \lim_{R\to+\infty}  R^{-d} \sum_{p\geq 2}\Var\big( \Pi_pA_t(R) \big) =  \omega_d \int_{\R^d} \E\big( e^{\beta_{t,t}(z)  }  - \beta_{t,t}(z) -1   \big)dz \in (0,\infty).
  \end{align*}
  As a consequence, we have 
  \begin{align*}
 \frac{ A_t(R)}{\sqrt{ \Var\big( A_t(R) \big)  }} \xrightarrow[R\to+\infty]{\rm law} N(0,1). 
  \end{align*}
  
  {\rm  (2) } When $\gamma_1(z) = \| z \|^{-\beta}$ for some $\beta\in(0, 2\wedge d)$, we have 
  \begin{align}
   \frac{ A_t(R)}{R^{d-\frac{\beta}{2}} } \xrightarrow[R\to+\infty]{\rm law} N(0,\kappa_\beta), \label{Nchaotic2}
  \end{align}
  with 
  \begin{align*}
  \kappa_\beta:= \left( \int_0^t\int_0^t dr dv \gamma_0(r-v) \right)\int_{B_1^2 }dx dy  \| x -y \|^{-\beta}.
  \end{align*}
\end{theorem}
 
 Note that the Riesz kernel in part (2) satisfies the modified version of Dalang's condition \eqref{mDc2} if and only if  $d/2 < \beta < 2\wedge d$, which is equivalent to
 \begin{align}\label{betaequi}
  \begin{cases}
   \beta\in (1/2, 1) \quad \text{for $d=1$}\\
      \beta\in (1, 2) \qquad \text{for $d=2$}\\
   \beta\in (3/2, 2) \quad \text{for $d=3$.}
  \end{cases}
 \end{align}
 In particular, in  dimension one, $  \beta\in (1/2, 1)$ is equivalent to the fractional noise with Hurst parameter $H\in(1/2, 3/4)$.

    \begin{remark} 
 
 Unlike previous  studies, we consider a  noise that is colored in time, and our results complement,  in particular, those in \cite{HNV18,HNVZ19}. In  \cite{HNV18} where the noise is white in space and time,  the authors were able to obtain the chaotic central limit theorem for the linear equation (parabolic Anderson model), proving also a rate of convergence in the total variation distance.
  The quantitative CLT in the case  $\gamma_0 = \delta_0$ and $\gamma_1(z) =\| z\|^{-\beta} $, was  
obtained in  \cite{HNVZ19} for the nonlinear equation,  and the authors of \cite{HNVZ19} also proved that for the linear equation,  the first chaos is dominant so the central limit theorem is not chaotic.  
\end{remark}
 
 We point out that in both parts of Theorem \ref{Nchaotic} the first chaos    dominates, that is, the central limit theorem is not chaotic. Moreover, we are able to provide the following functional version of Theorem \ref{Nchaotic}.
 
 \begin{theorem}\label{NCF}    Suppose $\gamma_0:\R\to\R_+\cup \{ \infty\}$ is locally integrable  and $\gamma_1(\R^d)= + \infty$.
 
 {\rm (1)} Let the assumptions in part {\rm (1)} of Theorem \ref{Nchaotic} hold and we assume that  the condition \eqref{ADDc} is satisfied. We put
 \[
 \wh{A}_t(R) := \sum_{p\geq 2} \Pi_p\big( A_t(R) \big),
 \]
 then as $R\to\infty$, the process $\big( R^{-d/2}\wh{A}_t(R): t\in\R_+\big)$ converges in law to a centered continuous Gaussian process $\wh{\mathcal{G}}$ with covariance  given by
 \begin{align*}
 \E \big[ \wh{\mathcal{G}}_s \wh{\mathcal{G}}_t \big] := \omega_d \int_{\R^d} \E\left[ e^{\beta_{s,t}(z)} - \beta_{s,t}(z) -1 \right] dz.
\end{align*}

{\rm (2)}   If   condition \eqref{ADDc} is satisfied for some $\alpha\in(0,1/2)$ and $\gamma_1(z) = \|z\|^{-\beta}$ for some $\beta\in(0, 2\wedge d)$, then  the process $\big( R^{-d+ \frac{\beta}{2}}A_t(R): t\in\R_+\big)$ converges in law to a centered continuous Gaussian process $\wt{\mathcal{G}}$, as $R\to\infty$. Here the covariance structure of $\wt{\mathcal{G}}$ is given by 
\[
 \E \big[ \wt{\mathcal{G}}_s \wt{\mathcal{G}}_t \big] = \left( \int_0^t\int_0^s dr dv \gamma_0(r-v) \right)\int_{B_1^2 }dx dy  \| x -y \|^{-\beta}.
\]

 \end{theorem}

  We will organize the rest of our article into three sections.
  Section \ref{sec2} begins with a subsection on some preliminary knowledge, where we provide some important lemmas  for our later analysis. We devote Section \ref{22}     to the  investigation of  the central limit theorems on a fixed chaos by looking at assumptions on the covariance kernel and on the spectral measure separately. We derive the  corresponding  chaotic central limit theorems in Section \ref{23}.
  Section \ref{sec3} is devoted to the proof of Theorems \ref{thmSHE}, \ref{Nchaotic} and \ref{NCF}. For
  Theorem \ref{thmSHE}. we show the convergence of the finite-dimensional distributions and the tightness. Theorem \ref{Nchaotic} and Theorem \ref{NCF} are proved  as a by-product of the estimations in the proof of Theorem \ref{thmSHE}. 
Finally,  Section \ref{tech} provides the proofs of some technical results stated in previous sections.

 \bigskip

 \section{Infinite version of  the Breuer-Major theorem}\label{sec2}

\subsection{Preliminaries} In this section, we  introduce some notation for later reference and we provide several lemmas needed for our proofs.  

\medskip

Recall from our introduction that $\{W(h), h\in\H\}$ is an isonormal Gaussian process such that for any $\phi, \psi\in\H$,
\[
\E\big[ W(\phi) W(\psi) \big] = \langle \phi, \psi \rangle_\H = \int_{\R^{2d}} \phi(x) \psi(y) \gamma(x-y) dxdy = \int_{\R^d} \F\phi(\xi) \F\psi(-\xi) \mu(d\xi),
\]
where $\gamma$ is the covariance kernel and $\mu$ is the spectral measure whose Fourier transform is $\gamma$, understood in the generalized sense.  Let $\H_\mu$ be the   Hilbert space of functions $g:\R^d\to\C$
 such that $g(-x) = \overline{g}(x)$ for $\mu$-almost every $x\in\R^d$ and 
 \[
 \int_{\R^d}  \vert g(\xi) \vert^2\, \mu(d\xi) <+\infty \,.
 \]
Here $\overline{z}$ is the complex conjugate of $z\in\C$. It is clear that the Fourier transform stands as a linear isometry from $\H$ to $\H_\mu$.

For any integer $p\ge 2$, let $\H^{\otimes p}$ (resp. $\H^{\odot p}$) the $p$th tensor product (resp. symmetric tensor product) of $\H$. Note that 
for any integer $p\ge 2$,  the $p$th multiple stochastic integral $I_p^W$   is a linear and continuous operator from $\H^{\otimes p}$ into $L^2(\Omega)$. We can define spaces like  $\H_\mu^{\otimes p}$ and  $\H_{\mu}^{\odot p}$ in the obvious manner. 

\medskip

To simplify the display, we introduce some compact notation below. 

 \medskip
 
\noindent{\bf Notation A:} For any $R>0$, $B_R(x)$ stands for the $d$-dimensional Euclidean (closed) ball centered at $x$ with radius $R$ and we have used $B_R$ for $B_R(0)$. We write $\text{vol}(A)$ for the volume of $A\subset \R^d$ and $\omega_d=\text{vol}(B_1)$. We use $\| \cdot \|$ to denote the Euclidean norm in any dimension.

For $r\in\N$ and   $\pmb{x_r}= (x_1, \ldots, x_r)$,  we write $-\pmb{x_r}$ for $(-x_1, \ldots, -x_r)$,  $d\pmb{x_r} = dx_1 \cdots dx_r$ and  $\mu(d\pmb{x_r}) = \mu(dx_1)\cdots \mu(dx_r)$; we also write $\tau(\pmb{x_r}) = x_1 + \cdots +x_r$.  For integers $1\leq r < p $, we write $(\xi_1, \ldots, \xi_p) =\pmb{\xi_p} = ( \pmb{\xi_r} ,\pmb{\eta_{p-r}}    )$ with $ \pmb{\xi_r} = (\xi_1, \ldots, \xi_r)$ and $\pmb{\eta_{p-r}}    =(\xi_{r+1}, \ldots, \xi_p)$. With the above compact notation, we define the contraction operators $\otimes_r$ as follows.  For $f\in\H^{\otimes p}$ and $g\in\H^{\otimes q}$ ($p,q\in\N$), their $r$-contraction, with $0\leq r\leq p\wedge q$, belongs to $\H^{\otimes p+q-2r}$ and is defined by
  \begin{align*}
  (f\otimes _r g)\big( \pmb{\xi_{p-r}}, \pmb{\eta_{q-r}}   \big):= \int_{\R^{2rd}} f\big(  \pmb{\xi_{p-r}},    \pmb{a_{r}}\big) g\big(    \pmb{\eta_{q-r}},  \pmb{\wt{a}_{r}}\big)\prod_{j=1}^r \gamma(a_j-\wt{a}_j) d \pmb{a_{r}} d \pmb{\wt{a}_{r}}
  \end{align*}
 for $\pmb{\xi_{p-r}} \in\R^{pd-rd}$ and $\pmb{\eta_{q-r}} \in\R^{qd-rd}$. In particular, $f\otimes_0 g = f\otimes g$ is the usual tensor product and if $p=q$, $f\otimes_p g = \langle f, g \rangle_{\H^{\otimes p}}$; see also \cite[Appendix B]{bluebook}.    Let us introduce some useful lemmas now.

 \bigskip
 
 For $p$ positive, we denote by $J_p$ the Bessel function of first kind with order $p$: 
 \begin{align}\label{def-Bessel}
\qquad J_{p}( x ) = \frac{(x/2)^p}{\sqrt{\pi} \Gamma(p + \frac{1}{2} )} \int_0^{\pi} (\sin \theta)^{2p} \cos\big( x \cos\theta\big) \, d\theta,  \quad x\in\R\, ;
\end{align}
see   \cite[(5.10.4)]{Lebedev72}.   Let us also record here 
 \begin{align}\label{wdexp}
 \omega_d = \text{vol}(B_1) = \frac{ \pi^{d/2}}{   \Gamma\big(1 + \frac{d}{2}  \big)},
 \end{align} with $\Gamma$ the Euler's Gamma function.

 \begin{lemma}\label{Bessel}  $(1)$ Given $\xi\in\R^d$ and $R>0$, we have 
\[
\int_{B_R} e^{-\i \xi \cdot u}  \, du = (2\pi R)^{d/2} \| \xi \| ^{-d/2} J_{d/2}\big(R \| \xi \| \big) \,,
\]
where $J_{d/2}$ is the Bessel function of the first kind with order $d/2$.

\medskip

\noindent$(2)$ Given a positive real number $p$, we have 
\begin{align}
J_p(x) &\sim \sqrt{ 2/ (\pi x) } \cos\Big(  x - \frac{(2p+1)\pi}{4} \Big) \quad\text{as $x\to+\infty$,} \label{largex} \\
  J_p(x) &\sim \frac{x^p}{2^p \Gamma(p+1)} \quad\text{as $x\to0$.} \label{x=0}
\end{align}
As a consequence, we have $\sup\{ \vert J_p(x)\vert \,:\, x\in\R_+\} < +\infty$ and   $\vert J_p(x)\vert \leq  C | x |^{-1/2}$ for any $x\in\R$, here $C$ is some absolute constant.

\medskip

\noindent$(3)$ Put $\ell_R(x) = \omega_d^{-1} \| x\| ^{-d} J_{d/2}(R\| x\| )^2$, then $\{\ell_R: R > 0\}$ is an approximation of the identity. 
\end{lemma}

\begin{proof} (1)  Let us suppose first that $R = 1$.  In this case, one sees that the Fourier transform of ${\bf 1}_{\{\| u \| \leq 1\}}$ is rotationally symmetric, so without losing any generality, we assume $\xi = (0, \ldots,0, \rho)$ with $\rho = \| \xi\| > 0$. Then for $d\geq 2$,
\begin{align}
&\int_{\R^d} e^{-\i \xi \cdot u} {\bf 1}_{\{ \| u \| \leq 1  \}} \, du \notag \\
 &= \int_{-1}^{1} e^{-\i \rho x_d} \int_{\R^{d-1}} {\bf 1}_{\{\|  \pmb{x_{d-1}} \|^2 \leq 1 - x_d^2\}} d\pmb{x_{d-1}} \, dx_d = \int_{-1}^{1} e^{-\i \rho x_d}  \omega_{d-1}  \big( 1 - x_d^2\big)^{\frac{d-1}{2}}  dx_d \notag \\
&= \omega_{d-1}\int_{-1}^{1}  \cos(\rho y)    \big( 1 - y^2\big)^{\frac{d-1}{2}}  dy =\omega_{d-1} \int_0^\pi  \cos(\rho \cos(\theta))  \sin(\theta)^d \, d\theta  \notag \\
&= (2\pi)^{d/2} \rho^{-d/2} J_{d/2}(\rho)   \notag,
\end{align}
 where the last equality follows from  the expressions \eqref{wdexp} and \eqref{def-Bessel}. That is, for $d\geq 2$,
\[
\int_{\R^d} e^{-\i \xi \cdot u} {\bf 1}_{\{ \| u \| \leq 1  \}} \, du =(2\pi)^{d/2} \| \xi\|^{-d/2} J_{d/2}(\| \xi\|).
\]
The above equality also holds true for $d=1$, as one can verify by a direct computation for both sides.  So   the result in part (1) is established for $R = 1$.  The general case follows from a change of variable. 

\medskip

(2)  The asymptotic behavior of Bessel functions can be found in \emph{e.g.}  page 134 of the book \cite{Lebedev72}.  The uniform boundedness of $J_p$ on $\R_+$ follows immediately from this asymptotic behavior.  By \eqref{largex}, we can find some $L > 0$ such that 
$
\vert J_p(x) \vert \leq 1/\sqrt{x}$ for any $x\geq L$, 
while  it follows from \eqref{def-Bessel} that $\vert J_p(x)\vert \leq  C_1 x^p$ for any $x\geq 0$. It suffices to pick $C = 1+ C_1 L^{p+\frac{1}{2}}$ such that $C_1 
\leq  C L^{-p-\frac 12}$  to  conclude  that $\vert J_p(x)\vert \leq  C | x |^{-1/2}$ for any $x\in\R$.

\medskip

(3) It suffices to show $1 = \| \ell_1 \| _{L^1(\R^d)}$. It follows from point (1) that  
\begin{align*}
 &   \int_{\R^d}    \| x\| ^{-d} J_{d/2}(\| x\| )^2dx \\
 &=\int_{\R^d} \left( \lim_{a\downarrow 0} \frac{1}{(2\pi)^{d/2} } \int_{\R^d} \exp\left(-\i \xi \cdot x - \frac{a}{4}\| x\|^2  \right) {\bf 1}_{\{ \| \xi \| \leq 1  \}} d\xi        \right)^2  dx \\
 &= \lim_{a\downarrow 0} \int_{\R^{2d}}   d\xi d\xi' {\bf 1}_{ \{ \xi, \xi'\in B_1 \}} \frac{1}{(2\pi)^d} \int_{\R^d}\exp\left(-\i  (\xi + \xi') \cdot x - \frac{a}{2}\| x\|^2  \right) dx \\
 &=  \lim_{a\downarrow 0} \int_{\R^{2d}}   d\xi d\xi' {\bf 1}_{ \{ \xi, \xi'\in B_1 \}}   \frac{\exp\big(-\| \xi + \xi' \|^2/(2a)   \big) }{(2\pi a)^{d/2}} \\
 &=  \lim_{a\downarrow 0} \int_{\R^d}  \text{vol}\big(B_1\cap B_1(\xi) \big)   \frac{e^{-\| \xi \|^2/(2a)} }{(2\pi a)^{d/2}}    =   \omega_d \,,
\end{align*}
where the interexchanges of integrals and limits are valid due to the dominated convergence theorem. Our proof of this lemma is finished.  \qedhere

\end{proof}

The following lemma has its discrete analogue in \cite[(7.2.7)]{bluebook} and for the sake of completeness, we provide a short proof; see also \cite[(3.3)]{NZ19}.

\begin{lemma}\label{need001}  If $\phi:\R^d\to\R$ belongs to $L^p(\R^d,dx)$ for some positive number $p$. Then for any $r\in(0, p)$, one has
\begin{align*}
\frac{1}{R^{d(1 - rp^{-1})} }\int_{B_R}       \vert \phi(x) \vert^r  dx \xrightarrow{R\to+\infty} 0 \,.
\end{align*}
\end{lemma}

\begin{proof}  Fix $\delta\in(0,1)$. We deduce from  H\"older's inequality that
 \begin{align*}
& \quad \frac{1}{R^{d(1 - rp^{-1})} }\int_{B_R}        \vert \phi(x) \vert^r  dx \\
&=\frac{1}{R^{d(1 - rp^{-1})} }\int_{B_{\delta R}}    \vert \phi(x) \vert^r  dx  + \frac{1}{R^{d(1 - rp^{-1})} } \int_{B_R\setminus B_{\delta R}}         \vert \phi(x) \vert^r  dx  \\
  &\leq C \delta^{d(1 - rp^{-1}) } \left(\int_{\R^d}    \vert \phi(x) \vert^p  dx \right)^{r/p} +C \Big( 1 - \delta^{d(1 - rp^{-1}) }  \Big)  \left(\int_{B_R\setminus B_{\delta R}}         \vert \phi(x) \vert^p  dx\right)^{r/p}.
 \end{align*}
Note that  for any fixed $\delta\in(0,1)$, the second term goes to zero, as $R\to+\infty$, while the first term can be made arbitrarily small by choosing sufficiently small $\delta$.  \end{proof}

At the end of this section, we record a consequence of Young's   inequality. 

\begin{lemma} Suppose $\varphi:\R^d\to\R$ belongs to $L^q(\R^d,dx)$ with $q=p/(p-1)$ for some integer $p\geq 2$. Then,
    \begin{align}\label{conv-bdd}
\big\| \varphi^{\ast p}  \big\| _\infty  \leq    \| \varphi \|_{L^q(\R^d)} ^p\,,
    \end{align}
    where the $p$-convolution   can be defined iteratively: $\varphi^{\ast 2} = \varphi\ast\varphi$, ..., $\varphi^{\ast p} =\varphi\ast \varphi^{\ast p-1}$.
\end{lemma}

\begin{proof} Young's convolution inequality states that 
$$ \| h_1\ast h_2 \| _{L^r(\R^d)} \leq \| h_1 \| _{L^p(\R^d)} \| h_2 \| _{L^q(\R^d)}$$
 for any $h_1\in L^p(\R^d)$ and   $h_2\in L^q(\R^d)$ with $p^{-1} + q^{-1} = 1 + r^{-1}$ and $1\leq p, q,r\leq \infty$. 
 As a consequence,  we  obtain the following inequalities:
 \begin{align*}
 \begin{cases}  
 ~  \| \varphi^{\ast p}   \| _\infty =  \| \varphi \ast \varphi^{\ast p-1}   \| _\infty \leq   \| \varphi \|_{L^q(\R^d)}   \| \varphi^{\ast p-1} \|_{L^{q_1}(\R^d)}\quad\text{with $q_1=p$, } \\
   ~  \| \varphi^{\ast p-1} \|_{L^{q_1}(\R^d)} = \| \varphi \ast \varphi^{\ast p-2}   \| _{L^{q_1}(\R^d)} \leq  \| \varphi \|_{L^q(\R^d)}  \| \varphi^{\ast p-2} \|_{L^{q_2}(\R^d)}\quad\text{with $q_2=p/2$, }\\
     ~  \| \varphi^{\ast p-2} \|_{L^{q_2}(\R^d)} = \| \varphi \ast \varphi^{\ast p-3}   \| _{L^{q_2}(\R^d)} \leq  \| \varphi \|_{L^q(\R^d)}  \| \varphi^{\ast p-3} \|_{L^{q_3}(\R^d)}\quad\text{with $q_3= p/3$, } \\
     \qquad  \ldots \\
      ~     \| \varphi^{\ast 2} \|_{L^{q_{p-2}}(\R^d)} = \| \varphi \ast \varphi   \| _{L^{q_{p-2}}(\R^d)} \leq  \| \varphi \|_{L^q(\R^d)}  \| \varphi \|_{L^{q_{p-1}}(\R^d)}\quad\text{with $q_{p-1}=\frac{p}{p-1}$. }
      \end{cases}
 \end{align*}
This completes the proof of  \eqref{conv-bdd}.
\end{proof}

\bigskip

Recall from our introduction that we consider the case where $F = \sum_{k\geq m} I_k^W(f_k)$ has Hermite rank $m\geq 1$ with $f_k\in\H^{\odot k}$ for each $k\geq m$.  We write  
\[
G_R:=\int_{B_R} U_x F \, dx = \sum_{k\geq m} I^W_k(g_{k,R})=:  \sum_{k\geq m} G_{k,R}\quad\text{with} \quad g_{k,R} = \int_{B_R} f_k^x \, dx \,.
\]
In what follows, we first investigate the central limit theorem on each chaos based on two sets of assumptions.  One involves the covariance kernel $\gamma$ and the other is based on the spectral measure $\mu$. This is the content of Section \ref{22}, and in Section \ref{23}, we consider  the case where  $F$ has  a general chaos expansion.   In each situation, the random variable may depend on infinitely many coordinates, which shall be distinguished from the classical Breuer-Major theorem. 

\subsection{Central limit theorems on a fixed chaos}\label{22} Fix an integer $p\geq 2$ and note that the random field $\{I_p^W(f_p^x), x\in \R^d\}$ is  centered, strictly stationary. We put
\[
\E[I_p^W(f^x_p) I_p^W(f^y_p)]=: \Phi_p(x-y).
\]
Then, if 
\begin{equation} \label{cond1}
\int_{\R^d} | \Phi_p(x) | dx < \infty,
\end{equation}
we have, with the notation $G_{p,R} = I^W_p(g_{p,R}) $,
\begin{equation} \label{k1}
\lim _{R\rightarrow + \infty}  \frac {{\rm Var } (G_{p,R}) }{R^d} = \omega_d\int_{\R^d}  \Phi_p(x)  dx.
\end{equation}
Indeed, 
\[
 {\rm Var } (G_{p,R})= \int_{B_R^2} \Phi_p(x-y) dxdy=  \int_{B_R} \text{vol}\big( B_R \cap B_R(-z)   \big)   \Phi_p(z) dz. 
 \]
Because $\text{vol}\big( B_R \cap B_R(-z)   \big) / \text{vol}(B_R)$ is bounded by one and convergent to one, as $R\to+ \infty$,
  \eqref{k1} follows from (\ref{cond1})  and the  dominated convergence theorem. This fact leads us to stick on the situation that the normalization $\sigma(R)$ in \eqref{Q0} is of order $R^{d/2}$, as $R\to+\infty$.
  Such an order is also consistent with the Breuer-Major theorem (see Theorem \ref{thm0}).

\subsubsection{\bf CLT under assumptions on the covariance kernel}

We write 
\[
\Phi_p(x) = p! \langle f^x_p,f_p\rangle_{\H^{\otimes p}}= p!   \int_{\R^{2pd} }   f_p(\pmb{\xi_p})f_p(\pmb{\eta_p})  \prod_{i=1}^p\gamma\big(\xi_i-\eta_i+x\big) 
d\pmb{\xi_p} \, d\pmb{\eta_p}\,.
\]
Therefore, a sufficient condition for \eqref{cond1} to hold is the following hypothesis:

\noindent
\medskip
 {\bf (H1)}  \quad       $ f_p\in\H^{\odot p}$ satisfies  ${\displaystyle \int_{\R^d}   \int_{\R^{2pd} }   \vert f_p(\pmb{\xi_p})f_p(\pmb{\eta_p}) \vert \prod_{i=1}^p \vert \gamma\vert\big(\xi_i-\eta_i+x\big)
d\pmb{\eta_p} \, d\pmb{\xi_p} dx  <\infty}$.  

\medskip
Define 
\begin{equation} \label{kappa2}
\kappa_p(\pmb{\xi_p}-\pmb{\eta_p})= \int_{\R^d}  \prod_{i=1}^p \gamma (\xi_i-\eta_i+z )   dz.
\end{equation}
Then, under  {\bf (H1)},
\[
 \int_{\R^d} \Phi_p(x)dx =p!  \int_{\R^{2pd}}     f_p(\pmb{\xi_p})f_p(\pmb{\eta_p})       \kappa_p( \pmb{\xi_p} - \pmb{\eta_p})  \, d\pmb{\xi_p} \, d\pmb{\eta_p} .
 \]
 Suppose that $\gamma \in L^p(\R^d)$ and $f_p\in  L^1(\R^{pd})$. Then,  hypothesis  {\bf (H1)}  is satisfied. In fact, using H\"older's inequality, we obtain
\[
 \int_{\R^d}   \int_{\R^{2pd} }   | f_p(\pmb{\xi_p})f_p(\pmb{\eta_p}) | \prod_{i=1}^p     \vert \gamma\vert (\xi_i-\eta_i+x )  
d\pmb{\xi_p} \, d\pmb{\eta_p} dx\leq \| \gamma \| _{L^p(\R^d)}^p  \| f_p  \| _{L^1(\R^{pd})}^2  < \infty\,.
\]

\begin{remark}\label{rem1} {\rm  (i)  In  the particular case where $p=1$,  the conditions $f_1\in L^1(\R^d)\cap\H$ and $\gamma\in L^1(\R^d)$ are necessary, since  hypothesis {\bf (H1)} becomes
       \begin{align*}
\int_{\R^{2d}} \big\vert f_1(t) f_1(s) \big\vert \int_{\R^d} \vert\gamma\vert(t-s+z) dz dt ds =  \| f_1 \|_{L^1(\R^d)}^2  \| \gamma\| _{L^1(\R^d)}  < \infty \,.
\end{align*}
Under these necessary conditions, it is clear that
\[
\int_{B_R} I^{W}_1(f_1^x) \, dx
\]
is a centered Gaussian random variable  with  
\begin{center}
${\displaystyle \Var\left( \int_{B_R} I^{W}_1(f_1^x) \, dx\right) \sim \omega_d R^d  \| f_1 \|_{L^1(\R^d)}^2 \int_{\R^d}\gamma(z)dz}$, as $R\to+\infty$.
\end{center}

 (ii) Here is  an example of  non-integrable covariance kernel:  $\gamma(x) = \|x\|^{-\beta}$, with $\beta \in (0, d)$. Now let us   search for  sufficient condition for   $\kappa_p$ to be well defined. Notice that
\[
  \int_{\R^d} \prod_{i=1}^p\gamma(a_i + z)\, dz =    \int_{\R^d} \prod_{i=1}^p \| a_i + z\|^{-\beta}\, dz 
\]
and for $a_1, \ldots, a_p$ mutually distinct,  the product $\prod_{i=1}^p \| a_i + z\|^{-\beta}$ is integrable near the singularities.
Indeed, choosing  $\varepsilon =\frac 12 \min\{ \vert a_i - a_k \vert\,: 1\leq i<k\leq p\}$, we can write for each $j=1, \dots, p$,
\begin{align*}
\int_{B_\varepsilon(a_j)} \prod_{i=1}^p \| a_i + z\|^{-\beta}\, dz  & \leq C  \int_{B_\varepsilon(a_j)}   \| a_j + z\|^{-\beta}\, dz = C    \int_{B_\varepsilon}   \| z\|^{-\beta}\, dz \\
&= C \int_0^\varepsilon r^{-\beta} r^{d-1} \, dr  < \infty.
\end{align*} 
Thus,  we only need to control the integral at infinity. Notice that  for $L>0$ large (that may depend on the  $a_i$'s), there exist two constants   $C_1, C_2$ such that
\[
 C_1  \int_{\| z\| \geq L} \| z\|^{-\beta p}  \, dz      \leq  \int_{\| z\| \geq L} \prod_{i=1}^p\| a_i + z\|^{-\beta}\, dz  \leq   C_2  \int_{\| z\| \geq L} \| z\|^{-\beta p}  \, dz .
 \]
 Then   the  finiteness of the integral at infinity is   equivalent to   $p> d/\beta$. In other words, the function $\kappa_p$, given in \eqref{kappa2}, makes sense only for  $p> d/\beta$. This forces us to consider  chaoses of order at least $\lfloor d/\beta \rfloor +1 =:m_0$.  
 Now for $p\geq m_0$,   the kernel $f_p\in\mathfrak{H}^{\odot p}$   satisfies
{\bf (H1)} if
\[
\int_{\R^{2pd}} \big\vert f_p(\pmb{x_p})f_p(\pmb{y_p})\big\vert \int_{\R^d} \prod_{i=1}^p \| x_i - y_i + z \|^{-\beta}\, dz d\pmb{x_p} d\pmb{y_p} < \infty.
\]
}
\end{remark}

\bigskip

The following result is a central limit theorem under some restrictions on $\gamma$.

\begin{theorem}    \label{thm1}  Fix  an integer $p\ge 2$ and  $f_p\in \H^{\odot p}$.  Assume the hypothesis {\bf (H1)}. 
Moreover,   suppose that one of the  following two  conditions hold true:
\begin{itemize}
\item[(i)] The kernel $f_p$ has the form\footnote{If $h_1,\ldots, h_p\in\H$, we denote by $\texttt{sym}\big(h_{1} \otimes \cdots \otimes h_{p}\big)$ the symmetrization of the tensor product $h_{1} \otimes \cdots \otimes h_{p}$:
\[
\texttt{sym}\big(h_{1} \otimes \cdots \otimes h_{p}\big) :=\frac{1}{p!} \sum_{\pi\in\mathfrak{S}_p} h_{\pi(1)} \otimes \cdots \otimes h_{\pi(p)}\,,
\]
where $\mathfrak{S}_p$ is the permutation group on the first $p$ positive  integers.} $f_p= \texttt{sym}\big(h_{1} \otimes \cdots \otimes h_{p}\big)$, where the $h_j\in\H$ satisfy
 \begin{align}
\sum_{i,j=1}^p \int_{\R^d} \left | \int_{\R^{2d}} h_{i} (s) h_{j} (t) \gamma(s-t+z)  dsdt  \right|^p dz <\infty \,. \label{cond:tp}
 \end{align}

 \item[(ii)] $\gamma \in L^p(\R^d)$ and $f_p \in L^1(\R^{pd})$. $($Note that  {\rm (ii)} implies  {\bf (H1)}.$)$
 \end{itemize}
Then  
\[
\frac{G_{p,R}} {R^{d/2}}  \xrightarrow[R\to+\infty]{\rm law}  N(0,  \sigma^2_p),
\]
where 
\[
\sigma^2_p= p!   \omega_d \int_{\R^{2pd} }    f_p(\pmb{s_p})f_p(\pmb{t_p})      \kappa_p( \pmb{t_p} - \pmb{s_p})  \, d\pmb{t_p} \, d\pmb{s_p} .
\]
\end{theorem}

\begin{proof}
In view of the Fourth Moment Theorem of Nualart and Peccati \cite{FMT}, 
to prove this central convergence  it suffices to establish   
$$
\lim_{R\to+\infty}\frac{1}{R^{2d}} \big\| g_{p,R}\otimes_r g_{p,R} \big\| _{\mathfrak{H}^{\otimes  (2p-2r)  }}^2 = 0 
$$
 for $r=1, \dots, p-1$. 
By definition, we can write
\begin{align*}
 \big(g_{p,R}  \otimes_r g_{p,R}\big)( \pmb{s_{p-r}}, \pmb{t_{p-r}} )  = \int_{\R^{2rd}}  g_{p,R}(  \pmb{s_{p-r}} ,  \pmb{a_r})   g_{p,R}(\pmb{t_{p-r}},  \pmb{b_r} )   \prod_{i=1}^r  \gamma( a_i -b_i)   \, d\pmb{a_r} \, d\pmb{b_r} \,.
\end{align*}
As a consequence, 
 \begin{align}
&\quad \| g_{p,R} \otimes_r g_{p,R}\big\| ^2_{\H^{\otimes (2p-2r)} }  \notag\\
& =   \int_{\R^{4pd}}d\pmb{a_r} \, d\pmb{b_r}  d\pmb{\wt{a}_r}  \, d\pmb{\wt{b}_r}  \, d\pmb{t_{p-r}} \, d\pmb{s_{p-r}} \, d \pmb{\wt{t}_{p-r}} \, d\pmb{\wt{s}_{p-r}}   g_{p,R}( \pmb{s_{p-r}}, \pmb{a_r})   g_{p,R}(  \pmb{t_{p-r}}, \pmb{b_r} )   \notag  \\
&         \times g_{p,R}( \pmb{\wt{s}_{p-r}}, \pmb{\wt{a}_r} )  g_{p,R}( \pmb{\wt{t}_{p-r}}, \pmb{\wt{b}_r} )       \left(\prod_{i=1}^r  \gamma( a_i -b_i)      \gamma(   \wt{a}_i -\wt{b}_i ) \right) \left( \prod_{j=1}^{p-r}   \gamma(   t_j -\wt{t}_j)   \gamma(     \wt{s}_j -s_j ) \right) \,  \notag  \\ 
& =    \int_{B_R^4} d\pmb{x_4}  \int_{\R^{4dp}} d\pmb{a_r} \, d\pmb{b_r}  d\pmb{\wt{a}_r}  \, d\pmb{\wt{b}_r}  \, d\pmb{t_{p-r}} \, d\pmb{s_{p-r}} \, d \pmb{\wt{t}_{p-r}} \, d\pmb{\wt{s}_{p-r}}       f^{x_1}_p( \pmb{s_{p-r}}, \pmb{a_r})   f^{x_2}_p(  \pmb{t_{p-r}}, \pmb{b_r} )  \notag      \\
&     \times    f^{x_3}_p( \pmb{\wt{s}_{p-r}}, \pmb{\wt{a}_r} )      f^{x_4}_p( \pmb{\wt{t}_{p-r}}, \pmb{\wt{b}_r} )       \left(\prod_{i=1}^r  \gamma( a_i -b_i)      \gamma(   \wt{a}_i -\wt{b}_i ) \right)  \prod_{j=1}^{p-r}   \gamma(   t_j -\wt{t}_j )   \gamma(     \wt{s}_j -s_j )      \,.   \label{innerint}  
\end{align}
Shifting the variables from the kernels to the covariance,  we write
 \begin{align*}
&\quad  \| g_{p,R}\otimes_r g_{p,R}\big\| ^2_{\H^{\otimes (2p-2r)} }  \notag\\
 & =    \int_{B_R^4}d\pmb{x_4}  \int_{\R^{4dp}} d\pmb{a_r} \, d\pmb{b_r}  d\pmb{\wt{a}_r}  \, d\pmb{\wt{b}_r}  \, d\pmb{t_{p-r}} \, d\pmb{s_{p-r}} \, d \pmb{\wt{t}_{p-r}} \, d\pmb{\wt{s}_{p-r}}   f_p( \pmb{s_{p-r}}, \pmb{a_r})   f_p(  \pmb{t_{p-r}}, \pmb{b_r} )      \\
 &\qquad\qquad  \times f_p( \pmb{\wt{s}_{p-r}}, \pmb{\wt{a}_r} )   f_p( \pmb{\wt{t}_{p-r}}, \pmb{\wt{b}_r} )   \left(\prod_{i=1}^r  \gamma( a_i -b_i + x_1-x_2)      \gamma(   \wt{a}_i -\wt{b}_i  +x_3 -x_4) \right)        \\
&\qquad \qquad\qquad   \times  \left( \prod_{j=1}^{p-r}   \gamma(   t_j -\wt{t}_j  +x_2-x_4)   \gamma(     \wt{s}_j -s_j +x_3-x_1 ) \right)  \,  \, .    
\end{align*}
Making the change of variables $x_1-x_2=z_1$, $x_3-x_4 =z_2 $ and $x_2-x_4 =z_3$ (so $x_3-x_1=z_2-z_3-z_1$), we   obtain
\begin{align}
 &\quad  R^{-2d}  \| g_{p,R}\otimes_r g_{p,R} \| ^2_{\H^{\otimes (2p-2r)} } \notag  \\
& \leq  C  R^{-d}     \int_{B_{2R}^3} d\pmb{z_3} \Bigg|  \int_{\R^{4dp}}  d\pmb{a_r} \, d\pmb{b_r}  d\pmb{\wt{a}_r}  \, d\pmb{\wt{b}_r}  \, d\pmb{t_{p-r}} \, d\pmb{s_{p-r}} \, d \pmb{\wt{t}_{p-r}} \, d\pmb{\wt{s}_{p-r}} f_p( \pmb{s_{p-r}}, \pmb{a_r})   \notag \\
  &\qquad\quad \times    f_p(  \pmb{t_{p-r}}, \pmb{b_r} ) f_p( \pmb{\wt{s}_{p-r}}, \pmb{\wt{a}_r} )   f_p( \pmb{\wt{t}_{p-r}}, \pmb{\wt{b}_r} )       \left(\prod_{i=1}^r  \gamma( a_i -b_i + z_1)      \gamma(   \wt{a}_i -\wt{b}_i  + z_2) \right)    \notag    \\
&\qquad\qquad\qquad \qquad    \times  \left( \prod_{j=1}^{p-r}   \gamma(   t_j -\wt{t}_j  + z_3)   \gamma(     \wt{s}_j -s_j + z_2  - z_1- z_3) \right)  \,   \Bigg | . \label{RHS27}
  \end{align}
  The rest of our proof will be split into  two cases.

  \medskip
  
\noindent
{\it Proof under \rm(i)}.  
 Using the tensor-product structure of the kernels, we can further bound  \eqref{RHS27} by
  \begin{align*}
    C  R^{-d}     \int_{B_{2R}^3}    d\pmb{z_3}      \phi(z_1)^r     \phi(z_2)^r        \phi(z_3)^{p-r}    \phi(z_2-z_1 -z_3)^{p-r}   
  \,, 
  \end{align*}
   with
  \[
  \phi(z):=\sum_{i,j=1}^p \left| \int_{\R^{2d}}  h_{i}(a) h_{j} (b) 
  \gamma( a -b + z)    da db\right|\,.
  \]
In view of \eqref{cond:tp},    the function $\phi$ belong to $L^p(\R^d)$. It follows immediately from H\"older's inequality that
 \begin{align*}
& R^{-2d}  \| g_{p,R}\otimes_r g_{p,R}\big\| ^2_{\H^{\otimes (2p-2r)} } 
  \leq C  \left( \int_{\R^d}  \phi(z_1)^p   dz_1\right)  R^{-d}     \int_{B_{2R}^2}    dz_2dz_3        \phi(z_2)^r        \phi(z_3)^{p-r} \\
&\qquad\qquad =   C  \left( \int_{\R^d}  \phi(z_1)^p   dz_1\right)  R^{-d}   \left( \int_{B_{2R} }       \phi(z_2)^r  dz_2 \right)   \left( \int_{B_{2R}}   \phi(z_3)^{p-r} dz_3\right)\,.
  \end{align*}
Then, we can conclude our proof under the condition (i) by using Lemma \ref{need001}.  \hfill $\square$

\bigskip

     \noindent
{\it Proof under \rm (ii)}.    Note first that due to H\"older's inequality,
\[
\int_{B_{2R}}    \left(\prod_{i=1}^r   \vert \gamma \vert ( a_i -b_i + z_1)  \right)\left(\prod_{j=1}^{p-r} \vert  \gamma\vert (     \wt{s}_j -s_j + z_2  - z_1- z_3) \right)  dz_1 \leq \int_{\R^d}   \vert\gamma(z)\vert ^p\, dz \,,
\]
which implies that \eqref{RHS27} can be further bounded by 
\begin{align*}
&\quad C \| \gamma\|^p_{L^p(\R^d)}  \| f_p \| _{L^1(\R^{pd})}    R^{-d}     \int_{B_{2R}^2\times \R^{3dp}} dz_2dz_3   d\pmb{b_r}  d\pmb{\wt{a}_r}  \, d\pmb{\wt{b}_r}  \, d\pmb{t_{p-r}} \,  \, d \pmb{\wt{t}_{p-r}} \, d\pmb{\wt{s}_{p-r}}  \notag \\
  &\times   \big\vert  f_p(  \pmb{t_{p-r}}, \pmb{b_r} ) f_p( \pmb{\wt{s}_{p-r}}, \pmb{\wt{a}_r} )   f_p( \pmb{\wt{t}_{p-r}}, \pmb{\wt{b}_r} )  \big\vert     \left(\prod_{i=1}^r     \vert  \gamma\vert (   \wt{a}_i -\wt{b}_i  + z_2) \right)     \left( \prod_{j=1}^{p-r}   \vert \gamma \vert(   t_j -\wt{t}_j  + z_3)   \right)  \\
  &\leq  C  \int_{\R^{3dp}}   d\pmb{b_r}  d\pmb{\wt{a}_r}  \, d\pmb{\wt{b}_r}  \, d\pmb{t_{p-r}} \,  \, d \pmb{\wt{t}_{p-r}} \, d\pmb{\wt{s}_{p-r}}  \big\vert  f_p(  \pmb{t_{p-r}}, \pmb{b_r} ) f_p( \pmb{\wt{s}_{p-r}}, \pmb{\wt{a}_r} )   f_p( \pmb{\wt{t}_{p-r}}, \pmb{\wt{b}_r} )  \big\vert \times  \mathbf{L}_R  \,,
\end{align*}
where $  \mathbf{L}_R=\mathbf{L}_R\big( \pmb{\wt{a}_r}, \pmb{\wt{b}_r},  \pmb{\wt{t}_{p-r}},\pmb{t_{p-r} }\big)$ is given by
\begin{align*}
 \mathbf{L}_R= R^{-d} \left( \int_{B_{2R}}  \prod_{i=1}^r      \vert \gamma\vert (   \wt{a}_i -\wt{b}_i  + z_2) dz_2 \right)     \left(  \int_{B_{2R}} \prod_{j=1}^{p-r}    \vert\gamma\vert (   t_j -\wt{t}_j  + z_3) dz_3 \right). 
\end{align*}
Note that by H\"older's inequality and Lemma \ref{need001},
\begin{align*}
& \mathbf{L}_R  \leq \left( \prod_{i=1}^r \frac{1}{R^{d(1- rp^{-1}  )}}  \int_{B_{2R}}       \vert \gamma\vert ^r(   \wt{a}_i -\wt{b}_i  + z_2) dz_2 \right)^{1/r} \\
&\qquad\qquad\times \left( \prod_{j=1}^{p-r} \frac{1}{R^{d(1- (p-r)p^{-1}  )}}  \int_{B_{2R}}        \vert  \gamma\vert ^{p-r}(   t_j -\wt{t}_j  + z_3) dz_3    \right)^{1/(p-r)} \xrightarrow{R\to+\infty} 0\,,
\end{align*}
and that
 \[
 \mathbf{L}_R  \leq C R^{-d} \| \gamma \| _{L^p(\R^d)}^{r} R^{d\frac{p-r}{p}} \| \gamma \| _{L^p(\R^d)}^{p-r} R^{d\frac{r}{p}} = C \| \gamma \| _{L^p(\R^d)}^{p}  <+\infty\,.
 \] Thus, it follows from the dominated convergence theorem that,  as $R\to\infty$,
\[
 R^{-2d}  \| g_{p,R}\otimes_r g_{p,R} \| ^2_{\H^{\otimes (2p-2r)} }\to 0 
 \]
 for all $r\in\{1, \ldots, p-1\}$.
 This completes the proof.
     \end{proof}

 \subsubsection{\bf CLT under  assumptions on the spectral measure}

Let us first study the asymptotic variance  using the Fourier transform. Throughout this section, we are going to assume that   $\mu(d\xi) = \varphi(\xi)d\xi$, that is, the spectral measure is absolutely continuous with respect to the Lebesgue measure on $\R^d$. Note that $\varphi(\xi) = \varphi(-\xi)$. 
  
  We first write, 
\begin{align*}
\Phi_p(x-y) &= p! \langle f_x,f_y\rangle_{\H^{\otimes p}}=   p!     \int_{\R^{pd}}  (\F f^x_p)(\pmb {\xi_p}) (\F f^y_p)(- \pmb{ \xi_p}) ~\mu(d\pmb{\xi_p})  \\
    & = p!      \int_{\R^{pd}}  \exp\Big(-\i (x-y) \cdot \tau(\pmb{\xi_p}) \Big)  \vert \F f_p\vert^2   (\pmb{\xi_p})   ~\mu(d\pmb{\xi_p})\,,
     \end{align*}   
     where $\tau(\pmb{\xi_p}):= \xi_1 + \cdots +\xi_p$. As a consequence of Lemma \ref{Bessel}, we obtain
    \begin{align}
 {\rm Var} (G_{p,R})   & = p!   \int_{B^2_R}     
  \int_{\R^{pd}}  \exp\Big(-\i (x-y) \cdot \tau(\pmb{\xi_p}) \Big)      \vert \F f_p\vert^2(\pmb{\xi_p})    ~\mu(d\pmb{\xi_p})~dxdy \notag  \\
  &=p!  (2\pi R)^d      \int_{\R^{pd}}  \|   \tau(\pmb{\xi_p}) \| ^{-d} J_{d/2}\big( R\|  \tau(\pmb{\xi_p}) \| \big)^2   \vert \F f_p\vert^2 (\pmb{\xi_p})    ~\mu(d\pmb{\xi_p}) \,.  \label{need002}
    \end{align}
 Now making the change of variables 
 $\tau(\pmb{\xi_p}) = x$  yields
 \[
 {\rm Var} (G_{p,R}) R^{-d} =
  p! (2\pi )^{d} 
    \int_{\R^{d}}  \|   x \| ^{-d} J_{d/2}\big( R\|   x  \| \big)^2     \Psi_p(x) dx,
    \]
    where
    \begin{equation} \label{Psi}
    \Psi_p(x):=     \int_{\R^{pd-d}}  \vert \F f_p\vert^2\big( \pmb{\xi_{p-1}}, x -  \tau (\pmb{\xi_{p-1}})\big)  
 \varphi\big(x -  \tau (\pmb{\xi_{p-1}}) \big)     \prod_{i=1} ^{p-1} \varphi(\xi_i) 
    d\pmb{\xi_{p-1}}.
    \end{equation}
We remark that $\Psi_p$ is defined almost everywhere on $\R^d$ and recall that 
$$\big\{  \ell_R(x):= \omega_d^{-1} \|   x \| ^{-d} J_{d/2} ( R\|   x \|  )^2 \big\}_{R > 0}  $$ is an {\it approximation of the identity}.  
 Therefore, it is natural to introduce the following hypothesis:
 
 \medskip
 \noindent
 {\bf (H2)} \quad 
      $\Psi_p$, defined in  (\ref{Psi}),  is uniformly bounded on $\R^d$  and continuous at zero.

\medskip

  Under   {\bf (H2)},  we have
  \[
  \lim _{R\to +\infty}  \frac {{\rm Var } (G_{p,R}) }{R^d}= p! (2\pi)^{d} \omega_d \Psi_p(0)\,,
  \]
  where 
   \begin{align}\label{PSI0}
    \Psi_p(0)=     \int_{\R^{(p-1)d}}   \vert \F f_p\vert^2 \big(\pmb{\xi_{p-1}},  -  \tau(\pmb{\xi_{p-1}}) \big) 
 \varphi\big(  \tau (\pmb{\xi_{p-1}})   \big)   \prod_{i=1} ^{p-1} \varphi(\xi_i) d\pmb{\xi_{p-1}}.
    \end{align}
   Note that for the particular case $p=1$, $\Psi_1(x) = \vert \F f_1\vert^2(x) \varphi(x)$; if $f_1\in L^1(\R^d)$ and $\varphi$ is uniformly bounded with continuity at zero, then the function $\Psi_1$ is uniformly bounded and continuous at zero.

 \begin{remark}   (1)
 Heuristically, we can rewrite  $\Psi_p(0)$ as follows:
  \[
    \Psi_p(0)=     \int_{ \{ \tau(\pmb{\xi_p}) =0 \} } \vert \F f_p\vert^2 (\pmb{\xi_p})  
     \prod_{i=1} ^{p} \varphi(\xi_i) 
     \nu(d \pmb{\xi_p}),
    \]
    where $\nu$ is the surface measure on the hyperplane  $\{ \tau(\pmb{\xi_p}) =0 \} $. This is an informal expression, because the trace of 
$\F f_p$ on the hyperplane $\{ \tau(\pmb{\xi_p}) =0 \} $ is not properly defined for an arbitrary kernel $f_p$.

\medskip

(2)   Notice that  the quantity $ \frac {{\rm Var } (G_{p,R}) }{(2\pi R)^dp! \omega_d}$ is equal to
\[
 \int_{\R^{pd-d}}    \left(\int_{\R^d}dx \ell_R(x)   \varphi\big(x-\tau(\pmb{\xi_{p-1}}) \big)   \vert \F f_p\vert^2\big( \pmb{\xi_{p-1}},  x -\tau(\pmb{\xi_{p-1}})   \big)   \right)   \prod_{i=1} ^{p-1} \varphi(\xi_i) d\pmb{\xi_{p-1}}.
\]
It  is clear that $  \vert \F f_p\vert^2(\pmb{\xi_{p-1}},  x- \tau(\pmb{\xi_{p-1}})) $ is well-defined  almost everywhere with respect to $ \varphi\big(x-\tau(\pmb{\xi_{p-1}}) \big)dx$, and $ \varphi\big(x-\tau(\pmb{\xi_{p-1}}) \big)   \vert \F f_p\vert^2 ( \pmb{\xi_{p-1}},  x -\tau(\pmb{\xi_{p-1}})    )   $ is integrable with respect to the probability measure $\ell_R(x)dx$. We can also read from \eqref{PSI0} that the function  
    $\pmb{\xi_{p-1}}\mapsto \vert \F f_p\vert^2\big( \pmb{\xi_{p-1}},  -\tau(\pmb{\xi_{p-1}})   \big)  $   is integrable with respect to the measure     $ \varphi\big(\tau(\pmb{\xi_{p-1}}) \big)    \prod_{i=1} ^{p-1} \varphi(\xi_i) 
    d \pmb{\xi_{p-1}}$.

 \end{remark} 

\bigskip
To obtain the Gaussian fluctuation  of $G_{p,R}$, one shall first establish the order of the  variance  and then compute the contractions.  Our  hypothesis {\bf (H2)} gives the exact asymptotic behavior of ${\rm Var}(G_{p,R})$.
 In fact, it is enough to impose a weaker condition, known as the    Maruyama's condition concerning the variance; see \cite{Maruyama76}.   
 
  \begin{proposition}[Maruyama's condition]    \label{Maru}    Put 
  $$\wh{\Psi}_p(h) := \int_{\{  \| \tau(\pmb{\xi_p}  ) \| \leq h \}} \vert \F f_p\vert^2(\pmb{\xi_p}) \mu(d\pmb{\xi_p}) \,.$$  If
   \begin{align}\label{M-cond}
 0   < \liminf_{h\downarrow 0}  h^{-d} ~ \wh{\Psi}_p(h)     \leq  \limsup_{h\downarrow 0}  h^{-d} ~\wh{\Psi}_p(h) < \infty,
 \end{align}
 then we have, with $\sigma_{p,R}^2 =  { \rm Var}(G_{p,R})$ 
\[
  0 < \liminf_{R\to+\infty } \sigma_{p,R}^2 R^{-d}    \leq  \limsup_{R\to+\infty }  \sigma_{p,R}^2 R^{-d}  < \infty.
\]
  \end{proposition}  
     We will provide a proof of Proposition \ref{Maru} in Section \ref{tech}, see also \cite[Corollary 2.2]{CS89}.

\medskip

  The following lemma provides  sufficient conditions for {\bf (H2)} to hold.  One of the conditions is  $\varphi \in L^q(\R^d)$, which is the   condition imposed on the spectral density in the version of the classical Breuer-Major theorem proved in \cite[Theorem 2.10]{BBL}. 
  
  \begin{lemma} \label{lem2.1} 
  Suppose that  $f_p \in L^1(\R^{pd})\cap \H^{\odot p}$ and $\varphi \in L^q(\R^d)$, with $q=p/(p-1)$.  Then $\Psi_p$ is bounded  and continuous on $\R^d$, in particular  hypothesis   {\bf (H2)} is true.  
  \end{lemma}
 The proof of Lemma \ref{lem2.1} is given in Section \ref{tech}.

  \begin{remark}  It is worth comparing the sufficient conditions for the hypotheses {\bf (H1)} and {\bf (H2)} here:
  \begin{align*}
  \big\{\gamma\in L^p(\R^d)~  \text{and} ~ f_p\in L^1(\R^{pd})\big\} &\Rightarrow  {\bf (H1)}   \\
   \big\{\varphi\in L^q(\R^d) ~  \text{and} ~    f_p\in L^1(\R^{pd})\big\} &\Rightarrow  {\bf (H2)}.
  \end{align*}
This is natural in view of the Hausdorff-Young's inequality. Indeed, $q = p/(p-1)\in(1,2]$, so $\gamma=\F\varphi$ belongs to $L^p(\R^d)$, provided $\varphi\in L^q(\R^d)$. Note that both hypotheses imply that   the fluctuation of $G_{p,R}$ is of order $R^{d/2}$; moreover, as we will see shortly, both hypotheses ($\gamma\in L^p(\R^d)$ and $\varphi\in L^q(\R^d)$) imply that   the fluctuation of $G_{p,R}$ is Gaussian, as $R$ tends to infinity.

  \end{remark}

 Let us   introduce the following hypothesis, which can be seen as the contraction-analogue of {\bf (H2)}.

 \medskip
 \noindent
{\bf (H3)}  \quad   For $1\leq r\leq p-1 $ and any $\delta > 0$, 
   $\Psi_p^{(r,\delta)}$   is uniformly bounded on $\R^d$ and continuous  at zero,  
   where  
\begin{align}
&\quad \Psi_p^{(r,\delta)}(x,y) \label{k6} \\
&= \int_{\R^{2pd-2d}} d\pmb{\xi_r}  d\pmb{\eta_{p-r}} d\pmb{\wt{\xi}_{r-1}}  d\pmb{\wt{\eta}_{p-r-1}}   \vert\F f_p\vert^2\Big(\pmb{\eta_{p-r}},    \pmb{\wt{\xi}_{r-1}} , x- \tau(  \pmb{\eta_{p-r}}  ) - \tau(\pmb{\wt{\xi}_{r-1}})     \Big) \varphi(\xi_r) \notag  \\
&  \times    \vert\F f_p\vert^2 \Big(\pmb{\wt{\eta}_{p-r-1}},  y-     \tau(  \pmb{\wt{\eta}_{p-r-1}}   ) - \tau(\pmb{\xi_r}),    \pmb{\xi_r}   \Big)  \left(\prod_{i=1}^{r-1} \varphi(\xi_i)\varphi(\wt{\xi}_i)\right)    \mathbf{1}_{\{ \| \tau(\pmb{\xi_r}) + \tau(\pmb{\eta_{p-r}}) \| < \delta    \}}  \notag    \\
     &    \times \varphi(\eta_{p-r} ) \varphi\big(  \tau(  \pmb{\wt{\eta}_{p-r-1}}  ) + \tau(\pmb{\xi_r} )-y \big) \left(\prod_{j=1}^{p-r-1} \varphi(\eta_j)\varphi(\wt{\eta}_j)\right)  \varphi\big(  \tau(  \pmb{\eta_{p-r}}  ) +\tau(\pmb{\wt{\xi}_{r-1}})-x  \big).  \notag
\end{align}

  We remark that the function $\Psi_p^{(r,\delta)}$ is defined almost everywhere on $\R^{2d}$  and  with the same proof as  in Lemma \ref{lem2.1}, we can show that   $f_p \in L^1(\R^{pd})$ and $\varphi \in L^q(\R^d)$  for $q=p/(p-1)$ guarantee  {\bf (H3)}.    
 
 \medskip

  \begin{lemma} \label{lem2.3} 
  Suppose that  $f_p \in L^1(\R^{pd})\cap \H^{\odot p}$ and $\varphi \in L^q(\R^d)$, with $q=p/(p-1)$.  Then for every $r\in\{1,\ldots, p-1\}$ and $\delta>0$,  $\Psi_p^{(r,\delta)}$ is bounded continuous on $\R^{2d}$.  In particular  hypothesis   {\bf (H3)} is true.  
  \end{lemma}
For the sake of completeness, we provide a proof in Section \ref{tech}.

\medskip

\begin{theorem}   \label{thm2}  Fix an integer $p\ge 2$ and  $f_p\in \H^{\odot p}$ satisfying hypotheses  {\bf (H2)}  and {\bf (H3)}. 
Then,
\[
\frac {G_{p,R}} { R^{d/2}}      \xrightarrow[R\to+\infty]{\rm law } N(0,\sigma_p^2) , 
\]
where $\sigma_p^2= p! (2\pi)^{d} \omega_d \Psi_p(0)$, with $\Psi_p(0)$ given by \eqref{PSI0}. 
\end{theorem}

If   {\bf (H2)}  is replaced by the Maruyama's condition \eqref{M-cond}, we have the following corollary.

\begin{corollary}\label{Cor29} 

Fix an integer $p\ge 2$ and  $f_p\in \H^{\odot p}$ satisfying hypotheses   {\bf (H3)}. Assume that Maruyama's condition \eqref{M-cond} holds true.
Then,
\[
\frac {G_{p,R}}{\sigma_{p,R}}      \xrightarrow[R\to+\infty]{\rm law } N(0,1) , 
\]
with $ \sigma_{p,R}$ being the standard deviation of $G_{p,R}$.
\end{corollary}

We will omit the proof of this corollary, as it follows simply  from  Proposition \ref{Maru} and the following proof of Theorem \ref{thm2}.

\begin{proof}[Proof of Theorem \ref{thm2}]  It suffices to show the contraction condition (\ref{contractions}).  
 We spilt  the proof into several steps. We will use Fourier transform to rewrite \eqref{innerint} in  Steps 1-3 and  we will carry out the asymptotic analysis in Step 4.

\medskip

  \noindent{\it Step 1:}  Plancherel's formula implies
\begin{align*}
 &\int_{\R^{2rd}}    f_p^{x_1}( \pmb{s_{p-r}}, \pmb{a_r})  f_p^{x_2}(  \pmb{t_{p-r}}, \pmb{b_r} )  \prod_{i=1}^r  \gamma( a_i -b_i)  d\pmb{a_r} \, d\pmb{b_r}  \\
 &=   \int_{\R^{rd}} (\F_r f_p^{x_1})(\pmb{s_{p-r}}, \pmb{\xi_r}) (\F_r f_p^{x_2})(\pmb{t_{p-r}}, -\pmb{\xi_r})  \, \mu(d\pmb{\xi_r} ).
\end{align*}
and 
\begin{align*}
& \int_{\R^{2rd}}   f_p^{x_3}( \pmb{\wt{s}_{p-r}},\pmb{\wt{a}_r} )   f_p^{x_4}(   \pmb{\wt{t}_{p-r}}, \pmb{\wt{b}_r} )   \prod_{i=1}^r        \gamma(   \wt{a}_i -\wt{b}_i )    d\pmb{\wt{a}_r}\, d\pmb{\wt{b}_r}  \\
&=  \int_{\R^{rd}} (\F_r f_p^{x_3})(  \pmb{\wt{s}_{p-r}},\pmb{\wt{\xi}_r}) (\F_r f_p^{x_4})(   \pmb{\wt{t}_{p-r}}, -\pmb{\wt{\xi}_r} )\, \mu(d \pmb{\wt{\xi}_r})\,,
\end{align*}
where $\F_r$ denotes the Fourier transform with respect to the right-most $r$ variables. 

\medskip

\noindent
{\it Step 2:}  Similarly, we have
\begin{align*}
&  \int_{\R^{4(p-r)d} }  (\F_r f_p^{x_1})(\pmb{s_{p-r}}, \pmb{\xi_r}) (\F_r f_p^{x_2})(\pmb{t_{p-r}}, -\pmb{\xi_r})   (\F_r f_p^{x_3})(  \pmb{\wt{s}_{p-r}},\pmb{\wt{\xi}_r}) (\F_r f_p^{x_4})(   \pmb{\wt{t}_{p-r}}, -\pmb{\wt{\xi}_r} ) \\
& \qquad\qquad\qquad\qquad \times \left( \prod_{j=1}^{p-r}   \gamma(   t_i -\wt{t}_i )   \gamma(     \wt{s}_i -s_i ) \right)  \,  d\pmb{t_{p-r}} \, d\pmb{s_{p-r}} \, d\pmb{\wt{t}_{p-r}} \, d\pmb{\wt{s}_{p-r}} \\
&= \left( \int_{\R^{2(p-r)d} } (\F_r f_p^{x_1})(\pmb{s_{p-r}}, \pmb{\xi_r})     (\F_r f_p^{x_3})(  \pmb{\wt{s}_{p-r}},\pmb{\wt{\xi}_r})   \prod_{j=1}^{p-r}     \gamma(     \wt{s}_i -s_i )     d\pmb{s_{p-r}} \, d\pmb{\wt{s}_{p-r}}\right)  \\
& \quad \times   \left(  \int_{\R^{2(p-r)d} }  (\F_r f_p^{x_2})(\pmb{t_{p-r}}, -\pmb{\xi_r})    (\F_r f_p^{x_4})(   \pmb{\wt{t}_{p-r}}, -\pmb{\wt{\xi}_r} ) \prod_{j=1}^{p-r}   \gamma(   t_i -\wt{t}_i ) d\pmb{t_{p-r}} \,  d\pmb{\wt{t}_{p-r}}  \right)\\
& =  \left( \int_{\R^{pd-rd} } (\mathcal{F}_{p-r} \F_r f_p^{x_1})(\pmb{\eta_{p-r}},\pmb{\xi_r})    (\mathcal{F}_{p-r} \F_r f_p^{x_3})( -\pmb{\eta_{p-r}},  -\pmb{\wt{\xi}_r}) \,  \mu(d\pmb{\eta_{p-r}})   \right) 
\\&\qquad\quad \times \left( \int_{\R^{pd-rd} }  (\mathcal{F}_{p-r}  \F_r f_p^{x_2})( \pmb{\wt{\eta}_{p-r}},-\pmb{\xi_r})   (  \mathcal{F}_{p-r}  \F_r f_p^{x_4})(- \pmb{\wt{\eta}_{p-r}} , \pmb{\wt{\xi}_r}) \,  \mu(d \pmb{\wt{\eta}_{p-r}})  \right),
\end{align*}
where $\mathcal{F}_{p-r}$ denotes the Fourier transform with respect to the left-most $p-r$ variables.  It is clear that the composition of $\mathcal{F}_{p-r}$ and $\F_r$ is the usual Fourier transform.  

\medskip

\noindent
{\it Step 3:} 
 Using basic properties of the Fourier transform, we have $ (\mathcal{F}_{p-r} \F_r f_p^{x})(\pmb{\xi_p}) =e^{-\i x \cdot \tau(\pmb{\xi_p})} (\F f_p)(\pmb{\xi_p}  ) $.  So combining facts from the above steps yields that the second integral in \eqref{innerint} is equal to 
\begin{align*}
&  \int_{\R^{2pd}}\mu(d\pmb{\xi_r}) \, \mu(d\pmb{\wt{\xi}_r})  \mu(d\pmb{\eta_{p-r}})  \mu(d\pmb{\wt{\eta}_{p-r}}) (\F f_p)(\pmb{\eta_{p-r}},\pmb{\xi_r})   (\F f_p)( -\pmb{\eta_{p-r}},  -\pmb{\wt{\xi}_r})   \\
& \quad    \times  (\F f_p) ( \pmb{\wt{\eta}_{p-r}},-\pmb{\xi_r})   (    \F f_p)(- \pmb{\wt{\eta}_{p-r}} , \pmb{\wt{\xi}_r})  \,      ~ e^{-\i x_1\cdot (a+b) }  e^{-\i x_2\cdot (\wt{b}-a) }  e^{-\i x_3\cdot ( -\wt{a}-b ) }  e^{-\i x_4\cdot (\wt{a}-\wt{b} ) } \,,
\end{align*}
with the  notation $a = \tau(\pmb{\xi_r}),  b=\tau(\pmb{\eta_{p-r}}), \wt{a} = \tau(\pmb{\wt{\xi}_r})$ and $ \wt{b}=\tau(\pmb{\wt{\eta}_{p-r}})$ {\it throughout this proof.}

 It follows from  Lemma \ref{Bessel} that 
\begin{align*}
&\quad   \int_{ B_R^4} e^{-\i x_1\cdot (a+b) }  e^{-\i x_2\cdot (\wt{b}-a) }  e^{-\i x_3\cdot ( -\wt{a}-b ) }  e^{-\i x_4\cdot (\wt{a}-\wt{b} ) } \, d\pmb{x_4}\\
&=  (2\pi R )^{2d} \| a + b \| ^{-d/2}  \| \wt{b} - a \| ^{-d/2}  \| \wt{a} + b   \| ^{-d/2}   \| \wt{a}-\wt{b} \| ^{-d/2} \\
& \qquad\qquad \times J_{d/2}\big( R\| a + b \| \big) J_{d/2}\big( R\|\wt{b} - a \| \big)J_{d/2}\big( R\| \wt{a} + b \| \big) J_{d/2}\big( R\|\wt{ a} - \wt{b} \| \big) .
\end{align*}
Thus, we have for $r\in\{1, \ldots, p-1\}$,
\begin{align}
&\quad \mathcal{I}_{R}: = (2\pi R)^{-2d}   \big\| g_{p,R} \otimes_r g_{p,R} \big\|^2_{\H^{\otimes (2p-2r)}}  \label{kk2}\\
&=   \int_{\R^{2pd}}\mu(d\pmb{\xi_r}) \, \mu(d\pmb{\wt{\xi}_r})  \mu(d\pmb{\eta_{p-r}})  \mu(d\pmb{\wt{\eta}_{p-r}}) (\F f_p)(\pmb{\eta_{p-r}},\pmb{\xi_r})   (\F f_p)( -\pmb{\eta_{p-r}},  -\pmb{\wt{\xi}_r})  \notag  \\
&\quad     \times (\F f_p) ( \pmb{\wt{\eta}_{p-r}},-\pmb{\xi_r})   (    \F f_p)(- \pmb{\wt{\eta}_{p-r}} , \pmb{\wt{\xi}_r})   \| a + b \| ^{-d/2}  \| \wt{b} - a \| ^{-d/2}  \| \wt{a} + b   \| ^{-d/2}   \notag \\
&\qquad  \times  \| \wt{a}-\wt{b} \| ^{-d/2}J_{d/2}\big( R\| a + b \| \big) J_{d/2}\big( R\|\wt{b} - a \| \big)J_{d/2}\big( R\| \wt{a} + b \| \big) J_{d/2}\big( R\|\wt{ a} - \wt{b} \| \big) .\notag
\end{align}

\medskip

\noindent{\it Step 4:}  In what follows, we prove   that $\lim_{R\to+ \infty} \mathcal{I}_{R} =0$.

\medskip

We decompose the above integral   into  two parts: $\mathcal{I}_{R} ={\displaystyle \int_{\R^{pd}\times \mathcal{D}_\delta} +  \int_{\R^{pd}\times \mathcal{D}^c_\delta} }$, with
\[
 \mathcal{D}_\delta = \{  (\pmb{\xi_r},  \pmb{\eta_{p-r}})\in \R^{pd}: 
\| a +b \| \geq \delta \}  .
\]
To ease the presentation, we introduce for every $\delta\in[0,\infty)$,
\[
\mathbf{T}_\delta(R) :=  \int_{\{  \| \tau(\pmb{\xi_p})  \| \geq \delta\} } \mu(d\pmb{\xi_p})   \vert\F f_p \vert^2 (\pmb{\xi_p})       \| \tau(\pmb{\xi_p}) \| ^{-d}     J_{d/2}\big( R\| \tau(\pmb{\xi_p})  \| \big)^2 .
\]
Note that, by \eqref{need002}  and the symmetry of $\mu$, we have
\[
\mathbf{T}_0(R)=\frac{\Var(G_{p,R})}{p! (2\pi R)^d }, 
\]
which, under the hypothesis {\bf (H2)}, converges to $ \omega_d \Psi_p(0)$, as $R\to+\infty$. 
 
\medskip

 Now on $\R^{pd}\times\mathcal{D}_\delta$, we can write, using Cauchy-Schwarz inequality, 
\begin{align*}
&\left\vert \int_{\R^{pd}\times \mathcal{D}_\delta}  \right\vert \leq   \int_{\R^{pd}} \, \mu(d\pmb{\wt{\xi}_r})   \mu(d\pmb{\wt{\eta}_{p-r}})   \vert\F f_p\vert (-\pmb{\wt{\eta}_{p-r}}, \pmb{\wt{\xi}_r})     \| \wt{a}-\wt{b} \| ^{-d/2} \big\vert J_{d/2}\big( R\|\wt{ a} - \wt{b} \| \big) \big\vert   \\
&         \times  \int_{\mathcal{D}_\delta} \mu(d\pmb{\xi_r})\mu(d\pmb{\eta_{p-r}})   \vert\F f_p \vert( \pmb{\eta_{p-r}}, \pmb{\xi_r})     \| a + b \| ^{-d/2}  \big\vert   J_{d/2}\big( R\| a + b \| \big)\big\vert   \vert\F f_p \vert( -\pmb{\eta_{p-r}}, -\pmb{\wt{\xi}_r})           \\
&      \times     \vert\F f_p \vert(\pmb{\wt{\eta}_{p-r}}, -\pmb{\xi_r})    \| \wt{b} - a \| ^{-d/2}  \| \wt{a} + b   \| ^{-d/2}\big\vert J_{d/2}\big( R\|\wt{b} - a \| \big)J_{d/2}\big( R\| \wt{a} + b \| \big) \big\vert \\
 &\leq  \sqrt{\mathbf{T}_\delta(R)} \int_{\R^{pd}} \, \mu(d\pmb{\wt{\xi}_r})   \mu(d\pmb{\wt{\eta}_{p-r}})   \vert\F f_p\vert (-\pmb{\wt{\eta}_{p-r}}, \pmb{\wt{\xi}_r})     \| \wt{a}-\wt{b} \| ^{-d/2} \big\vert  J_{d/2}\big( R\|\wt{ a} - \wt{b} \| \big) \big\vert    \\
 &\qquad \times    \Bigg(   \int_{\mathcal{D}_\delta}  \mu(d\pmb{\xi_r})\mu(d\pmb{\eta_{p-r}})    \vert\F f_p \vert^2( -\pmb{\eta_{p-r}}, -\pmb{\wt{\xi}_r})         \vert\F f_p \vert^2(\pmb{\wt{\eta}_{p-r}}, -\pmb{\xi_r})    \\
 &\qquad\qquad\qquad \times   \| \wt{b} - a \| ^{-d}  \| \wt{a} + b   \| ^{-d}J_{d/2}\big( R\|\wt{b} - a \| \big)^2  J_{d/2}\big( R\| \wt{a} + b \| \big)^2 \Bigg)^{1/2}\\
 &\leq \sqrt{ \mathbf{T}_\delta(R) \mathbf{T}_0(R)} \Bigg(        \int_{\R^{2pd}} \mu(d\pmb{\wt{\xi}_r})   \mu(d\pmb{\wt{\eta}_{p-r}}) \mu(d\pmb{\xi_r})\mu(d\pmb{\eta_{p-r}})  \vert\F f_p \vert^2( -\pmb{\eta_{p-r}}, -\pmb{\wt{\xi}_r})     \\
 &\qquad \times    \vert\F f_p \vert^2(\pmb{\wt{\eta}_{p-r}}, -\pmb{\xi_r})      \| \wt{b} - a \| ^{-d}  \| \wt{a} + b   \| ^{-d}J_{d/2}\big( R\|\wt{b} - a \| \big)^2  J_{d/2}\big( R\| \wt{a} + b \| \big)^2 \Bigg)^{1/2}  \\
&= \mathbf{T}_0(R)^{3/2} \sqrt{\mathbf{T}_\delta(R)}\,.
\end{align*}
We claim that  
\begin{align}\label{claim1}
\text{for any fixed $\delta > 0$, $\mathbf{T}_\delta(R)\to 0$, as $R\to+\infty$.}
\end{align} Indeed, on $\{ \| \tau(\pmb{\xi_p})  \| \geq \delta>0 \}$, $   J_{d/2}\big( R\| \tau(\pmb{\xi_p})  \| \big)^2$  converges to zero, as $R\to+\infty$; and clearly,
\[
\mathbf{T}_\delta(R) \leq   \delta^{-d} \left( \sup_{t\in\R_+} J_{d/2}(t)^2 \right) \int_{\| \tau(\pmb{\xi_p})  \| \geq \delta} \mu(d\pmb{\xi_p})   \vert\F f_p\vert^2 (\pmb{\xi_p})   < \infty\,,    
\]
so     claim \eqref{claim1} follows from  the dominated convergence theorem. Therefore,  the first part $\int_{\R^{pd}\times \mathcal{D}_\delta}  $ goes to zero, as $R$ tends to infinity. \\

Then, it remains to estimate the integral over $\R^{pd}\times\mathcal{D}_\delta^c$. Similarly,   we  obtain, by applying  Cauchy-Schwarz inequality,    
\begin{align*}
  \left\vert \int_{\R^{pd}\times \mathcal{D}_\delta^c}\right\vert & \leq  \int_{ \mathcal{D}_\delta^c}   \mu(d\pmb{\xi_r}) \mu( d \pmb{\eta_{p-r}} )   \| a + b \| ^{-d/2} \big\vert J_{d/2}\big( R\| a + b \| \big) \big\vert \vert \F f_p\vert ( \pmb{\eta_{p-r}}, \pmb{\xi_r})    \\
  &   \times \sqrt{\mathbf{T}_0(R)} \Bigg( \int_{\R^{pd}}  \| \wt{a}+b \| ^{-d} \| \wt{b}-a \| ^{-d}  J_{d/2}\big( R\|\wt{ a} +b\| \big)^2 J_{d/2}\big( R\|\wt{ b} -a \|\big)^2\\
  &    \qquad\times 
     \vert\F f_p\vert^2  (-\pmb{\eta_{p-r}}, -\pmb{\wt{\xi}_r})       \vert \F f_p\vert^2(\pmb{\wt{\eta}_{p-r}}, -\pmb{\xi_r})   
 \mu\big( d\pmb{\wt{\xi}_r} \big)  \mu\big(  d \pmb{\wt{\eta}_{p-r }} \big)\Bigg)^{  1/2} .
 \end{align*}
Recall that $\mu$ is symmetric. We can write,  after the change of variable ($\pmb{\wt{\eta}_{p-r}}\to -\pmb{\wt{\eta}_{p-r}}$) and then applying  Cauchy-Schwarz inequality,  
\begin{align*}
  \left\vert \int_{\R^{pd}\times \mathcal{D}_\delta^c}\right\vert & \leq \mathbf{T}_0(R)  \mathbf{K}_R,
 \end{align*}
where  
\begin{align*}
\mathbf{K}_R:&=\int_{\R^{pd}\times \{ \| a+b\| < \delta\} }   \mu\big(d\pmb{\xi_r}\big) \mu\big( d \pmb{\eta_{p-r}} \big)  \mu\big( d\pmb{\wt{\xi}_r} \big)  \mu\big(  d \pmb{\wt{\eta}_{p-r }} \big)   \| \wt{a}+b \| ^{-d} \| a+\wt{b} \| ^{-d}  \\
&\qquad\qquad  \times J_{d/2}\big( R\|\wt{ a} +b\| \big)^2 J_{d/2}\big( R\| a+\wt{ b}  \|\big)^2  \vert\F f_p\vert^2(\pmb{\eta_{p-r}}, \pmb{\wt{\xi}_r})      \vert\F f_p\vert^2(\pmb{\wt{\eta}_{p-r}}, \pmb{\xi_r}).
 \end{align*}
  From previous discussion, it holds under  hypothesis {\bf (H2)} that 
  $$\sup\big\{ \mathbf{T}_0(R) : R>0 \big\} < +\infty.$$ So it remains to show that  $\mathbf{K}_R\to 0$, as $R\to+\infty$.  
  
  \medskip
 


Making the following change of variables
 \begin{align*}
& \wt{a} + b \to x \,, \quad (\pmb{\eta_{p-r}}, \pmb{\wt{\xi}_r})  \to \Big(\pmb{\eta_{p-r}}, \pmb{\wt{\xi}_{r-1}}, x - \tau(  \pmb{\eta_{p-r}}  ) - \tau(\pmb{\wt{\xi}_{r-1}})  \Big)  \\
& \wt{b} + a \to y \,, \quad (\pmb{\wt{\eta}_{p-r}}, \pmb{\xi_r})  \to \Big(\pmb{\wt{\eta}_{p-r-1}}, y -     \tau(  \pmb{\wt{\eta}_{p-r-1}}   ) - \tau(\pmb{\xi_r}),    \pmb{\xi_r}   \Big)  
 \end{align*}
  yields
  \begin{align*}
  \mathbf{K}_{R} &= \omega_d^2\int_{\R^{2d}} dxdy \ell_R(x) \ell_R(y) \Psi_p^{(r,\delta)}(x,y),
  \end{align*}
  where 
 $ \Psi_p^{(r,\delta)}(x,y)$ is defined in \eqref{k6}. By our hypothesis {\bf (H3)}, we have as $R\to+\infty$, that $  \omega_d^{-2} \mathbf{K}_{R}$ is convergent to 
 \begin{align*}
&\quad \Psi_p^{(r,\delta)}(0,0)\\
&= \int_{\R^{2pd-2d}} d\pmb{\xi_r}  d\pmb{\eta_{p-r}} d\pmb{\wt{\xi}_{r-1}}  d\pmb{\wt{\eta}_{p-r-1}}   \vert\F f_p\vert^2\Big(\pmb{\eta_{p-r}},    \pmb{\wt{\xi}_{r-1}} , - \tau(  \pmb{\eta_{p-r}}  ) - \tau(\pmb{\wt{\xi}_{r-1}})     \Big)  \\
&  \times    \vert\F f_p\vert^2 \Big(\pmb{\wt{\eta}_{p-r-1}},  -     \tau(  \pmb{\wt{\eta}_{p-r-1}}   ) - \tau(\pmb{\xi_r}),    \pmb{\xi_r}   \Big)  \left(\prod_{i=1}^{r-1} \varphi(\xi_i)\varphi(\wt{\xi}_i)\right)   \varphi\big(  \tau(  \pmb{\eta_{p-r}}  ) +\tau(\pmb{\wt{\xi}_{r-1}} ) \big)\\
     &   \times \varphi(\xi_r)\varphi(\eta_{p-r} ) \varphi\big(  \tau(  \pmb{\wt{\eta}_{p-r-1}}  ) + \tau(\pmb{\xi_r} ) \big) \left(\prod_{j=1}^{p-r-1} \varphi(\eta_j)\varphi(\wt{\eta}_j)\right)\mathbf{1}_{\{ \| \tau(\pmb{\xi_r}) + \tau(\pmb{\eta_{p-r}}) \| < \delta    \}} \,,
\end{align*}
 which converges to zero, as $\delta\downarrow 0$.  This concludes our proof. \qedhere

\end{proof}

\bigskip

Recall the Hilbert-space notation $\H_\mu $ and $ \H_\mu^{\otimes p}$  from the beginning of Section \ref{sec2}. It is clear that 
$$\pmb{\xi_p}\in\R^{pd}\lmto F_R(\pmb{\xi_p}) := (\F f_p)( \pmb{\xi_p} ) \| \tau( \pmb{\xi_p}) \| ^{-d/2} J_{d/2}( R \| \tau( \pmb{\xi_p} ) \| )$$ 
belongs to $\H_\mu^{\otimes p}$ for each $R>0$, since $\F f_p\in \H_\mu^{\otimes p}$ and   $ \| \tau( \pmb{\xi_p}) \| ^{-d/2} J_{d/2}( R \| \tau( \pmb{\xi_p} ) \| )$ is uniformly bounded for any given $R>0$ (see  Lemma \ref{Bessel}).
     We can also define the corresponding contractions in this framework. For $h_1\in\H_\mu^{\otimes p}$ and $h_2\in\H_\mu^{\otimes q}$ ($p,q\in\N$), their $r$-contraction, with $0\leq r \leq p\wedge q$, belongs to $\H_\mu^{\otimes p+q-2r}$ and is defined by 
     \[
    (  h_1\otimes_{r,\mu} h_2  )\big(  \pmb{\xi_{p-r} },  \pmb{\eta_{p-r} }  \big) =  \int_{\R^{rd}} h_1(  \pmb{\xi_{p-r} },  \pmb{a_r }  ) \overline{h_2}\big( \pmb{\eta_{p-r} },  \pmb{a_r }\big) \, \mu(d\pmb{a_r} )\,.
          \]
   One should not confuse this notion with the one introduced in {\bf \small Notation A}.  
     
     \medskip

   With the notation $F_R$ and $ \otimes_{r,\mu}$, we can rewrite $\mathcal{I}_R$ in \eqref{kk2} as follows:
\begin{align*}
\mathcal{I}_R & =\int_{\R^{2pd}}d \mu ~ F_R( \pmb{\eta_{p-r}}, \pmb{\xi_r} )\overline{F_R}( \pmb{\eta_{p-r}}, \pmb{\wt{\xi}_r} ) \overline{F_R}(\pmb{\wt{\eta}_{p-r}}, \pmb{\xi_r} )      F_R( \pmb{\wt{\eta}_{p-r}}, \pmb{\wt{\xi}_r} )  \\
&= \int_{\R^{2pd}} \mu(d\pmb{\eta_{p-r}})  \mu(d\pmb{\wt{\eta}_{p-r}}) \big(F_R\otimes_{r,\mu} F_R\big)\big( \pmb{\eta_{p-r}},  \pmb{\wt{\eta}_{p-r}}   \big)\big(F_R\otimes_{r,\mu} F_R\big)\big(  \pmb{\wt{\eta}_{p-r}},\pmb{\eta_{p-r}}   \big)\\
& =\big\| F_R \otimes_{r,\mu} F_R \big\|^2_{\H_\mu^{\otimes 2p-2r}}\,,
\end{align*}
where we used the fact that $\big(F_R\otimes_{r,\mu} F_R\big)\big( \pmb{\eta_{p-r}},  \pmb{\wt{\eta}_{p-r}}   \big) = \overline{\big(F_R\otimes_{r,\mu} F_R\big)}\big(  \pmb{\wt{\eta}_{p-r}}, \pmb{\eta_{p-r}}   \big)$, which follows simply from the definition of contraction. Hence, we can formulate the following Fourth Moment Theorem.

\begin{theorem}  Fix an integer $p\geq 2$ and $f_p\in\H^{\odot p}$.   Assume  {\bf (H2)}, which implies that, in view of \eqref{need002},
\begin{align}\label{varp}
\sigma^2_p: = p! (2\pi)^d \lim_{R\to+\infty} \big\| F_R \big\|^2_{\H_\mu^{\otimes p}} \in [ 0, +\infty) \,.
\end{align}
Then, the following statements are equivalent: \medskip
\begin{enumerate}

 \item[\bf (S1)] $\dfrac{G_{p,R}}{R^{d/2}}$ converges in law to $N(0,\sigma_p^2)$, as $R\to+\infty$;
 \medskip
 \item[\bf (S2)] $\E\big[ G_{p,R}^4 \big] R^{-2d}$ converges to $3\sigma_p^4$, as $R\to+\infty$;
 \medskip
 \item[\bf (S3)] For every $r\in\{1, \ldots, p-1\}$, $\| F_R \otimes_{r,\mu} F_R\| _{\H_\mu^{\otimes 2p-2r}}\to 0$, as $R\to+\infty$.
  
\end{enumerate}     
 \end{theorem}

\begin{remark}  
(i) Recall from Lemma \ref{Bessel} that on $\R_+$, $J_{d/2}(x)\leq C\big(1 \wedge \frac{1}{\sqrt{x} }  \big) $. Therefore, we  obtain the following   estimates:
\begin{align*}
\| F_R \otimes_{r,\mu} F_R\| _{\H_\mu^{\otimes 2p-2r}} &\leq C \big\| F^{(1)} \otimes_{r,\mu} F^{(1)}  \big\| _{\H_\mu^{\otimes 2p-2r}}
\end{align*}
and
\begin{align*}
\| F_R \otimes_{r,\mu} F_R\| _{\H_\mu^{\otimes 2p-2r}} &\leq \frac{C}{R^2} \big\| F^{(2)} \otimes_{r,\mu} F^{(2)}  \big\| _{\H_\mu^{\otimes 2p-2r}}, 
\end{align*}
with $F^{(j)}(\pmb{\xi_p} ) : = \vert\F f_p\vert( \pmb{\xi_p}  )     \| \tau( \pmb{\xi_p}) \| ^{-\frac{d+j-1}{2}}$, $j=1,2$. As a consequence,

 \medskip
 
 (1) if $\| F^{(1)} \otimes_{r,\mu} F^{(1)}   \| _{\H_\mu^{\otimes 2p-2r}} <\infty$  and $\mu$ admits a spectral density, then by the dominated convergence theorem, we have $\| F_R \otimes_{r,\mu} F_R\| _{\H_\mu^{\otimes 2p-2r}} \to 0$, which implies the   Gaussian fluctuation;

\medskip

(2) if $\| F^{(2)} \otimes_{r,\mu} F^{(2)}   \| _{\H_\mu^{\otimes 2p-2r}} <\infty$, we deduce from \eqref{dist} that 
$$
d_{\rm TV}\Big(G_{p,R} /\sigma_{p,R} , N(0,1)\Big) \leq C/R \,.
$$
(ii) In view of the Cauchy-Schwarz inequality for contractions, one has 
$$
\| F^{(j)} \otimes_{r,\mu} F^{(j)}   \| _{\H_\mu^{\otimes 2p-2r}} \leq \| F^{(j)} \| ^2_{\H_\mu^{\otimes p}}\quad\text{for $j=1,2$.}
$$ 
So one may intend to assume 
 \begin{align}
 \| F^{(1)} \| _{\H_\mu^{\otimes p}} \wedge \| F^{(2)} \| _{\H_\mu^{\otimes p}} < \infty \,,    \label{intend}
 \end{align}
which, however, is not reasonable in our framework.  In fact,  \eqref{varp} and \eqref{x=0} tell us that $\| F_R \|^2_{\H_\mu^{\otimes p}}$, which is equal to
\begin{align*}
 \frac{R^d}{2^d \Gamma(\frac{d}{2} +1   )^2} \int_{\{ \tau(\pmb{\xi_p} ) =0 \}}  \vert\F f_p\vert^2(\pmb{\xi_p})    \mu(d\pmb{\xi_p}) + \int_{\{ \| \tau(\pmb{\xi_p} ) \| >0 \}}   \vert \F f_p\vert^2(\pmb{\xi_p}) \ell_R(  \tau(\pmb{\xi_p} ))    \mu(d\pmb{\xi_p}) ,
\end{align*}
converges to $\dfrac{\sigma_p^2}{p! (2\pi)^d}$; if we assume \eqref{intend} or we assume the weaker condition  
\begin{align*}
 \int_{\R^{pd}} \left(   \|   \tau(\pmb{\xi_p}) \| ^{-d-1} \wedge   \|   \tau(\pmb{\xi_p}) \| ^{-d} \right)    \vert\F f_p\vert^2 (\pmb{\xi_p})     ~\mu(d\pmb{\xi_p}) <\infty\,,
 \end{align*}
then the integral over $\{ \| \tau(\pmb{\xi_p} ) \| >0 \}$ vanishes asymptotically, so that we can write
 \begin{align}\label{converge1}
  \frac{R^d}{2^d \Gamma(\frac{d}{2} +1   )^2} \int_{\{ \tau(\pmb{\xi_p} ) =0 \}}    \vert\F f_p\vert^2 (\pmb{\xi_p})   ~\mu(d\pmb{\xi_p})  \xrightarrow{R\to+\infty}   \frac{\sigma_p^2}{p! (2\pi)^d} \,.
 \end{align}
 This forces the integral in \eqref{converge1}   to be zero by dominated convergence, so that $\sigma_p^2 = 0$.

 \end{remark}
 
 \medskip

\subsection{Chaotic central limit theorems} \label{23}

As a continuation of previous section,  we consider the case of infinitely many chaoses and we derive  a chaotic central limit theorem. 
Recall $F\in L^2\big(\Omega\big)$  admits the following chaos expansion \eqref{Fchaos} with Hermite rank $m\geq 1$:
\begin{equation*}
F(W) = \sum_{p\geq m} I^W_p(f_p) \quad{\rm with} \quad f_p\in \H^{\odot p}\,.
\end{equation*}

Let us introduce the following natural hypothesis:
\[ 
{\bf (H4)} \qquad\qquad   \sum_{p\geq m}   p!       \int_{\R^{2pd}}  d\pmb{t_p} \, d\pmb{s_p}  \,    |f_p|(\pmb{s_p}) |f_p|(\pmb{t_p}) \int_{\R^d} \prod_{i=1}^p \vert \gamma\vert\big(t_i-s_i+z\big) dz  < \infty.  \qquad\qquad
 \]
 
 \medskip
 Recall the notation $\kappa_p$ from \eqref{kappa2} and we put
 \begin{equation*}  
\| f_p\|_{\kappa_p} ^2:= \int_{\R^{2dp}}    f_p(\pmb{s_p})f_p(\pmb{t_p})     \kappa_p( \pmb{t_p} - \pmb{s_p}   )  \, d\pmb{t_p} \, d\pmb{s_p} \,.
\end{equation*}
So under   {\bf (H4)},
\begin{equation} \label{sigma}
  \sigma^2 :=  \omega_d\sum_{p\geq m}   p!       \| f_p\|_{\kappa_p}^2\in[0,\infty).
 \end{equation}

Note that an immediate consequence of our hypothesis {\bf (H4)} is the following   result
\begin{align}\label{ded}
\lim_{N\to+\infty} \sup_{R>0} R^{-d}\sum_{q\geq N+1}   \Var\left( \int_{B_R} I^{W}_p(f_p^x) \, dx \right)   = 0 \,.
\end{align}
In fact, one can write,  similarly as before,
\begin{align*}
&\qquad \sup_{R>0}  \frac{1}{\omega_d R^d}\sum_{q\geq N+1}   \Var\left( \int_{B_R} I^{W}_p(f^x_p) \, dx \right) \\
&  = \sum_{q\geq N+1}  p!       \int_{\R^{2pd}}  d\pmb{t_p} \, d\pmb{s_p} \,    f_p(\pmb{s_p})f_p(\pmb{t_p})  \left( \int_{ \R^d} \frac{\text{vol}\big( B_R \cap B_R(-z)   \big) }{\text{vol}(B_R)}    \prod_{i=1}^p\gamma\big(t_i-s_i+z\big) dz  \right) \\
&\leq   \sum_{q\geq N+1}  p!       \int_{\R^{2pd}}  d\pmb{t_p} \, d\pmb{s_p} \,  \big\vert  f_p(\pmb{s_p})f_p(\pmb{t_p})  \big\vert  \left( \int_{ \R^d}     \prod_{i=1}^p  \vert\gamma\vert \big(t_i-s_i+z\big) dz  \right)  \xrightarrow{N\to+\infty} 0 \,.
\end{align*}

Now we state our main result   as a consequence of \eqref{ded}, Theorems \ref{thm1} and  \ref{thm00}.
 
\begin{theorem} \label{CCLT1}   Suppose $F\in L^2(\Omega )$ admits the   chaos expansion \eqref{Fchaos} with Hermite rank $m\geq 2$
and assume that {\bf (H4)} is satisfied.
Suppose that for each $p \geq m$,    the  kernel $f_p\in\H^{\odot p}$  satisfies  
{\rm (i)} {\it or} {\rm (ii)} in Theorem \ref{thm1}. Let $\sigma^2$ be given by \eqref{sigma}.
Then,  as $R\to+\infty$,
\[
R^{-d/2} \int_{B_R} U_xF(W) \, dx \quad\text{converges in law to} \quad N(0, \sigma^2 ) \,.
\]
\end{theorem}

  \begin{remark}\label{rem66}  (1) In Theorem \ref{CCLT1}, we exclude the first chaos for the following obvious reason. Under the assumption that $\{f_1, \gamma\}\subset L^1(\R^d)$, 
  $
  R^{-d/2} \int_{B_R} I_1^W(f_1^x)dx$
 is a  centered Gaussian  random variable with variance tending to  $\omega_d\| f_1\|^2_{L^1(\R^d)} \int_{\R^d}\gamma(z)dz$, as $R\to+\infty$; see point (i) in Remark \ref{rem1}.

  \medskip
  
  (2)   Suppose $\gamma(0) < + \infty$ or equivalently $\mu(\R^d)<+\infty$, then $\gamma=\F\mu$ is a   function bounded by $\gamma(0)$.  If $\gamma \in L^{m}(\R^d)$ (for some integer $m\geq 1$), then $\gamma \in L^p(\R^d)$ for any 
    $p\ge m$, so that $ \| \gamma\|^p_{L^p(\R^d)} \leq \gamma(0)^{p-m}  \| \gamma\|^m_{L^m(\R^d)} $. As a result,
    \begin{align*}
    &\quad  \sum_{p\geq m}   p!       \int_{\R^{2pd}}  d\pmb{t_p} \, d\pmb{s_p}  \,    |f_p|(\pmb{s_p}) |f_p|(\pmb{t_p}) \int_{\R^d} \prod_{i=1}^p\vert \gamma\vert\big(t_i-s_i+z\big) dz \\
     &\leq  \sum_{p\geq m}   p!    \| \gamma\|^p_{L^p(\R^d)} \| f_p \|^2_{L^1(\R^{pd})} \leq C \sum_{p\geq m}   p!   \gamma(0)^p \| f_p \|^2_{L^1(\R^{pd})} .
    \end{align*}
    This tells us that  condition \eqref{gamma0} implies {\bf (H4)}, so Theorem \ref{cor0} stands as an easy corollary of our Theorem \ref{CCLT1} and previous point (1).

 \end{remark}

  We can formulate another chaotic central limit theorem based on the spectral measure.      
  
\begin{theorem}\label{CCLT2}  Suppose that  $F\in L^2(\Omega )$ admits the   chaos expansion \eqref{Fchaos} with Hermite rank $m\geq 1$.
Assume that  the spectral measure has a density. 
Suppose that for each $p \geq m$,  the function $\Psi_p$ defined in \eqref{Psi} is continuous at zero and the following boundedness condition holds
(which  implies {\bf (H2)} for each $p$):
\[
{\bf (H4')} \qquad\qquad\qquad\qquad\qquad\qquad    \sum_{p\geq m}   p!      \| \Psi_p\|_\infty  < \infty.\qquad\qquad\qquad\qquad\qquad\qquad\qquad
\]
 Assume additionally that  hypothesis  {\bf (H3)}  holds for each $p\geq m$. Then, 
\[
R^{-d/2} \int_{B_R} U_xF(W) \, dx \xrightarrow[\rm law]{R\to+\infty} N\left(0, (2\pi)^d\omega_d \sum_{p\geq m} p! \Psi_p(0) \right) \,.
\]
  \end{theorem}
 \begin{proof}  For $m=1$, we should consider the first chaos and it is clear that $R^{-d/2} G_{1,R}$ is centered Gaussian with variance tending to $\omega_d (2\pi)^d\Psi_1(0)$.   
 
 Now let us consider higher-order chaoses. For each $p\geq m\vee 2$,  hypotheses {\bf (H2)} and {\bf (H3)} hold true. This implies that $G_{p,R}R^{-d/2}$ converges in law to $N(0,\sigma_p^2)$, with $\sigma_p$ introduced in Theorem \ref{thm1}. In view of the chaotic central limit theorem (Theorem \ref{thm00}), it remains to check  condition \eqref{ded}. We can write
 \begin{align*}
 \sum_{p\geq N+1}     \frac{\Var\big( G_{p,R}\big)}{\omega_d R^d} =(2\pi)^d  \sum_{p\geq N+1}    p! \int_{\R^d} \ell_R(x) \Psi_p(x) \, dx \leq (2\pi)^d  \sum_{p\geq N+1}    p!  \|\Psi_p \| _\infty \,,
 \end{align*}
where the last inequality follows from the fact that $\ell_R(x)dx$ is a probability measure on $\R^d$; so hypothesis {\bf (H4')} implies \eqref{ded}. Hence, our proof is finished.\qedhere

 \end{proof}

  \begin{corollary}  Suppose that  $F\in L^2(\Omega )$ admits the   chaos expansion \eqref{Fchaos} with Hermite rank $m\geq 1$ and for each $p\geq m$, the kernel $f_p$ belongs to $ L^1(\R^{pd})\cap \H^{\odot p}$. Assume that the spectral measure $\mu$ is finite with spectral density $\varphi$ such that $\varphi$ is uniformly bounded with continuity at zero and 
  \begin{align}\label{ccond}
   \sum_{p\geq m} p!  \| \F f_p \|_\infty^2  \|\varphi\|_{L^1(\R^d)}^{p}   < \infty \,.
  \end{align}
  Then, ${\displaystyle
R^{-d/2} \int_{B_R} U_xF(W) \, dx \xrightarrow[\rm law]{R\to+\infty} N\left(0, (2\pi)^d\omega_d \sum_{p\geq m} p! \Psi_p(0) \right) \,.
}$
 \end{corollary}

 \begin{proof} Note that  $\mu$ is finite, which is equivalent to $\varphi\in L^1(\R^d)$. This implies with boundedness of $\varphi$ that $\varphi\in L^q(\R^d)$ for any $q>1$. It is clear that  for any $p\geq 2\vee m$, $f_p\in L^1(\R^d)\cap \H^{\odot p}$ and $\gamma\in L^{p/(p-1)}(\R^d)$, so Lemma \ref{lem2.3} and Lemma \ref{lem2.1} ensure that  hypotheses {\bf (H2)} and {\bf (H3)} are valid on the $p$th chaos.  
 
 If $F$ has the first chaos with $f_1\in L^1(\R^d)$, then $\Psi_1$ is uniformly bounded with continuity at zero (the continuity of $\varphi$ at zero is only required at this point). Therefore, 
 $
G_{1, R}  R^{-d/2} 
 $ converges in law to a centered Gaussian with variance $(2\pi)^d \Psi_1(0)$.

 It remains to notice that $\Psi_p(x) \leq \big\| |\F f_p |^2 \big\|_\infty \varphi^{\ast p}(x) \leq   \big\| |\F f_p |^2 \big\|_\infty \| \varphi\|^p_{L^{p/(p-1)}(\R^d)}$ by \eqref{conv-bdd}. We know that
 $
  \| \varphi\|^p_{L^{p/(p-1)}(\R^d)} \leq \|\varphi\|_\infty  \|\varphi\|_{L^1(\R^d)}^{p-1}
 $
 so that {\bf (H4')} holds in this setting. To see this, we write 
 \[
 \sum_{p\geq m} p! \| \Psi_p\|_\infty \leq   C \sum_{p\geq m} p!  \| \F f_p \|_\infty^2  \|\varphi\|_{L^1(\R^d)}^{p}   \,,
 \]
that is,  {\bf (H4')} is implied by \eqref{ccond}.  Hence, the proof is done by applying Theorem \ref{CCLT2}.            \end{proof}

\section{Proof of Theorems \ref{thmSHE},   \ref{Nchaotic} and \ref{NCF} }\label{sec3}

Let  $u_{t,x}$ be the mild solution to the   linear stochastic heat equation \eqref{pl1}
with initial condition $u_{0,x}=1$ for all $x\in \R^d$,
driven by a Gaussian noise with temporal and spatial covariance kernels being $\gamma_0$  and $\gamma_1$, respectively. We assume  $\gamma_0:\R\to [0,\infty]$  locally integrable and   the Fourier transform of  $\gamma_1$ is a nonnegative tempered measure $\mu_1$ that satisfies the Dalang's condition \eqref{Da}.

Recall that
 \[
 A_t(R) = \int_{B_R} \big( u_{t,x} - 1 \big) \, dx = \sum_{p=1} ^\infty  I_p^W\left( \int_{B_R} f_{t,x,p}dx \right),
 \]
 where, for any integer $p\ge 1$,    $f_{t,x,p}$ is the kernel appearing in the Wiener chaos expansion of $u_{t,x}$ (see \eqref{form-f}).

Let us introduce some notation for later convenience. 

\medskip
\noindent{\bf Notation B.}   For given $t>0$ and $p\in\N$, $\Delta_p(t) = \{ \pmb{s_p}\in\R_+^p: t> s_1 > \ldots > s_p > 0\}$ and $\text{SIM}_p(t)=\{\pmb{s_p}\in\R_+^p: s_1 + \cdots +s_p \leq t \} $. For $\sigma\in\mathfrak{S}_p$, we write $\pmb{x^\sigma_p} =  (x^\sigma_1, \dots, x^\sigma_p) =  (x_{\sigma(1)}, \dots, x_{\sigma(p)})  $, so $\pmb{s^\sigma_p}\in\Delta_p(t)$ means $t > s_{\sigma(1)} > \cdots > s_{\sigma(p)}$ and we write $\int_{\Delta_p(t)} d\pmb{s^\sigma_p}$ for $\int_{[0,t]^p} d\pmb{s_p} {\bf 1}_{\Delta_p(t)}( \pmb{s^\sigma_p}  )$. For fixed integers $1\leq r \leq p-1$, the $r$-contraction $f\otimes_r g$ of $f,g \in\mathscr{H}^{\otimes p}$ is the element in $\mathscr{H}^{\otimes 2p-2r}$ given by 
 \begin{align*}
& \big(f \otimes_{r}g\big)\big( \pmb{s_{p-r}}, \pmb{\wt{s}_{p-r}}, \pmb{\xi_{p-r}}, \pmb{\wt{\xi}_{p-r}}     \big)= \int_{\R_+^{2r} }  d\pmb{a_r}   d\pmb{\wt{a}_r}  \left( \prod_{i=1}^r \gamma_0(a_i-\wt{a}_i) \right)  \int_{ \R^{2dr} }  d\pmb{x_r}     d\pmb{\wt{x}_r}  \\
 &\qquad\qquad\times \left(  \prod_{i=1}^r \gamma_0(x_i-\wt{x}_i) \right)  f( \pmb{s_{p-r}}, \pmb{a_r},  \pmb{\xi_{p-r}} ,  \pmb{x_r}    )g( \pmb{\wt{s}_{p-r}}, \pmb{\wt{a}_r},  \pmb{\wt{\xi}_{p-r}} ,  \pmb{\wt{x}_r}    ),
\end{align*}
 which may be a generalized function.

 \medskip

Here is the plan  for the proof of Theorems \ref{thmSHE} and \ref{Nchaotic}.  Section \ref{sec310}  deals with computing the limit of the  covariance  function of the process  $A_t(R)$ as $R\rightarrow +\infty$,  provided that  $\gamma_1(\R^d)$ is finite. Section
 \ref{sec311} is devoted to the proof of the {\it convergence of the finite-dimensional distributions}, and we prove the {\it tightness} of $\{R^{-d/2} A_t(R), t\ge 0 \}$ in   Section \ref{312}    under the extra assumption \eqref{ADDc}. As a by-product of the computations  in Section \ref{sec310}, we provide a proof of Theorem \ref{Nchaotic}  in Section \ref{sec310a}.

 \subsection{Limiting covariance structure in Theorem \ref{thmSHE}} \label{sec310}
 The main ingredient  is the following \emph{Feymann-Kac representation}.
 
 \begin{lemma}[Feynman-Kac formula]  \label{FKAC}   Let $\gamma_0, \gamma_1$ be given as in Theorem \ref{thmSHE} and we fix  $t, s>0$. Then for any $x,y\in\R^d$, we have 
\begin{equation*}
 \phi_{t,s}(x-y) := \E\big[ u_{t,x} u_{s,y} \big] = \E\left[  e^{\beta_{t,s}(x-y)}   \right] 
 \end{equation*}
 with 
 \begin{equation*} \beta_{t,s}(z):=  \int_0^{t}\int_0^{s} \gamma_0(u-v) \gamma_1(X_u^1 - X_v^2 + z) dudv, 
 \end{equation*}
 where $X^1, X^2$ are two independent standard Brownian motions on $\R^d$ that start at zero.
  \end{lemma}

We refer to  \cite[Theorem 3.6]{HHNT15} for  the proof of a more general statement. We point out that in  this reference, the moment formula is stated for $x=y$ and $t=s$, see equation (3.21) therein; one can prove the case $x\neq y$ or $t\neq s$ verbatim.

 It follows from Lemma \ref{FKAC} that  
 \begin{align*}
\Sigma_{s,t} :&= \lim_{R+\infty} R^{-d} \E\big[ A_t(R) A_s(R) \big]  =\lim_{R\to+\infty} R^{-d}  \int_{B_R^2}  \big(\phi_{t,s}(x-y) - 1 \big)dxdy \\
&=\lim_{R\to+\infty} R^{-d}  \int_{\R^d} \big(\phi_{t,s}(z) - 1 \big) \text{vol}\big( B_R \cap B_R(-z) \big)dz =\omega_d \int_{\R^d} \big(\phi_{t,s}(z) - 1 \big) dz,
 \end{align*}
 provided the integral $ \int_{\R^d} \big(\phi_{t,s}(z) - 1 \big) dz $ is finite. Note that in our setting, $\phi(z) \geq 1$ for every $z\in\R^d$; note also that, since $\gamma_1$ is integrable,
 \begin{align}  \notag 
\int_{\R^d} \big(\phi_{t,s}(z) - 1 \big) dz  &  \geq \int_{\R^d}  \E[ \beta_{t,s}(z) ]  dz\\
&   = \left(\int_0^t\int_0^s \gamma_0(u-v) dudv \right) \int_{\R^d}\gamma_1(z)dz  \in (0,\infty), \label{DN3}
 \end{align}
 where the equality follows from Fubini's theorem.

   Note that 
 \[
  \int_{\R^d} \big(\phi_{t,s}(z) - 1 \big) dz  = \sum_{p\geq 1} \frac{1}{p!} \int_{\R^d} \E\big[  \beta_{t,s}(z)^p \big]dz,
  \]
   where the object $ \beta_{t,s}(z)$ can be understood as the ``weighted'' intersection local time of two independent Brownian motions $X^1$ and  $X^2$.

   In order to show that  $ \int_{\R^d} \big(\phi_{t,s}(z) - 1 \big) dz < \infty$, we first estimate
   the $p$th moment of $  \beta_{t,s}(z)$.  Without losing any generality, we assume  $s\leq t$.
   Using that $\gamma_1$ is the Fourier transform of  the spectral density $ \varphi_1$,  which is continuous and bounded due to the finiteness of  $\gamma_1(\R^d)$, we can  write 
\begin{align} 
\E\big[ \beta_{s,t}(z)^p \big] & = \int_{[0,s]^{p}\times [0,t]^{p} } \left( \prod_{j=1}^p\gamma_0(s_j - r_j ) \right) \E\left[ \prod_{j=1}^p\gamma_1(X_{s_j}^1 - X_{r_j}^2 + z) \right] d\pmb{s_p} d\pmb{r_p} \notag  \\
  &= \int_{[0,s]^{p}\times [0,t]^{p} }  \int_{\R^{pd}}  d\pmb{\xi_p} d\pmb{s_p} d\pmb{r_p}   \left( \prod_{j=1}^p \gamma_0(s_j - r_j ) \right) \left(  \prod_{j=1}^p  \varphi_1(\xi_j)  \right)   \notag  \\
  &\qquad\qquad \times \E \left(   \prod_{j=1}^p e^{-\i \xi_j  \cdot (X_{s_j}^1 - X_{r_j}^2 + z) } \right)  \notag    \\ 
  &= \int_{[0,s]^{p}\times [0,t]^{p} }  \int_{\R^{pd}}  d\pmb{\xi_p} d\pmb{s_p} d\pmb{r_p}   \left( \prod_{j=1}^p \gamma_0(s_j - r_j ) \right) \left(  \prod_{j=1}^p  \varphi_1(\xi_j)  \right) e^{-\i z \cdot \tau(\pmb{\xi_p})}  \notag \\
  &\qquad\qquad \times \exp\left(- \frac{1}{2}   \sum_{1\leq i,j \leq p} (s_i\wedge s_j + r_i \wedge r_j) \xi_i \cdot \xi_j       \right),   \label{abexp}
 \end{align}
which is a nonnegative, uniformly continuous and uniformly bounded function in $z$. 
Indeed, it is clear that $0 \le \E[\beta_{s,t}(z)^p] \le \E[(\beta_{s,t}(0)^p] <+\infty$ and the uniform continuity follows from the dominated convergence theorem. Then by  the monotone convergence theorem, we write 
\begin{align*}
\int_{\R^d}\E\big[ \beta_{s,t}(z)^p \big] dz = \lim_{\e\downarrow 0} \int_{\R^d}\E\big[ \beta_{s,t}(z)^p \big] \exp\left( -\frac{\e}{2} \| z\|^2 \right) dz\in [0, \infty].  
\end{align*}
Recall from   \eqref{abexp} that the finiteness of   $ \E\big[ \beta_{s,t}(0)^p \big]$ allows us to apply    Fubini's theorem to get for any $\e > 0$, 
\begin{align}
&T_{p,\e}:=\int_{\R^d}\E\big[ \beta_{s,t}(z)^p \big] \exp\left( -\frac{\e}{2} \| z\|^2 \right) dz  \notag \\
&=  (2\pi)^d  \int_{[0,s]^{p}\times [0,t]^{p} }  \int_{\R^{pd}}  d\pmb{\xi_p} d\pmb{s_p} d\pmb{r_p}   \left( \prod_{j=1}^p \gamma_0(s_j - r_j ) \right) \left(  \prod_{j=1}^p  \varphi_1(\xi_j)  \right) \notag  \\ 
  &\qquad\qquad\quad  \times   G\big(\e,  \tau(\pmb{\xi_p})  \big) \exp\left(- \frac{1}{2}   \sum_{1\leq i,j \leq p} (s_i\wedge s_j + r_i \wedge r_j) \xi_i \cdot \xi_j       \right)  \notag ,
\end{align}
which is  finite. 

Consider first the case $p\ge 2$.  Using that $s\le t$
and  
$$\exp\left(- \frac{1}{2}   \sum_{1\leq i,j \leq p}   (r_i \wedge r_j )\xi_i \cdot \xi_j       \right)  \leq 1,$$
we can bound  $T_{p,\e}$ as follows
\begin{align*}
T_{p,\e} &\leq (2\pi)^d \Gamma_t^p  \int_{\R^{pd}} d\pmb{\xi_p} \int_{[0,t]^p } d\pmb{s_p}  \left( \prod_{j=1}^p    \varphi_1(\xi_j)  \right) \\
&\qquad  \times G\big(\e,  \tau(\pmb{\xi_p})  \big) \exp\left(- \frac{1}{2}   \sum_{1\leq i,j \leq p}   (s_i \wedge s_j) \xi_i \cdot \xi_j       \right) ,
\end{align*}
where the constant $\Gamma_t := \int_{-t}^t \gamma_0(u)du$ is finite for each $t > 0$ in view of the local integrability of $\gamma_0$.   Making  the change of variables   $\pmb{\xi_p} = (\eta_1 - \eta_2, \dots, \eta_{p-1} - \eta_p, \eta_p)$, yields, with the convention  $s_{p+1} = 0$ and  $\eta_0=0$,
\begin{align*}
& T_{p,\e} \leq  (2\pi)^d \Gamma_t^p p! \int_{\R^{pd}} d\pmb{\xi_p} \int_{\Delta_p(t)} d\pmb{s_p} e^{- \frac 12\sum_{j=1}^p (s_j - s_{j+1})\| \xi_1 + \dots + \xi_j\|^2 }G (\e,  \tau(\pmb{\xi_p})  )   \prod_{j=1}^p    \varphi_1(\xi_j)      \\
&=  (2\pi)^d \Gamma_t^p p! \int_{\R^d} d\eta_p G(\e, \eta_p)  \int_{\R^{pd-d}} d\pmb{\eta_{p-1}} \int_{\text{SIM}_p(t)} d\pmb{w_p}   \prod_{j=1}^p    \varphi_1(\eta_{j} - \eta_{j-1})   e^{-  \frac 12w_j\| \eta_j\|^2 }  .
\end{align*}
Put 
\[
Q_p(\eta_p) =  \int_{\R^{pd-d}} d\pmb{\eta_{p-1}} \int_{\text{SIM}_p(t)} d\pmb{w_p}  \left( \prod_{j=1}^p    \varphi_1(\eta_{j} - \eta_{j-1})   e^{- \frac 12 w_j\| \eta_j\|^2 }\right),
\]
then we just obtained
\[
T_{p,\e} \leq  (2\pi)^d \Gamma_t^p p! \int_{\R^d} d\eta_p G(\e, \eta_p) Q_p(\eta_p).
\]
In the following, we will  prove that $ Q_p(\eta_p)$ is uniformly bounded and provide an estimate.  We rewrite $Q_p(\eta_p)$ as follows. With $h_j(\eta) = \exp\big(-\frac 12w_j   \|  \eta \|^2 \big)$,
\begin{align*}
&Q_p(\eta_p) =   \int_{\text{SIM}_p(t)}   d\pmb{w_p}  ~ h_p(\eta_p)   \int_{\R^{pd-d}} d\pmb{\eta_{k-1} }~ \varphi_1( \eta_1) h_1(\eta_1) \varphi_1(\eta_2 - \eta_1) h_2(\eta_2)  \\
& \qquad   \times  \varphi_1(  \eta_3 - \eta_2) h_3(\eta_3) \times\dots \times  \varphi_1(\eta_{p-1}-\eta_{p-2}) h_{p-1}(\eta_{p-1})\varphi_1(\eta_p - \eta_{p-1})  \,.
\end{align*}
Using that $\varphi_1$ is bounded, we get
\begin{equation}
\int_{\R^d}  \varphi_1(\eta_1)  \varphi_1(\eta_2 - \eta_1) h_1(\eta_1)  d\eta_1   \leq \| \varphi_1 \|_\infty   \int_{\R^d} \varphi_1(\eta_1) h_1(\eta_1)d\eta_1.  \label{KU2}
\end{equation}
 On the other hand, using \eqref{ecuu1}, we have 
\[
\int_{\R^d} d\eta_j  h_j(\eta_j)  \varphi_1(\eta_{j+1} -\eta_j)  \leq \int_{\R^d} d\eta_j  \varphi_1(\eta_j)   h_j(\eta_j)
\]
for $j=2,\dots, p-1$. So,
         \begin{align}
Q_p(\eta_p) &\leq \| \varphi_1\|_\infty  \int_{  \text{SIM}_p(t) } d\pmb{w_p}  e^{-\frac 12w_p \| \eta_p\|^2}    \prod_{j=1}^{p-1}    \int_{\R^d} e^{-w_j \| \eta_j\|^2} \varphi_1(\eta_j) d\eta_j  \notag\\
  &\leq t \| \varphi_1\|_\infty  \int_{\R^{pd-d}}  \int_{ \text{SIM}_{p-1}(t)  }     \prod_{i=1}^{p-1} \varphi_1(\xi_i)
   e^{-\frac 12 w_i \| \xi_i\|^2}         d\pmb{w_{p-1}}     d\pmb{\xi_{p-1}} \notag \\
&\leq  t \| \varphi_1\|_\infty  \sum_{j=0}^{p-1} {p-1 \choose j} \frac{t^j}{j!} D_N^j (2C_N)^{p-1-j},    \label{fromh}
\end{align}
 where the last inequality follows from Lemma 3.3 in \cite{HHNT15},
 with the notation
 \begin{align} \label{CNN}
C_N = \int_{\{ \| \xi \| \geq N  \}} \frac{\varphi_1(\xi)     }{\| \xi\|^2} d\xi  
\end{align}
and 
\begin{align*}
D_N =  \int_{\{ \| \xi \| \leq N  \}}   \varphi_1(\xi)    d\xi .
\end{align*}
Notice that these quantities  are finite for any $N >0$ by condition \eqref{Da}. We fix  $N$ such that $0< 4\Gamma_t C_N < 1$.  
  This gives us the uniform boundedness of $Q_p$ and moreover,
 \[
T_{p,\e} \leq   (2\pi)^d \Gamma_t^p p! \| Q_p \|_\infty \leq \| \varphi_1\|_\infty   (2\pi)^d \Gamma_t^pp! t (4C_N)^{p-1}\exp\Big( \frac{tD_N}{2C_N} \Big),
 \]
which  immediately implies   
\begin{align}
\int_{\R^d}\E\big[ \beta_{s,t}(z)^p\big] dz  \leq \| \varphi_1\|_\infty   (2\pi)^d \Gamma_t^p p! t (4C_N)^{p-1}\exp\Big( \frac{tD_N}{2C_N} \Big)< \infty \label{stp1}
\end{align}
and  
\begin{align}
\sum_{p\geq 2}\frac{1}{p!} \int_{\R^d}\E\big[ \beta_{s,t}(z)^p \big] dz & \leq \frac{ (2\pi)^d  \| \varphi_1\|_\infty  t}{4C_N}  \exp\Big( \frac{tD_N}{2C_N} \Big)  \sum_{p\geq 2} (4\Gamma_t C_N)^p  \notag\\
&= \frac{  4\| \varphi_1\|_\infty (2\pi)^d  t C_N \Gamma_t^2    }{1- 4\Gamma_t C_N }  \exp\Big( \frac{tD_N}{2C_N} \Big)  \label{stp2}
\end{align}
is finite, since $0< 4\Gamma_t C_N < 1 $.

To show  the integrability of $\phi_{s,t}-1$,  it remains to check that
\begin{align} \label{bddT1}
 \int_{\R^d}\E\big[ \beta_{s,t}(z) \big] dz < \infty,
\end{align}
which follows from \eqref{DN3}.
Therefore,
\[
\int_{\R^d} \big( \phi_{s,t}(z) - 1 \big) dz \leq  t \Gamma_t \| \gamma_1\|_{L^1(\R^d)}  + \frac{  4\| \varphi_1\|_\infty (2\pi)^d  t C_N \Gamma_t^2    }{1- 4\Gamma_t C_N }  \exp\Big( \frac{tD_N}{2C_N} \Big)  < \infty.
\]
 As a consequence, we proved that, for any $s,t\in\R_+$,
\[
\lim_{R\to+ \infty} \frac{ \E[ A_t(R) A_s(R) ]   }{R^d} = \Sigma_{s,t}=  \omega_d \int_{\R^d} \big( \phi_{s,t}(z) -1 \big)dz \in  (0, \infty).
\]

 \subsection{Convergence of the finite-dimensional distributions in Theorem \ref{thmSHE}}\label{sec311}

 Fix $0 < t_1 < \dots < t_n < \infty$ and put
\[
g_{q,R}(t) = R^{-d/2} \int_{B_R} f_{t, x, q}dx\,.
\]
Then $A_R:=R^{-d/2}\big( A_{t_1}(R), \ldots, A_{t_n}(R) \big)$ falls into the framework of the following Proposition \ref{tool},  the multivariate chaotic central limit theorem borrowed  from \cite[Theorem 2.1]{CNN18}.

\begin{proposition} \label{tool}    
 Fix an integer $n\geq 1$ and consider a family $\big\{ A_R, R> 0 \big\}$ of  random vectors in $\R^n$ such that each component of $A_R=(A_{R, 1}, \ldots, A_{R,n})$ belongs to $ L^2(\Omega, \sigma\{W\}, \mathbb{P})$ and has the following chaos expansion
\[
\qquad\qquad\qquad A_{R,j} = \sum_{q\geq 1} I_q^W (  g_{q, j,R}  )\quad \text{with $g_{q, j,R}$    symmetric kernels. }
\]
Suppose the following conditions {\rm (a)-(d)} hold:
\begin{itemize}
\item[(a)]  $\forall i,j\in\{1, \ldots, n\}$ and  $\forall q\geq 1$, $ \E\big[ I_q^W (  g_{q, j,R}  ) I_q^W (  g_{q, i,R}  )\big]\xrightarrow{R\to+\infty}  \sigma_{i,j,q}$.
\medskip
\item[(b)]  $\forall i\in\{1, \ldots, n\}$, ${\displaystyle \sum_{q\geq 1} \sigma_{i,i,q} < \infty}$.
\medskip
\item[(c)]   For any $1\leq r \leq q-1$,     
$
\big\| g_{q, i,R}\otimes_{r} g_{q, i,R} \big\|_{ \mathscr{H}^{\otimes (2q-2r)}} \xrightarrow{R\to+ \infty} 0$.

 \medskip
\item[(d)]  $\forall i\in\{1, \ldots, n\}$, ${\displaystyle \lim_{N\to+\infty} \sup_{R>0} \sum_{q\geq N+1}   \E\big[   I_q^W (  g_{q, i,R}  )^2 \big]  = 0}$.

\end{itemize}
Then $A_R$ converges in law to $N(0, \Sigma)$ as $R\to+\infty$, where $\Sigma = \big( \sigma_{i,j} \big)_{i,j=1}^n$ is given by  $\sigma_{i,j} = \sum_{q\geq 1} \sigma_{i,j,q}$.
\end{proposition}

 \medskip
  \noindent{\bf Proof of conditions {\rm (a), (b)} {\bf and} {\rm  (d):  }}  It suffices to prove  that for any $t,s\in\R_+$ and for any $p\geq 1$, $p!\langle g_{p,R}(t), g_{p,R}(s) \rangle_{\mathscr{H}^{\otimes p}}$ is convergent to some limit, denoted by $\sigma_p(t,s)$ and  for each $t\ge 0$,
 \begin{align}\label{Belval}
 \sum_{p\geq 1} \sigma_p(t,t) < +\infty
 \end{align}
 and
 \begin{equation} \label{DN1}
 \lim_{N\to+\infty} \sup_{R>0} \sum_{q\geq N+1}    p! \| g_{p,R}(t) \|_{\mathscr{H}^{\otimes p}} ^2  = 0.
 \end{equation}
  
  It is well-known in the  literature that the $p$th moment of $ \beta_{t,t}(0)$ coincides with the variance of the $p$th chaotic component of the solution $u_{t,x}$; see for instance \cite{HN09}. 
  Then, it is natural to expect that our verification of condition (a) in Proposition \ref{tool} will resemble the  computations we have done for $\E\big[  \beta_{t,s}(z)^p \big]$. Moreover, we will see that condition (\ref{Belval}) is a consequence of the finiteness of the integral
   $ \int_{\R^d} \big(\phi_{t,s}(z) - 1 \big) dz $ proved in Section \ref{sec310}.
    The verification of condition (\ref{DN1}) will be straightforward, as a by-product of  the computations in 
    Section \ref{sec310}.  
  
  Let us start with the case $p=1$.
By an easy computation,
\begin{align} \notag
  &\langle g_{1,R}(t), g_{1,R}(s) \rangle_{\mathscr{H}}  =
  R^{-d} \int_{B^2_R} \langle G(t-\bullet, x-\bullet), G(s- \bullet,  y-\bullet) \rangle_{\mathscr{H}} dxdy\\
  &  \qquad =
 (2\pi)^d  \omega_d   \int_0^{t} \int_0^{s} dudv \gamma_0(u-v) \int_{\R^d} d\xi ~\ell_R(\xi)  \varphi_1(\xi) e^{-\frac{1}{2}(t -u + s- v )   \| \xi\|^2   },  \label{DN2}
\end{align}
where $\ell_R(\xi)$ is the approximation of the identity introduced in Point (3) of Lemma \ref{Bessel}.
Since $\gamma_1$ is integrable on $\R^d$, $\varphi_1$ is uniformly continuous and uniformly bounded. Then, taking the limit as $R\to+\infty$ in \eqref{DN2}, yields 
\[
 \langle g_{1,R}(t), g_{1,R}(s) \rangle_{\mathscr{H}} \xrightarrow{R\to+\infty}  (2\pi)^d\omega_d \varphi_1(0)  \int_0^{t} \int_0^{s} dudv \gamma_0(u-v) =\sigma_1(t,s).
\]
Notice that   $\sigma_1(t,s)=\omega_d \int_{\R^d} \E\big[\beta_{s,t}(z) \big]dz$, in view of \eqref{DN3} and  $(2\pi)^d \varphi_1(0) =\gamma_1(\R^d)$.
\medskip

  Now let us consider higher-order chaos.   For a fixed   $p\geq 2$, we   write
 \begin{align*}
   \E\Big[ I^W_p\big(g_{p,R}(t)  \big) I^W_p\big(g_{p, R}(s) \big)  \Big] = \frac{p!}{R^d} \int_{B_R^2} dx dy ~ \big\langle f_{t,x,p} ,  f_{s,y,p}  \big\rangle_{\mathscr{H}^{\otimes p}}.   
 \end{align*}
The kernel $f_{t,x,p}$ is a nonnegative  function on $\R_+^p\times\R^{pd}$, so         $ \langle f_{t,x,p} ,  f_{s,y,p}   \rangle_{\mathscr{H}^{\otimes p}}\geq 0$.  We first  write, by using  the Fourier transform in space,
 \begin{align}
&\quad \big\langle f_{t,x,p} ,  f_{s,y,p}  \big\rangle_{\mathscr{H}^{\otimes p}}  \notag  \\
&=  \int_{\R_+^{2p}} d\pmb{s_p} d\pmb{\wt{s}_p}  \prod_{j=1}^p    \gamma_0( s_j-\wt{s}_j) \int_{\R^{pd}} \mu_1(\pmb{\xi_p})   \F f_{t, x, p}(\pmb{s_p}, \pmb{\xi_p})   \F f_{s, y, p}(\pmb{\wt{s}_p}, -\pmb{\xi_p}). \label{DTW}
 \end{align}
       Note that for $\pmb{s^\sigma_p}\in \Delta_p(t)$, by the change of  variables $y_1 = x^\sigma_1 - x$, $y_j = x^\sigma_j - x^\sigma_{j-1}$ for $j\geq 2$, we can write, with $X^1$ standard Brownian motion on $\R^d$ as before,
      \begin{align}
       &\quad{\bf 1}_{\Delta_p(t)}( \pmb{s^\sigma_p}   ) \int_{\R^{dp}} d\pmb{x^\sigma_p} e^{-\i \pmb{x^\sigma_p} \cdot \pmb{\xi^\sigma_p}}  G(t - s^\sigma_1, x - x^\sigma_1)  \prod_{i=1}^{p-1} G(s^\sigma_i - s^\sigma_{i+1}, x^\sigma_i - x^\sigma_{i+1})  \notag \\
       &={\bf 1}_{\Delta_p(t)}( \pmb{s^\sigma_p}   ) e^{-\i x\cdot \tau(\pmb{\xi_p})}  \E\left[  \prod_{j=1}^p \exp\left(  -\i  (X^1_{t}  - X^1_{s^\sigma_j} )\cdot \xi^\sigma_j     \right)  \right]  \notag \\
       &={\bf 1}_{\Delta_p(t)}( \pmb{s^\sigma_p}   )  e^{-\i x\cdot \tau(\pmb{\xi_p})}  \E\left[  \prod_{j=1}^p \exp\left(  -\i  (X^1_{t}  - X^1_{s_j} )\cdot \xi_j     \right)  \right],   \label{FK}
      \end{align}
so that
 \begin{align*}
   \F f_{t, x, p}(\pmb{s_p},  \pmb{\xi_p}) =\frac{1}{p!}  e^{-\i x\cdot \tau(\pmb{\xi_p})}  \E\left[  \prod_{j=1}^p \exp\left(  -\i  (X^1_{t}  - X^1_{s_j} )\cdot \xi_j     \right)  \right],
  \end{align*}
for  $\pmb{s_p}\in[0,t]^p$ and
  \begin{align*}
   \F f_{s, y, p}(\pmb{\wt{s}_p},  -\pmb{\xi_p})  = \frac{1}{p!}  e^{\i y\cdot \tau(\pmb{\xi_p})}  \E\left[  \prod_{j=1}^p \exp\left(  \i  (X^1_{s}  - X^1_{\wt{s}_j} )\cdot \xi_j     \right)  \right]
  \end{align*}
for  $\pmb{\wt{s}_p}\in[0,s]^p$. Keeping in mind the  above expressions and making the time changes in \eqref{DTW} (from $s_j$ to $t - s_j$ and from  $\wt{s}_j$ to $s - \wt{s}_j$, for $j=1,\dots,p$) yields 
 \begin{align}
&\quad \big\langle f_{s,x,p} ,  f_{t,y,p}  \big\rangle_{\mathscr{H}^{\otimes p}}  \notag  \\
&=\frac{1}{(p!)^2}  \int_{[0,s]^p\times  [0,t]^p} d\pmb{s_p} d\pmb{r_p}  \prod_{j=1}^p    \gamma_0(t- s_j- s + r_j) \int_{\R^{pd}} \mu_1(\pmb{\xi_p}) e^{-\i (x-y)\cdot \tau(\pmb{\xi_p})} \notag  \\&\qquad \times \E\left[  \prod_{j=1}^p \exp\left(  -\i   X^1_{s_j} \cdot \xi_j     \right)  \right] \cdot  \E\left[  \prod_{j=1}^p \exp\left(  -\i  X^1_{r_j} \cdot \xi_j     \right)  \right], \label{acexp}
 \end{align}
since $\{X^1_t - X^1_{t-u}, u\in[0,t]\}$ and $\{X^1_s - X^1_{s-u}, u\in[0,s]\}$ have the same law as $\{X^1_u , u\in[0,t]\}$ and $\{X^1_u, u\in[0,s]\}$ respectively.  So the expression \eqref{DTW} is indeed a function that depends only on the difference $x-y$.  Furthermore, a quick comparison between  \eqref{abexp} and  \eqref{acexp} reveals that the only difference is that the variables inside the temporal covariance kernel are $\gamma_0(s_j - r_j)$ in  \eqref{abexp}  and   $\gamma_0(t-s_j - s + r_j)$ in \eqref{acexp}.  Going through the same arguments that lead to \eqref{stp1} and \eqref{stp2}, we get (with $s\leq t$)
\begin{align*}
p!  \int_{\R^d}  \langle f_{t,z,p} ,  f_{s,0,p}   \rangle_{\mathscr{H}^{\otimes p}} dz \leq (2\pi)^d \| \varphi_1\|_\infty  \Gamma_t^p t (4C_N)^{p-1} \exp\Big( \frac{tD_N}{2C_N} \Big)  
\end{align*}
and 
  \begin{align*}
   \E\Big[ I^W_p\big(g_{p,R}(t)  \big) I^W_p\big(g_{p, R}(s) \big)  \Big] & = \frac{p!}{R^d} \int_{B_R^2} dx dy ~ \big\langle f_{t,x,p} ,  f_{s,y,p}  \big\rangle_{\mathscr{H}^{\otimes p}} \\
   & = p! \omega_d \int_{\R^d} dz  \big\langle f_{t,0,p} ,  f_{s,z,p}  \big\rangle_{\mathscr{H}^{\otimes p}} \frac{\text{vol}\big( B_R \cap B_R(-z) \big)}{\omega_d R^d} \\
   &\xrightarrow{R\to+\infty}  p! \omega_d \int_{\R^d} dz  \big\langle f_{t,0,p} ,  f_{s,z,p}  \big\rangle_{\mathscr{H}^{\otimes p}}  = \sigma_{p}(t,s),
\end{align*}
with 
 \begin{align}
 \sup_{R>0} \E\big[ I^W_p (g_{p,R}(t)   ) I^W_p (g_{p, R}(s))  \big] \leq  \sigma_{p}(t,s). \label{UNIB}
 \end{align}
This completes the verification of condition (a).  
Notice that 
\[
\sigma_p(t,t)= \frac{\omega_d}{p!}\int_{\R^d} \E[ \beta_{t,t}(z)^p] dz,
\]
 so condition (b) follows from
\eqref{bddT1} and \eqref{stp2}.
 To see condition (d), it is enough to use \eqref{UNIB} and condition (b).

\medskip
 \noindent{\bf Proof of  condition {\rm(c):}}  Given  $t>0$ and  $1\leq r \leq p-1$, we need to prove  that
 \[
\lim_{R\to+\infty}  \big\| g_{p,R}(t) \otimes_r g_{p,R}(t) \big\| _{\mathscr{H}^{\otimes (2p-2r)}} = 0.
 \]
We follow the same routine that leads to \eqref{kk2}.    We put 
 \[
 \mathfrak{f}(\pmb{s_p}, \pmb{y_p}) = f_{t,0,p}(\pmb{s_p}, \pmb{y_p}),
 \]
 and in this way, we have $f_{t,x,p} = \mathfrak{f}^x$, with $\mathfrak{f}^x$ being the spatially shifted version of $\mathfrak{f}$. Now we write (notice that we have the extra temporal variables now)
 \begin{align*}
&\quad (2\pi)^{-2d} \big\| g_{p,R}(t) \otimes_r g_{p,R}(t) \big\| _{\mathscr{H}^{\otimes (2p-2r)}}^2 \\
& = \int_{[0,t]^{4p}} d\pmb{s_r}  d\pmb{\wt{s}_r}d\pmb{v_r} d\pmb{\wt{v}_r} d\pmb{t_{p-r}} d\pmb{\wt{t}_{p-r}}d\pmb{w_{p-r}} d\pmb{\wt{w}_{p-r}}   \left( \prod_{i=1}^r \gamma_0(s_i-\wt{s}_i)\gamma_0(v_i-\wt{v}_i)\right) \\
&\qquad \times \left( \prod_{j=1}^{p-r} \gamma_0(t_j-\wt{t}_j)\gamma_0(w_j-\wt{w}_j)\right)  \wt{\mathcal{J}}_R,
 \end{align*}
 with $\wt{\mathcal{J}}_R = \wt{\mathcal{J}}_R\big( \pmb{s_r},  \pmb{\wt{s}_r}, \pmb{v_r} , \pmb{\wt{v}_r}, \pmb{t_{p-r}}, \pmb{\wt{t}_{p-r}},\pmb{w_{p-r}} ,\pmb{\wt{w}_{p-r}}  \big)$ given by
 \begin{align*}
 \wt{\mathcal{J}}_R &=  \int_{\R^{2pd}}\mu_1(d\pmb{\xi_r}) \, \mu_1(d\pmb{\wt{\xi}_r})  \mu_1(d\pmb{\eta_{p-r}})  \mu_1(d\pmb{\wt{\eta}_{p-r}})  \\
&\quad \times (\F \mathfrak{f})(\pmb{s_r}, \pmb{t_{p-r}},\pmb{\eta_{p-r}},\pmb{\xi_r})   (\F \mathfrak{f})(\pmb{\wt{s}_r}, \pmb{w_{p-r}}, \pmb{\eta_{p-r}},  \pmb{\wt{\xi}_r})  \| a + b \| ^{-d/2}  \| \wt{b} + a \| ^{-d/2}    \\
&\quad     \times (\F \mathfrak{f}) (\pmb{v_r}, \pmb{\wt{t}_{p-r}}, \pmb{\wt{\eta}_{p-r}},\pmb{\xi_r})   (    \F \mathfrak{f})(\pmb{\wt{v}_r},\pmb{\wt{w}_{p-r}},     \pmb{\wt{\eta}_{p-r}} , \pmb{\wt{\xi}_r})    \| \wt{a} + b   \| ^{-d/2} \| \wt{a}+\wt{b} \| ^{-d/2}    \\
&\quad  \times J_{d/2} \big( R\| a + b \| \big) J_{d/2}\big( R\|\wt{b} + a \| \big)J_{d/2}\big( R\| \wt{a} + b \| \big) J_{d/2}\big( R\|\wt{ a} + \wt{b} \| \big), 
 \end{align*}
  where $\F \mathfrak{f}$ stands for the Fourier transform with respect to the spatial variables  and we have used the short-hand notation 
  \begin{center}
  $a = \tau(\pmb{\xi_r}),  b=\tau(\pmb{\eta_{p-r}}), \wt{a} = \tau(\pmb{\wt{\xi}_r})$ and $ \wt{b}=\tau(\pmb{\wt{\eta}_{p-r}})$.
 \end{center}
 Recall from previous steps that,  with $X^1$ standard Brownian motion on $\R^d$,
 \begin{equation}
 (\F \mathfrak{f})(\pmb{s_p}, \pmb{\xi_p}) = (\F f_{t,0,p})(\pmb{s_p}, \pmb{\xi_p})  \label{expform}
  = \frac{1}{p!}  \E\left[ \exp\Big( -\i \sum_{j=1}^p (X_t^1 - X_{s_j }^1) \cdot \xi_j\Big) \right], 
 \end{equation}
 which is a positive, bounded and uniformly continuous function in $\pmb{\xi_p}$. As in the proof of Theorem \ref{thm2} (Step 4), we decompose  the integral in the spatial variable into two parts, that is, we write  for any given $\delta>0$,
 \[
 \wt{\mathcal{J}}_R =  \wt{\mathcal{J}}_{1,R} +  \wt{\mathcal{J}}_{2,R} := \int_{\R^{2pd}} {\bf 1}_{\{ \| a+ b\| \geq \delta \}} + \int_{\R^{2pd}} {\bf 1}_{\{ \| a+ b\| < \delta \}}.
 \]
 Similar to the arguments in Step 4 of  the proof of Theorem \ref{thm2}, by using Cauchy-Schwarz inequality several times, we can write 
 \begin{align*}
 \wt{\mathcal{J}}_{1,R} &\leq \omega_d^2 \left( \int_{ \{ \| \tau(\pmb{\xi_p}) \geq \delta   \} } \ell_R( \tau(\pmb{\xi_p})  )  \bv\F \mathfrak{f} \bv^2(\pmb{s_r}, \pmb{t_{p-r}},\pmb{\xi_p})   \mu_1(d\pmb{\xi_p})  \right)^{1/2} \\
 &\quad \times \left( \int_{  \R^{pd}} \ell_R( \tau(\pmb{\xi_p})  )  \bv\F \mathfrak{f} \bv^2(\pmb{\wt{v}_r}, \pmb{\wt{w}_{p-r}},\pmb{\xi_p})   \mu_1(d\pmb{\xi_p})  \right)^{1/2} \\
  &\quad \times \left( \int_{  \R^{pd}} \ell_R( \tau(\pmb{\xi_p})  )  \bv\F \mathfrak{f} \bv^2(\pmb{\wt{s}_r}, \pmb{w_{p-r}},\pmb{\xi_p})   \mu_1(d\pmb{\xi_p})  \right)^{1/2} \\
   &\quad \times \left( \int_{  \R^{pd}} \ell_R( \tau(\pmb{\xi_p})  )  \bv\F \mathfrak{f} \bv^2(\pmb{v_r}, \pmb{\wt{t}_{p-r}},\pmb{\xi_p})   \mu_1(d\pmb{\xi_p})  \right)^{1/2}.
 \end{align*}
 Therefore, by Cauchy-Schwarz  inequality again applied to the integration in time, we get
  \begin{align}
&\quad  \int_{[0,t]^{4p}} d\pmb{s_r}  d\pmb{\wt{s}_r}d\pmb{v_r} d\pmb{\wt{v}_r} d\pmb{t_{p-r}} d\pmb{\wt{t}_{p-r}}d\pmb{w_{p-r}} d\pmb{\wt{w}_{p-r}}   \left( \prod_{i=1}^r \gamma_0(s_i-\wt{s}_i)\gamma_0(v_i-\wt{v}_i)\right) \notag \\
&\qquad \times \left( \prod_{j=1}^{p-r} \gamma_0(t_j-\wt{t}_j)\gamma_0(w_j-\wt{w}_j)\right)  \wt{\mathcal{J}}_{1,R}  \label{part-one}\\
&\leq   \omega_d^2\Bigg\{ \int_{[0,t]^{4p}} d\pmb{s_r}  d\pmb{\wt{s}_r}d\pmb{v_r} d\pmb{\wt{v}_r} d\pmb{t_{p-r}} d\pmb{\wt{t}_{p-r}}d\pmb{w_{p-r}} d\pmb{\wt{w}_{p-r}}   \left( \prod_{i=1}^r \gamma_0(s_i-\wt{s}_i)\gamma_0(v_i-\wt{v}_i)\right) \notag \\
&\qquad \times \left( \prod_{j=1}^{p-r} \gamma_0(t_j-\wt{t}_j)\gamma_0(w_j-\wt{w}_j)\right) \left( \int_{  \R^{pd}} \ell_R( \tau(\pmb{\xi_p})  )  \bv\F \mathfrak{f} \bv^2(\pmb{\wt{v}_r}, \pmb{\wt{w}_{p-r}},\pmb{\xi_p})   \mu_1(d\pmb{\xi_p})  \right)  \notag \\
&\qquad \qquad\qquad \times   \int_{ \{ \| \tau(\pmb{\xi_p}) \geq \delta   \} } \ell_R( \tau(\pmb{\xi_p})  )  \bv\F \mathfrak{f} \bv^2(\pmb{s_r}, \pmb{t_{p-r}},\pmb{\xi_p})   \mu_1(d\pmb{\xi_p})  \Bigg\}^{1/2}  \notag\\
&\times  \Bigg\{ \int_{[0,t]^{4p}} d\pmb{s_r}  d\pmb{\wt{s}_r}d\pmb{v_r} d\pmb{\wt{v}_r} d\pmb{t_{p-r}} d\pmb{\wt{t}_{p-r}}d\pmb{w_{p-r}} d\pmb{\wt{w}_{p-r}}   \left( \prod_{i=1}^r \gamma_0(s_i-\wt{s}_i)\gamma_0(v_i-\wt{v}_i)\right) \notag \\
&\qquad \times \left( \prod_{j=1}^{p-r} \gamma_0(t_j-\wt{t}_j)\gamma_0(w_j-\wt{w}_j)\right) \left( \int_{  \R^{pd}} \ell_R( \tau(\pmb{\xi_p})  )  \bv\F \mathfrak{f} \bv^2(\pmb{\wt{s}_r}, \pmb{w_{p-r}},\pmb{\xi_p})   \mu_1(d\pmb{\xi_p})  \right)  \notag \\
&\qquad \qquad\qquad \times   \int_{ \R^{pd}} \ell_R( \tau(\pmb{\xi_p})  )  \bv\F \mathfrak{f} \bv^2(\pmb{v_r}, \pmb{\wt{t}_{p-r}},\pmb{\xi_p})   \mu_1(d\pmb{\xi_p})  \Bigg\}^{1/2} \notag \\
& =:  \omega_d^2 V_1 ^{1/2}     V_2  ^{1/2}.  \notag
 \end{align}
 We will prove that  $V_1\to 0$   as $R\to +\infty$ and $V_2$ is uniformly bounded. For the term $V_1$,
 we have the estimate
  \begin{align*}
 V_1 &\leq   \Gamma_t^{2p} \left[ \int_{[0,t]^{p}} d\pmb{t_p}  \int_{  \R^{pd}} \ell_R( \tau(\pmb{\xi_p})  )  \bv\F \mathfrak{f} \bv^2(\pmb{t_p},\pmb{\xi_p})   \mu_1(d\pmb{\xi_p})  \right]\\
 &\qquad\qquad\qquad  \times \int_{[0,t]^p} d\pmb{s_p}   \int_{ \{ \| \tau(\pmb{\xi_p}) \geq \delta   \} } \ell_R( \tau(\pmb{\xi_p})  )  \bv\F \mathfrak{f} \bv^2(\pmb{s_p},\pmb{\xi_p})   \mu_1(d\pmb{\xi_p})\\
 &=:   \Gamma_t^{2p} V_{11}  V_{12}.
 \end{align*}
   We claim  that $V_{11}$ is uniformly bounded and $ V_{12}$ vanishes asymptotically as $R\rightarrow +\infty$. 
  In view of   \eqref{expform},  making the   change of variables $t_j = t - s_j$ and $\eta_j =  \xi_1 +\dots + \xi_j $ for each $j=1,\dots, p$, with $\eta_0=0$, we obtain, using \eqref{fromh}
  \begin{align*}
V_{11}& =\frac{1}{(p!)^2} \int_{[0,t]^{p}} d\pmb{s_p}  \int_{  \R^{pd}} \mu_1(d\pmb{\xi_p})\ell_R( \tau(\pmb{\xi_p})  ) \left( \E\left[ \exp\Big( -\i \sum_{j=1}^p  X_{s_j }^1 \cdot \xi_j\Big) \right] \right)^2\\
&=\frac{1}{p!} \int_{\Delta_p(t)}d\pmb{s_p}  \int_{  \R^{pd}} \mu_1(d\pmb{\xi_p})\ell_R( \tau(\pmb{\xi_p})  ) \exp\left( -\sum_{j=1}^p (s_j -s_{j+1}) \| \xi_1 + \dots+ \xi_j \|^2  \right) \\
&=\frac{1}{p!} \int_{\R^d} d\eta_p \ell_R(\eta_p)  \int_{\R^{pd-d}}d\pmb{\eta_{p-1}} \int_{\text{SIM}_p(t)}d\pmb{w_p}    \prod_{j=1}^p e^{- w_j \| \eta_j \|^2  } \varphi_1(\eta_j - \eta_{j-1}) \\
&\leq \frac{t}{p!} \| \varphi_1\|_\infty \sum_{j=0}^{p-1} {p-1\choose j} \frac{t^j}{j!} D_N^j C_N^{p-1-j} < +\infty.
\end{align*}
    In the same way, we have 
\[
V_{12} \le \left( \int_{\{ \| \tau_1\| \geq \delta \}  } d\tau_1 \ell_R(\tau_1)  \right)\frac{t  \| \varphi_1\|_\infty}{p!} \sum_{j=0}^{p-1} {p-1\choose j} \frac{t^j}{j!} D_N^j C_N^{p-1-j},
\]
 which converges to zero as $R$ tends to infinity.  By the same arguments, we can get the uniform boundedness of $V_2$ as   $R$ tends to infinity. Thus,  the term \eqref{part-one} does not contribute to the limit of $\big\| g_{p,R}(t) \otimes_r g_{p,R}(t) \big\| _{\mathscr{H}^{\otimes  (2p-2r)}}^2$ as $R\to +\infty$.

\medskip

 Now let us look at the second term and we need to prove that
  \begin{align}
&\quad \mathfrak{X}_R := \int_{[0,t]^{4p}} d\pmb{s_r}  d\pmb{\wt{s}_r}d\pmb{v_r} d\pmb{\wt{v}_r} d\pmb{t_{p-r}} d\pmb{\wt{t}_{p-r}}d\pmb{w_{p-r}} d\pmb{\wt{w}_{p-r}}   \left( \prod_{i=1}^r \gamma_0(s_i-\wt{s}_i)\gamma_0(v_i-\wt{v}_i)\right) \notag \\
&\qquad \times \left( \prod_{j=1}^{p-r} \gamma_0(t_j-\wt{t}_j)\gamma_0(w_j-\wt{w}_j)\right)  \wt{J}_{2,R} \xrightarrow{R\to+\infty} 0.  \notag 
 \end{align}
 We can first rewrite $\omega_d^{-2} \wt{J}_{2,R}$ as we did for $\int_{\R^{pd}\times D_\delta^c }$ in the proof of Theorem \ref{thm2}. In fact, using Cauchy-Schwarz multiple times, we  obtain
 \begin{align*}
&   \omega_d^{-2} \wt{J}_{2,R} \leq  \int_{\{  \| a+b\| < \delta  \}}  \mu_1(d\pmb{\xi_r})  \mu_1(d\pmb{\eta_{p-r}}) \sqrt{ \ell_R(a+b) } ~ \F\mathfrak{f} \big(  \pmb{s_r}, \pmb{t_{p-r}},   \pmb{ \eta_{p-r}}, \pmb{\xi_r} \big) \\
&  \times \left(   \int_{\R^{pd}} \mu_1(d\pmb{\wt{\xi}_r})  \mu_1(d\pmb{\wt{\eta}_{p-r}})  \ell_R(\wt{a}+\wt{b}) \bv \F\mathfrak{f} \bv^2\big(  \pmb{\wt{v}_r}, \pmb{\wt{w}_{p-r}},   \pmb{ \wt{\eta}_{p-r}}, \pmb{\wt{\xi}_r} \big)   \right)^{1/2}   \Bigg\{   \int_{\R^{pd}} \mu_1(d\pmb{\wt{\xi}_r})  \\
&\times  \mu_1(d\pmb{\wt{\eta}_{p-r}}) \ell_R(\wt{a}+ b) \ell_R(a+\wt{b}) \bv\F \mathfrak{f} \bv^2\big(  \pmb{v_r}, \pmb{\wt{t}_{p-r}},   \pmb{ \wt{\eta}_{p-r}}, \pmb{\xi_r} \big)  \bv\F \mathfrak{f}\bv^2\big(  \pmb{\wt{s}_r}, \pmb{w_{p-r}}, \pmb{\eta_{p-r}}, \pmb{\wt{\xi}_r}\big)   \Bigg\}^{1/2} \\
&  \leq \Bigg[  \left(   \int_{\R^{pd}} \mu_1(d\pmb{\wt{\xi}_r})  \mu_1(d\pmb{\wt{\eta}_{p-r}})  \ell_R(\wt{a}+\wt{b}) \bv \F\mathfrak{f} \bv^2\big(  \pmb{\wt{v}_r}, \pmb{\wt{w}_{p-r}},   \pmb{ \wt{\eta}_{p-r}}, \pmb{\wt{\xi}_r} \big)   \right) \\
&\qquad \times \left( \int_{\{  \| a+b\| < \delta  \}}  \mu_1(d\pmb{\xi_r})  \mu_1(d\pmb{\eta_{p-r}})  \ell_R(a+b)  \bv \F\mathfrak{f}\bv^2 \big(  \pmb{s_r}, \pmb{t_{p-r}},   \pmb{ \eta_{p-r}}, \pmb{\xi_r} \big) \right)\Bigg]^{1/2} \\
&\qquad\times \Bigg[  \int_{ \{ \| a+b\| < \delta\} \times\R^{pd}}  \mu_1(d\pmb{\xi_r})  \mu_1(d\pmb{\eta_{p-r}})   \mu_1(d\pmb{\wt{\xi}_r})  \mu_1(d\pmb{\wt{\eta}_{p-r}})   \\
&\qquad \times   \bv \F\mathfrak{f}\bv^2 \big(  \pmb{\wt{s}_r}, \pmb{w_{p-r}},   \pmb{ \eta_{p-r}}, \pmb{\wt{\xi}_r} \big)   \bv\F \mathfrak{f}\bv^2 \big(  \pmb{v_r}, \pmb{\wt{t}_{p-r}},   \pmb{\wt{\eta}_{p-r}}, \pmb{\xi_r} \big)     \ell_R(\wt{a}+ b) \ell_R(a+\wt{b})        \Bigg]^{1/2}\\
&: = \wt{V}_1 ^{1/2}  \wt{V}_2 ^{1/2}.
  \end{align*}
 Therefore, 
 \[ 
    \omega_d^{-2} \mathfrak{X}_R \leq  \sqrt{  \mathfrak{X}_{1,R}  \mathfrak{X}_{2,R}  },
 \]
 where 
 \begin{align*}
 & \mathfrak{X}_{1,R} := \int_{[0,t]^{4p}} d\pmb{s_r}  d\pmb{\wt{s}_r}d\pmb{v_r} d\pmb{\wt{v}_r} d\pmb{t_{p-r}} d\pmb{\wt{t}_{p-r}}d\pmb{w_{p-r}} d\pmb{\wt{w}_{p-r}}   \left( \prod_{i=1}^r \gamma_0(s_i-\wt{s}_i)\gamma_0(v_i-\wt{v}_i)\right)\\
 &\qquad \times \left( \prod_{j=1}^{p-r} \gamma_0(t_j-\wt{t}_j)\gamma_0(w_j-\wt{w}_j)\right)  \wt{V}_1 
 \end{align*}
 is uniformly bounded over  $R>0$, as one can verify by the same arguments as before, and
 \begin{align*}
&  \mathfrak{X}_{2,R} :=  \int_{[0,t]^{4p}} d\pmb{s_r}  d\pmb{\wt{s}_r}d\pmb{v_r} d\pmb{\wt{v}_r} d\pmb{t_{p-r}} d\pmb{\wt{t}_{p-r}}d\pmb{w_{p-r}} d\pmb{\wt{w}_{p-r}}   \left( \prod_{i=1}^r \gamma_0(s_i-\wt{s}_i)\gamma_0(v_i-\wt{v}_i)\right)\\
 & \times \left( \prod_{j=1}^{p-r} \gamma_0(t_j-\wt{t}_j)\gamma_0(w_j-\wt{w}_j)\right)   \int_{ \{ \| a+b\| < \delta\} \times\R^{pd}}  \mu_1(d\pmb{\xi_r})  \mu_1(d\pmb{\eta_{p-r}})   \mu_1(d\pmb{\wt{\xi}_r})  \mu_1(d\pmb{\wt{\eta}_{p-r}})   \\
&\qquad \times   \bv\F \mathfrak{f}\bv^2 \big(  \pmb{\wt{s}_r}, \pmb{w_{p-r}},   \pmb{ \eta_{p-r}}, \pmb{\wt{\xi}_r} \big)   \bv\F \mathfrak{f}\bv^2 \big(  \pmb{v_r}, \pmb{\wt{t}_{p-r}},   \pmb{\wt{\eta}_{p-r}}, \pmb{\xi_r} \big)     \ell_R(\wt{a}+ b) \ell_R(a+\wt{b})   \\
&\leq \Gamma_t^{2p}\int_{[0,t]^{2p}} d\pmb{\wt{s_r} }d\pmb{\wt{t}_{p-r} } d\pmb{v_r}d\pmb{w_{p-r}}  \int_{ \{ \| a+b\| < \delta\} \times\R^{pd}}  \mu_1(d\pmb{\xi_r})  \mu_1(d\pmb{\eta_{p-r}})   \mu_1(d\pmb{\wt{\xi}_r}) \mu_1(d\pmb{\wt{\eta}_{p-r}})     \\
&\qquad \times  \bv\F \mathfrak{f}\bv^2 \big(  \pmb{\wt{s}_r}, \pmb{w_{p-r}},   \pmb{ \eta_{p-r}}, \pmb{\wt{\xi}_r} \big)   \bv\F \mathfrak{f}\bv^2 \big(  \pmb{v_r}, \pmb{\wt{t}_{p-r}},   \pmb{\wt{\eta}_{p-r}}, \pmb{\xi_r} \big)     \ell_R(\wt{a}+ b) \ell_R(a+\wt{b}) \\
&=\Gamma_t^{2p}   \int_{ \R^{2pd}}  \mu_1(d\pmb{\xi_p})  \mu_1(d\pmb{\wt{\xi}_{p}})  {\bf 1}_{ \{  \| \xi_1 + \dots+ \xi_r  + \wt{\xi}_{r+1} + \dots+ \wt{\xi}_{p}   \| < \delta\} }  \ell_R\big( \tau(\pmb{\xi_p}) \big)    \ell_R\big( \tau(\pmb{\wt{\xi}_p}) \big)   \\  
&\qquad  \times  \left( \int_{[0,t]^p} d\pmb{s_p}    \bv\F \mathfrak{f}\bv^2 \big(  \pmb{s_p},    \pmb{\wt{\xi}_p} \big) \right) \left(  \int_{[0,t]^p}  d\pmb{t_{p}}     \bv\F \mathfrak{f}\bv^2 \big(  \pmb{t_p},   \pmb{\xi_p} \big) \right)    . 
 \end{align*}
Using  (\ref{expform}) and a change of variable in time,  we can rewrite the last  expression  as follows 
\begin{align} 
\mathfrak{X}_{2,R} & \le  \frac{\Gamma_t^{2p}}{(p!)^2}   \int_{ \R^{2pd}}  \mu_1(d\pmb{\xi_p})  \mu_1(d\pmb{\wt{\xi}_{p}})  {\bf 1}_{ \{  \| \xi_1 + \dots+ \xi_r  + \wt{\xi}_{r+1} + \dots+ \wt{\xi}_{p}   \| < \delta\} } \ell_R\big( \tau(\pmb{\xi_p}) \big)    \ell_R\big( \tau(\pmb{\wt{\xi}_p}) \big) \notag  \\  
&  \qquad \times    \int_{[0,t]^{2p}} d\pmb{s_p}   d\pmb{t_p}   \E\left[ \exp\Big( -\i \sum_{j=1}^p   X_{s_j }^1 \cdot \wt{\xi}_j\Big) \right]    \E\left[ \exp\Big( -\i \sum_{j=1}^p   X_{t_j }^2 \cdot \xi_j\Big) \right]. \notag
 \end{align}
For $\pmb{s_p}\in\Delta_p(t)$, we write 
\[
 \E\left[ \exp\Big( -\i \sum_{j=1}^p   X_{s_j }^1 \cdot \wt{\xi}_j\Big) \right] = \exp\left(  - \sum_{j=1}^p \frac{s_{\sigma(j)} - s_{\sigma(j+1)}}{2} \| \wt{\xi}_{\sigma(1)} + \dots+ \wt{\xi}_{\sigma(j)} \|^2 \right).
\]
Then
\begin{align*}
 &\int_{[0,t]^{p}} d\pmb{s_p}      \E\left[ \exp\Big( -\i \sum_{j=1}^p   X_{s_j }^1 \cdot \wt{\xi}_j\Big) \right]   \\
&= \sum_{\sigma\in\mathfrak{S}_p} \int_{\text{SIM}_p(t/2)} d\pmb{\wt{w}_p} \exp\left(  - \sum_{j=1}^p \wt{w}_j \| \wt{\xi}_{\sigma(1)} + \dots+ \wt{\xi}_{\sigma(j)} \|^2 \right)
\end{align*}
and in the same way,
\begin{align*}
 &\int_{[0,t]^{p}} d\pmb{t_p}      \E\left[ \exp\Big( -\i \sum_{j=1}^p   X_{t_j }^2 \cdot \xi_j\Big) \right]   \\
&= \sum_{\pi\in\mathfrak{S}_p} \int_{\text{SIM}_p(t/2)} d\pmb{w_p} \exp\left(  - \sum_{j=1}^p w_j \| \xi_{\pi(1)} + \dots+\xi_{\pi(j)} \|^2 \right).
\end{align*}
By a further change of variables $ \xi_{\pi(1)} + \dots +\xi_{\pi(j)}  = \eta_j$ and $\wt{\xi}_{\sigma(1)} + \dots + \wt{\xi}_{\sigma(j)} = \wt{\eta}_j$ for given $\sigma, \pi$, we can write 
\[
 {\bf 1}_{ \{  \| \xi_1 + \dots+ \xi_r  + \wt{\xi}_{r+1} + \dots+ \wt{\xi}_{p}   \| < \delta\} } =  {\bf 1}_{ \{  \|  L(\pmb{\eta_p},\pmb{\wt{\eta}_p} )  \| < \delta\} },
\]
where  $L(\pmb{\eta_p},\pmb{\wt{\eta}_p} )$ stands for linear combinations of $\eta_1, \dots, \eta_p, \wt{\eta}_1, \dots \wt{\eta}_p$ that depend on $\sigma,\pi$. With this notation, we have 
\begin{align*}
\mathfrak{X}_{2,R} & \le  \frac{\Gamma_t^{2p}}{(p!)^2}  \sum_{\sigma, \pi\in\mathfrak{S}_p} \int_{\R^{2d}} d\eta_p d\wt{\eta}_p \ell_R(\eta_p)\ell_R(\wt{\eta}_p) \int_{\text{SIM}_p(t/2)^2} d\pmb{w_p} d\pmb{\wt{w}_p} \int_{\R^{2pd-2d}} d\pmb{\eta_{p-1}} d\pmb{\wt{\eta}_{p-1}} \\
&  \qquad  \times \left(\prod_{j=1}^{p-1} \varphi_1(\eta_j - \eta_{j-1}) e^{-w_j \| \eta_j\|^2} \varphi_1(\wt{\eta}_j - \wt{\eta}_{j-1}) e^{-w_j \| \wt{\eta}_j\|^2} \right) \\
&\qquad   \times \varphi_1(\eta_p - \eta_{p-1})  \varphi_1(\eta_p - \eta_{p-1})  e^{-w_p \| \eta_p\|^2 - \wt{w}_p \| \wt{\eta}_p \|^2} {\bf 1}_{ \{  \|  L(\pmb{\eta_p},\pmb{\wt{\eta}_p} )  \| < \delta\} } \\
& =:   \frac{\Gamma_t^{2p}}{(p!)^2}   \sum_{\sigma, \pi\in\mathfrak{S}_p} \int_{\R^{2d}} d\eta_p d\wt{\eta}_p \ell_R(\eta_p)\ell_R(\wt{\eta}_p)  \mathcal{E}^{\sigma,\pi}_\delta(\eta_p, \wt{\eta}_p)
\end{align*}
where $ \mathcal{E}^{\sigma,\pi}_\delta$ is defined in an obvious way.
By the arguments leading to   \eqref{fromh}, it is clear that
 $ \mathcal{E}^{\sigma,\pi}_\delta$ is uniformly bounded. It follows that
 \begin{align*}
& \limsup_{R\to+\infty}\int_{\R^{2d}} d\eta_p d\wt{\eta}_p \ell_R(\eta_p)\ell_R(\wt{\eta}_p)  \mathcal{E}^{\sigma,\pi}_\delta(\eta_p, \wt{\eta}_p) \\
&=  \limsup_{R\to+\infty} \int_{\R^{2d}} d\eta_p d\wt{\eta}_p \ell_R(\eta_p)\ell_R(\wt{\eta}_p)  \mathcal{E}^{\sigma,\pi}_\delta(\eta_p, \wt{\eta}_p) {\bf 1}_{\{ \| \eta_p \| < \delta,  \|\wt{\eta}_p \| < \delta  \}}.
 \end{align*}
 For fixed $\sigma,\pi\in\mathfrak{S}_p$, we have the decomposition $ L(\pmb{\eta_p},\pmb{\wt{\eta}_p} ) = L_1(\eta_p, \wt{\eta}_p)  +  L_2(\pmb{\eta_{p-1}},\pmb{\wt{\eta}_{p-1}} )$, where $L_1(\eta_p, \wt{\eta}_p)$ stands for a linear combination of $\eta_p$ and $\wt{\eta}_p$, while $L_2(\pmb{\eta_{p-1}},\pmb{\wt{\eta}_{p-1}} )$   stands for linear combinations of $\eta_1,\dots, \eta_{p-1}, \wt{\eta}_1, \dots, \wt{\eta}_{p-1}$. Notice that $L_1$ and $L_2$ also depend on $\sigma,\pi$. If $\|\eta_p\| , \| \wt{\eta}_p\| < \delta $, then there exists some constant $K = K(\sigma, \pi)$ such that 
 \[
 \|  L_1(\eta_p, \wt{\eta}_p) \| < K\delta,
  \]
 thus $ {\bf 1}_{ \{  \|  L(\pmb{\eta_p},\pmb{\wt{\eta}_p} )  \| < \delta\} } \leq  {\bf 1}_{ \{  \|  L_2(\pmb{\eta_{p-1}},\pmb{\wt{\eta}_{p-1}} ) \| < (K+1)\delta\} }.$
 As a consequence,
 \begin{align*}
& \int_{\R^{2d}} d\eta_p d\wt{\eta}_p \ell_R(\eta_p)\ell_R(\wt{\eta}_p)  \mathcal{E}^{\sigma,\pi}_\delta(\eta_p, \wt{\eta}_p) {\bf 1}_{\{ \| \eta_p \| < \delta,  \|\wt{\eta}_p \| < \delta  \}} \\
&\leq  t^2 \| \varphi_1\|^2_\infty \int_{\R^{2d}} d\eta_p d\wt{\eta}_p \ell_R(\eta_p)\ell_R(\wt{\eta}_p) \int_{\text{SIM}_{p-1}(t)^2} d\pmb{w_{p-1}} d\pmb{\wt{w}_{p-1}} \int_{\R^{2pd-2d}} d\pmb{\eta_{p-1}} d\pmb{\wt{\eta}_{p-1}} \\
&    \quad \times \left(\prod_{j=1}^{p-1} \varphi_1(\eta_j - \eta_{j-1}) e^{-w_j \| \eta_j\|^2} \varphi_1(\wt{\eta}_j - \wt{\eta}_{j-1}) e^{-w_j \| \wt{\eta}_j\|^2} \right) {\bf 1}_{ \{  \|  L_2(\pmb{\eta_{p-1}},\pmb{\wt{\eta}_{p-1}} ) \| < (K+1)\delta\} } \\
& = t^2 \| \varphi_1\|^2_\infty \int_{\text{SIM}_{p-1}(t)^2} d\pmb{w_{p-1}} d\pmb{\wt{w}_{p-1}} \int_{\R^{2pd-2d}} d\pmb{\eta_{p-1}} d\pmb{\wt{\eta}_{p-1}} \\
&    \quad \times \left(\prod_{j=1}^{p-1} \varphi_1(\eta_j - \eta_{j-1}) e^{-w_j \| \eta_j\|^2} \varphi_1(\wt{\eta}_j - \wt{\eta}_{j-1}) e^{-w_j \| \wt{\eta}_j\|^2} \right) {\bf 1}_{ \{  \|  L_2(\pmb{\eta_{p-1}},\pmb{\wt{\eta}_{p-1}} ) \| < (K+1)\delta\} } \\
&  =: t^2 \| \varphi_1\|^2_\infty  T_\delta(\sigma,\pi).
 \end{align*}
By previous arguments, 
\begin{align*}
 &\int_{\text{SIM}_{p-1}(t)^2} d\pmb{w_{p-1}} d\pmb{\wt{w}_{p-1}} \int_{\R^{2pd-2d}} d\pmb{\eta_{p-1}} d\pmb{\wt{\eta}_{p-1}} \\
&    \quad \times \left( \prod_{j=1}^{p-1} \varphi_1(\eta_j - \eta_{j-1}) e^{-w_j \| \eta_j\|^2} \varphi_1(\wt{\eta}_j - \wt{\eta}_{j-1}) e^{-w_j \| \wt{\eta}_j\|^2} \right) < \infty.
\end{align*}
 Therefore, taking into account that  $L_2(\pmb{\eta_{p-1}},\pmb{\wt{\eta}_{p-1}} )\neq 0$ for almost every $\pmb{\eta_{p-1}}$ and $\pmb{\wt{\eta}_{p-1}}$,  we obtain $T_\delta(\sigma,\pi)\to 0$, as $\delta\downarrow 0$ and
 \[
 \limsup_{R\to+\infty} \mathfrak{X}_{2,R}  \leq  t^2 \| \varphi_1\|^2_\infty   \sum_{\sigma, \pi\in\mathfrak{S}_p}T_\delta(\sigma,\pi),
 \]
 which converges to zero,  as $\delta\downarrow 0$.  This concludes the proof of condition (c).

\medskip
  Combing the above steps, we conclude that if $t_1, t_2,  \dots, t_n\in\R_+$, then 
  \[
  R^{-d/2} \big( A_{t_1}(R), \dots, A_{t_n}(R) \big) \xrightarrow[R\to+\infty]{\rm law} N\Big(0,  \big(\Sigma_{t_i, t_j}\big)_{i,j=1}^n  \Big),
  \]
 where $\Sigma_{t_i, t_j}$ is defined  in \eqref{SIGst}.

\subsection{Proof of tightness in Theorem \ref{thmSHE}}\label{312} 
 In this section, we are going to prove the tightness of $\big\{ \frac{A_t(R)}{R^{d/2} }, t\ge 0 \big\}$ under the extra condition \eqref{ADDc}. Under this condition, one can see easily that 
\begin{align}
\Gamma_{t,\alpha}  := \int_0^{t}\int_0^{t} \gamma_0(r-v) r^{-\alpha} v^{-\alpha} drdv < +\infty \label{GAMMASA}
 \end{align}
 for any $t>0$.
 
 Recall that $\alpha\in(0,1/2)$ is fixed. For    any  $T>0$, we will   show for any $0< s< t \leq T$ and any  integer $k\in[2,\infty)$ 
  \begin{align}\label{G0}
R^{-d/2}\big\|  A_t(R) - A_s(R)\big\|_{L^k(\Omega)}\leq C  \vert t-s\vert^{\alpha},
  \end{align}
  where $C=C_{T,k,\alpha}$ is a constant that depends on $T, k$ and $\alpha$.  If we   pick a large $k$ such that $k\alpha > 2$, we get   the desired tightness by Kolmogorov's criterion.   To show \eqref{G0}, we first derive  the  Wiener chaos expansion of $A_t(R)- A_s(R)$ and apply  the hypercontractivity property of the Ornstein-Uhlenbeck semigroup  (see \emph{e.g.} \cite{Nualart06})  that allows us to
estimate the $L^k(\Omega)$-norm by the $L^2(\Omega)$-norm on a fixed Wiener chaos.

We know that
\[
u_{t,x} = 1 + \int_{\R_+\times\R^d} G(t-s_1, x- y_1){\bf 1}_{[0,t)}(s_1) u_{s_1, y_1} W(ds_1, dy_1)
\]
and if we put 
$$
d(s,t,x;s_1, y_1)=G(t-s_1, x- y_1){\bf 1}_{[0,t)}(s_1) -  G(s-s_1, x- y_1){\bf 1}_{[0,s)}(s_1)
$$
for $s<t$,  we can write 
\[
u_{t,x} - u_{s,x} =\int_{\R_+\times\R^d}  d(s,t,x;s_1, y_1)  u_{s_1, y_1} W(ds_1, dy_1).
\]
We can write $d(s,t,x;s_1, y_1)  = d_1(s,t,x;s_1, y_1) + d_2(s,t,x;s_1, y_1) $ with
\begin{align}\label{DD1}
d_1(s,t,x;s_1, y_1)  =  {\bf 1}_{[0,s)}(s_1) \big[ G(t-s_1, x- y_1)  -  G(s-s_1, x- y_1) \big]   
\end{align}
and
\begin{align}\label{DD2}
 d_2(s,t,x;s_1, y_1) =  {\bf 1}_{[s,t)}(s_1) G(t-s_1, x- y_1).
\end{align}
According to    \cite[Lemma 3.1]{CH}, there exists some constant $C_\alpha$ that depends on $\alpha$ such that
\begin{align}\label{CH31}
\big\vert d_1(s,t,x;s_1, y_1)  \big\vert \leq C_\alpha  (t-s)^{\alpha}  (s-s_1)^{-\alpha} G(4t-4s_1, x- y_1)  {\bf 1}_{[0,s)}(s_1).
\end{align}

Now we can express $A_t(R)-A_s(R)$ as a sum of two chaos expansions that correspond to $d_1$ and $d_2$:
\begin{align*}
A_t(R)-A_s(R)   &=  \sum_{p\ge 1} \int_{B_R}  I_p^W\big( \mathfrak{g}_{1,p,x} \big) dx   +  \sum_{q\geq 1} \int_{B_R} I_q^W\big( \mathfrak{g}_{2,q,x} dx \big)\\
  &=: \sum_{p\ge 1} J_{1,p,R}   +  \sum_{q\geq 1} J_{2,q,R},
\end{align*}
 where $J_{i,p,R} =\int_{B_R}  I_p^W\big( \mathfrak{g}_{i,p,x} \big) dx  $ for $i\in\{1,2\}$ and 
\begin{align*}
\mathfrak{g}_{1,p,x}(\pmb{s_p},\pmb{y_p}) &= \frac{1 }{p!} \sum_{\sigma\in\mathfrak{S}_p} {\bf 1}_{\Delta_p(s)}(\pmb{s^\sigma_p})  d_1(s,t,x;s^\sigma_1, y^\sigma_1)  \prod_{j=1}^{p-1}  G(s^\sigma_j -s^\sigma_{j+1}, y^\sigma_j- y^\sigma_{j+1})\\
\mathfrak{g}_{2,p,x}(\pmb{s_p},\pmb{y_p})  &= \frac{1}{p!} \sum_{\sigma\in\mathfrak{S}_p} {\bf 1}_{\Delta_p(s,t)}(\pmb{s^\sigma_p})  G(t-s^\sigma_1, x-y^\sigma_1) \prod_{j=1}^{p-1}  G(s^\sigma_j -s^\sigma_{j+1}, y^\sigma_j- y^\sigma_{j+1}),
\end{align*}
 with $\Delta_p(s,t) = \{ t > s_1 > \dots > s_p > s\}$.  

 \medskip

 Let us first estimate the $L^2(\Omega)$-norm of $J_{2,p,R}$ in several familiar steps.  As in \eqref{DTW}, \eqref{FK} and \eqref{acexp}, we write for $p\geq 1$, with $X^1, X^2$ independent  standard Brownian motions on $\R^d$, 
 \begin{align*}
 \big\langle    \mathfrak{g}_{2,p,x},  \mathfrak{g}_{2,p,y}  \big\rangle_{\mathscr{H}^{\otimes p}} &= \frac{1}{(p!)^2} \int_{[0,t-s)^{2p}} d\pmb{s_p}d\pmb{r_p} \prod_{j=1}^p \gamma_0(s_j - r_j) \int_{\R^{pd}} \mu_1(d\pmb{\xi_p})  e^{-\i (x-y) \cdot \tau(\pmb{\xi_p})} \\
 &\qquad \times \E\left[  \exp\left(  -\i \sum_{j=1}^p \xi_j  \cdot  X^1_{s_j}  \right) \right] \E\left[  \exp\left(  -\i \sum_{j=1}^p \xi_j  \cdot   X^2_{r_j}   \right) \right],
 \end{align*}
 which is a nonnegative function in $x,y$ that only depends on the difference $x-y$. Observe that this inner product coincides with $\frac{1}{(p!)^2} \E\big[ \beta_{t-s, t-s}(x-y)^p \big]$ for every $p\geq 1$, see \eqref{abexp}.  Therefore,   for $p\geq 2 $, we can write by using \eqref{stp1}
 \begin{align*}
 \big\| J_{2,p,R} \big\|^2_{L^2(\Omega)} &=  p! \int_{B_R^2}dxdy \big\langle    \mathfrak{g}_{2,p,x},  \mathfrak{g}_{2,p,y}  \big\rangle_{\mathscr{H}^{\otimes p}} \leq  p! \omega_d R^d \int_{\R^d}dz  \big\langle    \mathfrak{g}_{2,p,0},  \mathfrak{g}_{2,p,z}  \big\rangle_{\mathscr{H}^{\otimes p}}  \\
 &= \frac{ \omega_d R^d}{p!}\int_{\R^d} dz \E\big[ \beta_{t-s, t-s}(z)^p \big]  \\
 &\leq \omega_d R^d \| \varphi_1\|_\infty   (2\pi)^d \Gamma_{t-s}^p  (t-s) (4C_N)^{p-1}\exp\Big( \frac{(t-s)D_N}{2C_N} \Big)  \\
 &\leq  (t-s)R^d \Big\{ (2\pi)^d \omega_d \|\varphi_1\|_\infty \exp\Big( \frac{TD_N}{2C_N} \Big)\Big\}  \Gamma_{T}^p   (4C_N)^{p-1}. 
 \end{align*}
 Hence, as a consequence of the hypercontractivity  property (see \emph{e.g.} \cite[Corollary 2.8.14]{bluebook}), we have for $k\geq 2$
 \begin{align}
& \frac{1}{R^{d/2}}  \left\|   \sum_{p\geq 2} J_{2,p,R} \right\| _{L^k(\Omega)} \leq    \frac{1}{R^{d/2}}   \sum_{p\geq 2} \left\|  J_{2,p,R} \right\| _{L^k(\Omega)} \leq    \frac{1}{R^{d/2}}  \sum_{p\geq 2}(k-1)^{p/2}   \left\|  J_{2,p,R} \right\| _{L^2(\Omega)} \notag \\
 &\leq  \sqrt{t-s}  \Big\{ (2\pi)^d \omega_d \|\varphi_1\|_\infty \exp\Big( \frac{TD_N}{2C_N} \Big)/ (4C_N)\Big\}^{1/2} \sum_{p\geq 1}  \Big[ 4(k-1) \Gamma_{T}C_N  \Big]^{p/2} \notag \\
 &=  \sqrt{t-s}  \Big\{ (2\pi)^d \omega_d \|\varphi_1\|_\infty \exp\Big( \frac{TD_N}{2C_N} \Big) \Big\}^{1/2} \frac{  \sqrt{ (k-1) \Gamma_{T} }  }{ 1 -   2 \sqrt{ (k-1) \Gamma_{T}C_N }   }, \label{ccomb1}
 \end{align}
provided  $0< 4(k-1) \Gamma_T C_N < 1$, which is always valid for some  $N>0$. 
 For $p=1$, we have, in view of \eqref{DN3},
 \begin{align*}
 R^{-d/2} \|  J_{2,1,R} \| _{L^k(\Omega)} 
& =c_k   R^{-d/2}  \| J_{2,1,R} \| _{L^2(\Omega)} 
 \le c_k \left( \int_{\R^d} \E[ \beta_{t-s,t-s} (z)] dz \right)^{1/2} \\
& \le c_k \sqrt{t-s} (\Gamma_T  \| \gamma_1\|_{L^1(\R^d)} )^{1/2},
 \end{align*}
 where $c_k =( \E[|Z|^k])^{1/k}$, with $Z\sim N(0,1)$.
 \medskip
 
Now let us estimate  the $L^2(\Omega)$-norm of $J_{1,p,R}$. Put 
\[
\wh{d}_1(s,t,x; s_1, y_1) =(s-s_1)^{-\alpha} G(4t-4s_1, x- y_1)  {\bf 1}_{[0,s)}(s_1)
\]
 and 
 \[
 \wh{ \mathfrak{g}}_{1,p,x}(\pmb{s_p},\pmb{y_p}) = \frac{1 }{p!} \sum_{\sigma\in\mathfrak{S}_p} {\bf 1}_{\Delta_p(s)}(\pmb{s^\sigma_p})  \wh{d}_1(s,t,x;s^\sigma_1, y^\sigma_1)  \prod_{j=1}^{p-1}  G(s^\sigma_j -s^\sigma_{j+1}, y^\sigma_j- y^\sigma_{j+1}).
 \]
 From \eqref{CH31} we deduce  that 
 \[
\Big\vert  \big\langle   \mathfrak{g}_{1,p,x},  \mathfrak{g}_{1,p,y}  \big\rangle_{\mathscr{H}^{\otimes p}}   \Big\vert \leq C_\alpha^2 (t-s)^{2\alpha}\big\langle   \wh{\mathfrak{g}}_{1,p,x},  \wh{\mathfrak{g}}_{1,p,y}  \big\rangle_{\mathscr{H}^{\otimes p}}. 
 \] 
 Similarly as before, we can write 
 \begin{align*}
 \big(\F \, \wh{\mathfrak{g}}_{1,p,x}\big) (\pmb{s_p}, \pmb{\xi_p}) &= \frac{1 }{p!} \sum_{\sigma\in\mathfrak{S}_p} {\bf 1}_{\Delta_p(s)}(\pmb{s^\sigma_p}) e^{-\i x\cdot \tau(\pmb{\xi_p})} (s- s^\sigma_1)^{-\alpha}  \\
& \qquad \times \E\left[ e^{ -\i  \sum_{j=1}^p (X^1_{4t} - X^{1}_{4s_1^\sigma} + X^{1}_{s_1^\sigma}  - X^{1}_{s_j^\sigma}) \cdot \xi^\sigma_j } \right],
 \end{align*}
 from which we see that $\big\langle   \wh{\mathfrak{g}}_{1,p,x},  \wh{\mathfrak{g}}_{1,p,y}  \big\rangle_{\mathscr{H}^{\otimes p}}$ is a nonnegative function that depends only on the difference $x-y$ and is given by
 \begin{align}
& \big\langle   \wh{\mathfrak{g}}_{1,p,x},  \wh{\mathfrak{g}}_{1,p,y}  \big\rangle_{\mathscr{H}^{\otimes p}} \notag \\
&=\int_{[0,s]^{2p}} d\pmb{s_p} d\pmb{r_p}  \prod_{j=1}^p\gamma_0(s_j -r_j) \int_{\R^{pd}} \mu_1(d\pmb{\xi_p})\big(\F \, \wh{\mathfrak{g}}_{1,p,x}\big) (\pmb{s_p}, \pmb{\xi_p})  \big(\F \, \wh{\mathfrak{g}}_{1,p,y}\big) (\pmb{r_p}, -\pmb{\xi_p}) \notag  \\
&=\frac{1}{(p!)^2} \sum_{\sigma,\pi\in\mathfrak{S}_p}\int_{\Delta_p(s)^{2}}  d\pmb{s^\sigma_p} d\pmb{r^\pi_p} \frac{\prod_{j=1}^p\gamma_0(s_j -r_j)  }{(s - s_1^\sigma)^{\alpha} (s - r_1^\pi)^{\alpha}}  \int_{\R^{pd}} \mu_1(d\pmb{\xi_p}) e^{-\i (x-y)\cdot \tau(\pmb{\xi_p})}  \notag  \\
&\quad   \times \E\left[ e^{ -\i  \sum_{j=1}^p (X^1_{4t} - X^{1}_{4s_1^\sigma} + X^{1}_{s_1^\sigma}  - X^{1}_{s_j^\sigma}) \cdot \xi^\sigma_j } \right] \E\left[ e^{ -\i  \sum_{j=1}^p (X^1_{4t} - X^{1}_{4r_1^\pi} + X^{1}_{r_1^\pi}  - X^{1}_{r_j^\pi}) \cdot \xi^\pi_j } \right]. \label{see2}
 \end{align}
 Then, we can write for $p\geq 2$,
 \begin{align}
 \big\| J_{1,p,R} \big\|^2_{L^2(\Omega)} &= p! \int_{B_R^2}dxdy \big\langle \mathfrak{g}_{1,p,x},  \mathfrak{g}_{1,p,y} \big\rangle_{\mathscr{H}^{\otimes p}} \notag \\
& \leq  C_\alpha^2 (t-s)^{2\alpha}p! \int_{B_R^2}dxdy \big\langle \wh{\mathfrak{g}}_{1,p,x},  \wh{\mathfrak{g}}_{1,p,y} \big\rangle_{\mathscr{H}^{\otimes p}} \notag \\
&\leq C_\alpha^2 (t-s)^{2\alpha}p! \omega_d R^d\int_{\R^d}dz \big\langle \wh{\mathfrak{g}}_{1,p,0},  \wh{\mathfrak{g}}_{1,p,z} \big\rangle_{\mathscr{H}^{\otimes p}}. \label{see1}
 \end{align}
 By the same trick of inserting $\exp\big(-\frac{\e}{2} \|z\|^2\big)$, we have 
 \begin{align}
 \int_{\R^d}dz \big\langle \wh{\mathfrak{g}}_{1,p,0},  \wh{\mathfrak{g}}_{1,p,z} \big\rangle_{\mathscr{H}^{\otimes p}} =\lim_{\e\downarrow 0} \int_{\R^d}dz \big\langle \wh{\mathfrak{g}}_{1,p,0},  \wh{\mathfrak{g}}_{1,p,z} \big\rangle_{\mathscr{H}^{\otimes p}} e^
 {-\frac{\e}{2} \|z\|^2} = : \lim_{\e\downarrow 0}\wh{T}_{p,\e}, \label{see11}
 \end{align}
 where $\wh{T}_{p,\e}$ is equal to
 \begin{align}
&  \int_{[0,s]^{2p}} d\pmb{s_p} d\pmb{r_p}  \prod_{j=1}^p\gamma_0(s_j -r_j) \int_{\R^{pd+d}}dz \mu_1(d\pmb{\xi_p})\big(\F \, \wh{\mathfrak{g}}_{1,p,0}\big) (\pmb{s_p}, \pmb{\xi_p})  \big(\F \, \wh{\mathfrak{g}}_{1,p,z}\big) (\pmb{r_p}, -\pmb{\xi_p})  \notag \\
&=\frac{(2\pi)^d}{(p!)^2} \sum_{\sigma,\pi\in\mathfrak{S}_p}\int_{\Delta_p(s)^2} d\pmb{s^\sigma_p} d\pmb{r^\pi_p} \frac{\prod_{j=1}^p\gamma_0(s_j -r_j)  }{(s - s_1^\sigma)^{\alpha} (s - r_1^\pi)^{\alpha}}   \int_{\R^{pd}} \mu_1(d\pmb{\xi_p}) G(\e,  \tau(\pmb{\xi_p}) )  \notag  \\
&\quad   \times \E\left[ e^{ -\i  \sum_{j=1}^p (X^1_{4t} - X^{1}_{4s_1^\sigma} + X^{1}_{s_1^\sigma}  - X^{1}_{s_j^\sigma}) \cdot \xi^\sigma_j } \right] \E\left[ e^{ -\i  \sum_{j=1}^p (X^1_{4t} - X^{1}_{4r_1^\pi} + X^{1}_{r_1^\pi}  - X^{1}_{r_j^\pi}) \cdot \xi^\sigma_j } \right]. \label{EASee}
 \end{align}
 Note that for $\pmb{s^\sigma_p}\in\Delta_p(s)$, $2t - 2s^{\sigma}_1 > 2s - 2s^{\sigma}_1 > \frac{1}{2}(s - s^{\sigma}_1)$ so that
 \begin{align}
& \E\left[ e^{ -\i  \sum_{j=1}^p (X^1_{4t} - X^{1}_{4s_1^\sigma} + X^{1}_{s_1^\sigma}  - X^{1}_{s_j^\sigma}) \cdot \xi^\sigma_j } \right]  = e^{-(2t - 2s^\sigma_1) \| \tau(\pmb{\xi_p}) \|^2} e^{-\frac{1}{2} \sum_{j=1}^{p-1} (s^\sigma_j - s^\sigma_{j+1}) \| \xi^\sigma_{j+1} + \dots+ \xi^\sigma_p \|^2  } \notag  \\
&\qquad\qquad \leq e^{-\frac{1}{2}(s -s^\sigma_1) \| \tau(\pmb{\xi_p}) \|^2} e^{-\frac{1}{2} \sum_{j=1}^{p-1} (s^\sigma_j - s^\sigma_{j+1}) \| \xi^\sigma_{j+1} + \dots+ \xi^\sigma_p \|^2  } \notag \\
&\qquad\qquad   = \E\left[ e^{ -\i  \sum_{j=1}^p ( X^{1}_{s}  - X^{1}_{s_j^\sigma}) \cdot \xi^\sigma_j } \right] =  \E\left[ e^{ -\i  \sum_{j=1}^p ( X^{1}_{s}  - X^{1}_{s_j}) \cdot \xi_j } \right]  \notag  \\
&\qquad\qquad  = \exp\left( -\frac{1}{2} \Var  \sum_{j=1}^p ( X^{1}_{s}  - X^{1}_{s_j}) \cdot \xi_j  \right). \label{see3}
 \end{align}
 Therefore,  we can write 
 \begin{align*}
 \wh{T}_{p,\e}&\leq \frac{ (2\pi)^d }{(p!)^2}   \int_{[0,s]^{2p}}   d\pmb{s_p} d\pmb{r_p} \frac{\prod_{j=1}^p\gamma_0(s_j -r_j)  }{(s - s_1)^{\alpha} (s - r_1)^{\alpha}}   \int_{\R^{pd}} \mu_1(d\pmb{\xi_p}) G(\e,  \tau(\pmb{\xi_p}) )  \\
&  \quad  \times  \E\left[ e^{ -\i  \sum_{j=1}^p ( X^{1}_{s}  - X^{1}_{s_j}) \cdot \xi_j } \right]   \E\left[ e^{ -\i  \sum_{j=1}^p ( X^{1}_{s}  - X^{1}_{r_j}) \cdot \xi_j } \right]   {\bf 1}_{\{  s_1 > s_2\vee \dots\vee s_p  \}}  {\bf 1}_{\{  r_1 > r_2\vee \dots \vee r_p  \}}  \\
&\leq   \frac{ (2\pi)^d \Gamma_s^{p-1} }{(p!)^2}   \int_{[0,s]^{p+1}}dr_1 ds_1  \cdots ds_p  \frac{\gamma_0(s_1-r_1) } { (s-r_1)^{\alpha}(s-s_1)^{\alpha}} {\bf 1}_{\{  s_1 > s_2\vee \dots \vee s_p  \}} \\  &\qquad  \times \int_{\R^{pd}} \mu_1(d\pmb{\xi_p}) G(\e,  \tau(\pmb{\xi_p}) )      \exp\left( -   \frac 12 \Var  \sum_{j=1}^p ( X^{1}_{s}  - X^{1}_{s_j}) \cdot \xi_j  \right).  
 \end{align*}
 By the usual time change  $(r_1, s_j)\to (s-r_1, s-s_j)$, we have 
  \begin{align}
 \wh{T}_{p,\e}&\leq   \frac{ (2\pi)^d \Gamma_s^{p-1} }{(p!)^2}   \int_{[0,s]^{p+1}}dr_1 ds_1 \cdots  ds_p  \frac{\gamma_0(s_1-r_1) } { r_1^{\alpha} s_1^{\alpha}} {\bf 1}_{\{  s_1 < s_2\wedge  \dots \wedge s_p  \}}  \notag \\
   &\qquad  \times \int_{\R^{pd}} \mu_1(d\pmb{\xi_p}) G(\e,  \tau(\pmb{\xi_p}) )      \exp\left( - \frac 12 \Var  \sum_{j=1}^p  X^{1}_{s_j} \cdot \xi_j  \right). \notag 
 \end{align}
 Note that for  $s_1 < s_2\wedge \cdots \wedge s_p $
 \begin{align*}
  e^{ -\frac12   \Var  \sum_{j=1}^p  X^{1}_{s_j} \cdot \xi_j  } &= e^{- \frac 12s_1 \| \tau(\pmb{\xi_p})\|^2} e^{- \frac 12 \Var\sum_{j=2}^p (X^1_{s_j} - X^1_{s_1})\cdot \xi_j   } \\
  &= e^{-\frac 12s_1 \| \tau(\pmb{\xi_p})\|^2} e^{- \frac 12\Var\sum_{j=2}^p X^1_{s_j-s_1}\cdot \xi_j   }.
 \end{align*}
 Then,  by another time change $(s_j-s_1\to s_j)$ for $j\geq 2$, we can write
   \begin{align}
   \wh{T}_{p,\e}& \leq   \frac{ (2\pi)^d \Gamma_s^{p-1} }{(p!)^2}   \int_0^s  \int_0^s  dr_1 ds_1   \frac{\gamma_0(s_1-r_1) } { r_1^{\alpha} s_1^{\alpha}}  \int_{[0,s-s_1]^{p-1}} ds_2  \cdots ds_p    \notag  \\
   &\qquad\qquad  \times \int_{\R^{pd}} \mu_1(d\pmb{\xi_p}) G(\e,  \tau(\pmb{\xi_p}) )    e^{-\frac 12 s_1 \| \tau(\pmb{\xi_p})\|^2} e^{-\frac 12 \Var\sum_{j=2}^p X^1_{s_j}\cdot \xi_j   } \notag\\
   &\leq   \frac{ (2\pi)^d \Gamma_s^{p-1} }{(p!)^2} \left(  \int_0^s  \int_0^s  dr_1 ds_1   \frac{\gamma_0(s_1-r_1) } { r_1^{\alpha} s_1^{\alpha}}\right)       \notag\\
   &\quad   \times \int_{[0,s]^{p-1}} ds_2 \cdots ds_p  \int_{\R^{pd}} \mu_1(d\pmb{\xi_p}) G(\e,  \tau(\pmb{\xi_p}) )    e^{-\frac 12 \Var\sum_{j=2}^p X^1_{s_j}\cdot \xi_j   } \notag\\
   &= \frac{ (2\pi)^d \Gamma_s^{p-1} }{(p!)^2} \left(  \int_0^s  \int_0^s  dr_1 ds_1   \frac{\gamma_0(s_1-r_1) } { r_1^{\alpha} s_1^{\alpha}}\right)   (p-1)! \int_{\text{SIM}_{p-1}(s)} dw_2 \cdots  dw_p   \notag   \\
   &\quad   \times  \int_{\R^{pd}} \mu_1(d\pmb{\xi_p}) G(\e,  \tau(\pmb{\xi_p}) )    \exp\left(- \frac 12 \sum_{j=2}^p w_j \| \xi_2 + \cdots + \xi_j\|^2 \right). \label{see44}
 \end{align}
 Now making the change of variables $\eta_j = \xi_1+ \cdots + \xi_j$ yields
 \begin{align*}
 &\int_{\text{SIM}_{p-1}(s)} dw_2 \cdots dw_p \int_{\R^{pd}} \mu_1(d\pmb{\xi_p}) G(\e,  \tau(\pmb{\xi_p}) )    \exp\left(- \frac 12\sum_{j=2}^p w_j \| \xi_2 +  \cdots+ \xi_j\|^2 \right)\\
 &=\int_{\text{SIM}_{p-1}(s)} dw_2  \cdots dw_p \int_{\R^d} d\eta_p G(\e,  \eta_p ) \int_{\R^{pd-d}} d\pmb{\eta_{p-1}} \left( \varphi_1(\eta_1)  e^{-\frac 12 w_p \| \eta_p-\eta_1\|^2}    \right)  \\
  &\qquad \qquad\times \left( \varphi_1(\eta_2-\eta_1) \varphi_1(\eta_3-\eta_2) e^{-\frac 12 w_2\| \eta_2 - \eta_1\|^2}  \right) \left(   \varphi_1( \eta_4-\eta_3) e^{-\frac 12 w_3 \| \eta_3 - \eta_1  \|^2}  \right)\\
  &\qquad\qquad \times  \cdots \times \left(   \varphi_1( \eta_p-\eta_{p-1}) e^{-\frac 12 w_{p-1} \| \eta_{p-1} - \eta_1  \|^2}  \right).
 \end{align*}
 Moreover, we can apply \eqref{ecuu1} and \eqref{ecu1} to  the integral with respect to the variables  $d\eta_2, d\eta_3, \dots, d\eta_{p-1}, d\eta_1$ in order  to get
 \begin{align*}
 \int_{\R^d} d\eta_2 \varphi_1(\eta_2-\eta_1) \varphi_1(\eta_3-\eta_2) e^{-\frac 12 w_2\| \eta_2 - \eta_1\|^2} &\leq \int_{\R^d} \varphi_1(\xi)^2 e^{-\frac 12w_2\| \xi\|^2}d\xi     \\
 \int_{\R^d} d\eta_3  \varphi_1( \eta_4-\eta_3) e^{-\frac 12 w_3 \| \eta_3 - \eta_1  \|^2}  &\leq  \int_{\R^d} \varphi_1(\xi) e^{-\frac 12 w_3\| \xi\|^2}d\xi   \\
   \ldots & \ldots \\
   \int_{\R^d} d\eta_{p-1}  \varphi_1( \eta_p-\eta_{p-1}) e^{-\frac 12 w_{p-1} \| \eta_{p-1} - \eta_1  \|^2} & \leq \int_{\R^d} \varphi_1(\xi) e^{-\frac 12w_{p-1}\| \xi\|^2}d\xi   \\
      \int_{\R^d} d\eta_{1}  \varphi_1(\eta_1)  e^{-\frac 12w_p \| \eta_p-\eta_1\|^2}      & \leq \int_{\R^d} \varphi_1(\xi) e^{-\frac 12 w_p\| \xi\|^2}d\xi . 
 \end{align*}
 Thus, with $\Gamma_{s,\alpha} =    \int_0^s  \int_0^s  dr_1 ds_1    \gamma_0(s_1-r_1)  r_1^{-\alpha} s_1^{-\alpha} $, we have
 \begin{align*}
   \wh{T}_{p,\e}& \leq  \frac{ (2\pi)^d \Gamma_s^{p-1} \| \varphi_1\|_\infty \Gamma_{s,\alpha}}{p!p}  \int_{\text{SIM}_{p-1}(s)} dw_2  \cdots dw_p    \int_{\R^{pd-d}} \prod_{j=2}^p \varphi_1(\xi_j) e^{-\frac 12 w_j \| \xi_j\|^2}
    \\
   &\leq    \frac{ (2\pi)^d \Gamma_s^{p-1} \| \varphi_1\|_\infty \Gamma_{s,\alpha}}{p!p} \sum_{j=1}^{p-1} {p-1 \choose j} \frac{s^j}{j!} D_N^j (2C_N)^{p-1-j} \quad\text{by \eqref{fromh}}\\
   &\leq  \frac{ (2\pi)^d  \| \varphi_1\|_\infty \Gamma_{s,\alpha} \exp\Big( sD_N/(2C_N) \Big)}{p!p}    (4C_N\Gamma_s )^{p-1}.
 \end{align*}
Therefore, for $p\geq 2$,
 \begin{align*}
&\quad  \big\| J_{1,p,R} \big\|^2_{L^2(\Omega)} \\
& \leq  (t-s)^{2\alpha}    R^d \Big\{(2\pi)^d C_\alpha^2\omega_d \| \varphi_1\|_\infty \Gamma_{s,\alpha} \exp\big( sD_N/(2C_N) \big) \Big\}(4C_N\Gamma_s )^{p-1}.
 \end{align*}
 For $p=1$, it is   easier to get the desired bound. Indeed, from \eqref{EASee}, it follows  that 
 \begin{align*}
 \wh{T}_{1,\e} &= (2\pi)^d \int_0^s \int_0^s ds_1 dr_1 \gamma_0(s_1-r_1) (s-s_1)^{-\alpha}(s-r_1)^{-\alpha} \int_{\R^d} d\xi \varphi_1(\xi) G(\e, \xi) \\
 &\qquad\qquad\qquad \times \E\left[ e^{ -\i ( X^1_{4t} - X^1_{4s_1}) \cdot \xi } \right] \E\left[ e^{ -\i ( X^1_{4t} - X^1_{4r_1}) \cdot \xi } \right] \\
 &\leq   (2\pi)^d \| \varphi_1\|_\infty \Gamma_{s,\alpha},
 \end{align*}
 so that
 \[
  \big\| J_{1,1,R} \big\|^2_{L^2(\Omega)} \leq  (t-s)^{2\alpha}    R^d \Big\{(2\pi)^d C_\alpha^2\omega_d \| \varphi_1\|_\infty \Gamma_{s,\alpha}  \Big\}.
  \]
  Hence,
 \begin{align}
& \frac{1}{R^{d/2}}  \left\|   \sum_{p\geq 1} J_{1,p,R} \right\| _{L^k(\Omega)} \leq       \frac{1}{R^{d/2}}  \sum_{p\geq 1}(k-1)^{p/2}   \left\|  J_{1,p,R} \right\| _{L^2(\Omega)} \notag \\
 &\leq  (t-s)^\alpha  \Big\{ (2\pi)^d C_\alpha^2 \omega_d \|\varphi_1\|_\infty \big[1+ \exp( TD_NC_N^{-1}) \big] \Gamma_{s,\alpha} \Big\}^{1/2} \sum_{p\geq 0}  \Big[ 4(k-1) \Gamma_{T}C_N  \Big]^{p/2} \notag \\
 &= (t-s)^\alpha  \frac{ \Big\{ (2\pi)^d C_\alpha^2 \omega_d \|\varphi_1\|_\infty \big[1+ \exp( TD_NC_N^{-1}) \big] \Gamma_{s,\alpha} \Big\}^{1/2} }{1 - 2\sqrt{(k-1) \Gamma_T C_N   }}   , \label{ccomb2}
 \end{align}
provided  $0< 4(k-1) \Gamma_T C_N < 1$, which is always valid for some  $N>0$. 

\smallskip
 
     Combing \eqref{ccomb1} and \eqref{ccomb2}, we get \eqref{G0} and hence the desired tightness. \hfill $\square$

\subsection{Proof of Theorem \ref{Nchaotic}} \label{sec310a}
 We are going to show that, under the hypotheses of Theorem   \ref{Nchaotic}, the first chaos dominates
 and, as a consequence, the proof of the central limit theorem reduces to the computation of the limit variance of the first chaos. The proof will be done in several steps.
 
 \medskip
 \noindent
 {\it Step 1}.  
 We have shown in the proof of Theorem \ref{thmSHE} that, if  $\gamma_0$ is locally integrable, $\gamma_1$ is integrable and Dalang's condition \eqref{Da} is satisfied, then  for any integer $p\geq 2$,  
   \begin{align}
   \Var\Big ( \Pi_p A_t(R) \Big) \sim \sigma_p(t,t) R^d ~\text{as $R\to+\infty$ and }~ \sum_{p\geq 2} \sigma_{p}(t,t) < \infty. \label{SAMEr}
   \end{align}
 The above results also hold true, provided    $\gamma_0$ is locally integrable and the modified version of Dalang's condition \eqref{mDc2} is satisfied. To see the latter point, it is enough to proceed with the same arguments  but replacing the estimate  \eqref{KU2} by
 \[
\int_{\R^d}  \varphi_1(\eta_1)  \varphi_1(\eta_2 - \eta_1) h_1(\eta_1)  d\eta_1   \leq    \int_{\R^d} \varphi_1(\eta_1)^2 h_1(\eta_1)d\eta_1,
\] 
obtained by applying (\ref{ecu1}). 
 Then, we can use the same arguments as in  the proof of \cite[Lemma 3.3]{HHNT15}, with $C_N, D_N$ replaced by 
 \[
 C_N' =\int_{\{  \| \xi \| \geq N \}} \frac{  \varphi_1(\xi)+ \varphi_1(\xi)^2   }{\|\xi\|^2 } d\xi \quad{\rm and } \quad D_N' = \int_{\{  \| \xi \| \leq N \}} \big( \varphi_1(\xi)+ \varphi_1(\xi)^2  \big)d\xi.
 \]
 In this way, instead of the inequality \eqref{fromh}, we can get 
 \begin{align}\label{new35}
 Q_p(\eta_p) \leq t \sum_{j=0}^{p-1} {p-1 \choose j} \frac{t^j}{j!} (D_N')^j (2C_N')^{p-1-j}
 \end{align}
 and by choosing large $N$ such that $0< 4\Gamma_t C_N' < 1$, we can get  instead  of \eqref{stp1}
 \begin{align}\label{new36}
 \int_{\R^d}\E\big[ \beta_{s,t}(z)^p\big] dz  \leq   (2\pi)^d \Gamma_t^p p! t (4C'_N)^{p-1}\exp\Big( \frac{tD'_N}{2C'_N} \Big)< \infty
 \end{align}
  and as a result,
 \[
 \sum_{p\geq 2} \frac{1}{p!} \int_{\R^d} \E\big[ \beta_{t,t}(z)^p \big] dz < +\infty,
 \]
which is equivalent to   \eqref{SAMEr}.

\medskip
\noindent
{\it Step 2.}    
    For the first chaotic component,     if   $\gamma_1\notin L^1(\R^d)$, then 
      \[
    R^{-d} \Var\Big( \Pi_1 A_t(R) \Big) \to  \infty ~\text{as $R\to+\infty$.}
   \]
 This observation, together with Step 1,  justifies part (1)  of Theorem \ref{Nchaotic}.

\medskip
\noindent
{\it Step 3.} 
    When $\gamma_1(z) = \| z \|^{-\beta}$ for some $\beta\in(0, 2\wedge d)$, let us first compute the variance of $\Pi_1 A_t(R)$.  We have
\begin{align}
 \Var\big( \Pi_1 A_t(R) \big)  &=  \int_0^t\int_0^t du dv \gamma_0(u-v)  \notag \\
 &\qquad \times \int_{\R^d} d\xi \int_{B_R^2 }dx dy e^{-\i (x-y)\cdot \xi} c_{d,\beta} \| \xi\|^{\beta- d} e^{-\frac{1}{2} (u+v) \| \xi\|^2}, \notag
\end{align}
for some constant $c_{d,\beta}$.
Then by making change of variables $(x , y,\xi) \to (Rx, Ry, \xi/R)$, we get
\begin{align}
&  \Var\big( \Pi_1 A_t(R) \big) R^{-2d+\beta}\label{changesame}    \\
&= \int_0^t\int_0^t du dv \gamma_0(u-v) \int_{\R^d} d\xi \left[ \int_{B_1^2 }dx dy e^{-\i (x-y)\cdot \xi}\right] c_{d,\beta} \| \xi\|^{\beta- d} e^{-\frac{1}{2R^2} (u+v) \| \xi\|^2}. \notag
\end{align}
This expression is increasing in $R$ and it converges, as $R\rightarrow +\infty$, to
 \[
   \int_0^t\int_0^t du dv \gamma_0(u-v) \int_{\R^d} d\xi  \int_{B_1^2 }dx dy e^{-\i (x-y)\cdot \xi}  \varphi_1(\xi)  = \kappa_\beta \in (0,\infty).
\]
Then, it suffices to show that 
\[
\sum_{p\geq 2} {\rm Var}\big( \Pi_p A_t(R) \big) = o(R^{2d-\beta}),
\]
which implies the central limit theorem \eqref{Nchaotic2}  immediately.  For $p\geq 2$,   we read from \eqref{DTW}, \eqref{FK} and \eqref{acexp} that
\begin{align*}
 \Var\big( \Pi_p A_t(R) \big) & =  \frac{c_{d,\beta}^p}{p!} \int_{B_R^2}dxdy \int_{[0,t]^{2p}} d\pmb{s_p} d\pmb{r_p} \prod_{j=1}^p \gamma_0(s_j-r_j)\int_{\R^{pd}} d\pmb{\xi_p} \left( \prod_{j=1}^p \| \xi_j\|^{\beta -d}\right)  \\
  &\qquad \times e^{-\i (x-y) \cdot \tau(\pmb{\xi_p} )}  e^{-\frac{1}{2} \Var \sum_{j=1}^p \xi_j \cdot  X_{s_j}^1 } e^{-\frac{1}{2} \Var \sum_{j=1}^p \xi_j \cdot  X_{r_j}^2 }.
  \end{align*}
 Note that $$\int_{B_R^2}dxdy  e^{-\i (x-y) \cdot \tau(\pmb{\xi_p} )} = (2\pi R)^d\omega_d \ell_R\big(  \tau(\pmb{\xi_p} )\big)\geq 0  .$$ Then by   similar arguments as before, we obtain
 \begin{align*}
 \Var\big( \Pi_p A_t(R) \big) & \leq  \frac{c_{d,\beta}^p}{p!} \int_{B_R^2}dxdy \int_{[0,t]^{2p}} d\pmb{s_p} d\pmb{r_p} \prod_{j=1}^p \gamma_0(s_j-r_j)\int_{\R^{pd}} d\pmb{\xi_p} \left( \prod_{j=1}^p \| \xi_j\|^{\beta -d}\right)  \\
  &\qquad\qquad\qquad \times e^{-\i (x-y) \cdot \tau(\pmb{\xi_p} )}  \exp\left( - \frac 12\Var \sum_{j=1}^p \xi_j \cdot  X_{s_j}^1 \right) \\
    &\le c_{d,\beta}^p \Gamma_t^p \int_{B_R^2}dxdy \int_{\text{SIM}_p(t)} d\pmb{w_p} \int_{\R^{pd}} d\pmb{\xi_p} \left( \prod_{j=1}^p \| \xi_j\|^{\beta -d}\right)e^{-\i (x-y) \cdot \tau(\pmb{\xi_p} )}  \\
  &\qquad\qquad\qquad \times    \exp\left( -\frac 12 \sum_{j=1}^p w_j \| \xi_1  +\cdots + \xi_j \|^2 \right).
  \end{align*}
By the usual change of variables $\eta_j = \xi_1  +\cdots + \xi_j$, with $\eta_0=0$, and $(x,y,\eta_p)\to (Rx, Ry, \eta_p/R)$, we obtain 
 \begin{align}
& \Var\big( \Pi_p A_t(R) \big)  \leq  c_{d,\beta}^{p-1} \Gamma_t^p R^d  \int_{\text{SIM}_p(t)} d\pmb{w_p} \int_{\R^{pd-d}} d\pmb{\eta_{p-1}} \left( \prod_{j=1}^{p-1} \| \eta_j - \eta_{j-1}\|^{\beta -d} e^{-\frac 12w_j \| \eta_j \|^2}\right)  \notag   \\
  &\qquad\qquad  \times \int_{\R^d}d\eta_p   \|   \eta_p R^{-1} - \eta_{p-1} \|^{\beta-d}  \int_{B_1^2}dxdy  e^{-\i (x-y) \cdot   \eta_p}  e^{-w_p \| \eta_p \|^2/(2R^2)}. \label{displ1}
  \end{align}
Let us first analyze the part in the display \eqref{displ1}, which can be rewritten as 
\begin{align}
& R^{d-\beta}  \int_{\R^d}d\eta_p   \|    \eta_p - R\eta_{p-1} \|^{\beta-d}  \int_{B_1^2}dxdy  e^{-\i (x-y) \cdot   \eta_p}  e^{-w_p \| \eta_p \|^2/(2R^2)} \notag \\
&\leq R^{d-\beta} \int_{B_1^2}dxdy   \int_{\R^d}d\eta_p   \|    \eta_p - R\eta_{p-1} \|^{\beta-d}  e^{-\i (x-y) \cdot   \eta_p} \notag  \\
&=c_{d,\beta} ^{-1} R^{d-\beta} \int_{B_1^2}dxdy  e^{-\i (x-y)\cdot\eta_{p-1}R} \| x- y \|^{-\beta}  =: R^{d-\beta} U_{R}(\eta_{p-1}). \label{sei}
\end{align}
The function $U_R$ defined above is uniformly bounded by  $c_{d,\beta}^{-1} \int_{B_1^2}dxdy   \| x- y \|^{-\beta}$ and  for $\eta_{p-1}\neq 0$, by the Riemann-Lebesgue's Lemma, $0\leq U_R(\eta_{p-1}) $ converges to zero as $R\to+ \infty$.  As a result, 
\begin{align}
R^{-2d+\beta} \sum_{p\geq 2}  \Var\big( \Pi_p A_t(R) \big) & \leq   \sum_{p\geq 2} t\Gamma_t^p  c_{d,\beta}^p \int_{\text{SIM}_{p-1}(t)} d\pmb{w_{p-1}} \int_{\R^{pd-d}} d\pmb{\eta_{p-1}} \notag  \\
&\qquad   \times \left( \prod_{j=1}^{p-1} \| \eta_j - \eta_{j-1} \|^{\beta-d}  e^{- \frac 12w_j \| \eta_j \|^2}\right) U_R(\eta_{p-1}) \notag \\
&\leq  t \left(\int_{B_1^2}dxdy   \| x- y \|^{-\beta}\right)   \sum_{p\geq 2} \Gamma_t^p  \int_{\text{SIM}_{p-1}(t)} d\pmb{w_{p-1}} \notag  \\
&\quad   \times \int_{\R^{pd-d}} d\pmb{\eta_{p-1}} \left( \prod_{j=1}^{p-1} \varphi_1(\eta_j - \eta_{j-1}) e^{-\frac 12w_j \| \eta_j \|^2}\right). \notag
\end{align}
By using \eqref{ecuu1} for the integration with respect to $d\eta_{p-1},\dots,  d\eta_3, d\eta_2$ inductively, we get 
\begin{align*}
 &\quad \sum_{p\geq 2} \Gamma_t^p  \int_{\text{SIM}_{p-1}(t)} d\pmb{w_{p-1}}   \int_{\R^{pd-d}} d\pmb{\eta_{p-1}} \left( \prod_{j=1}^{p-1} \varphi_1(\eta_j - \eta_{j-1}) e^{-\frac 12 w_j \| \eta_j \|^2}\right)\\
 &\leq \sum_{p\geq 2} \Gamma_t^p  \int_{\text{SIM}_{p-1}(t)} d\pmb{w_{p-1}}   \int_{\R^{pd-d}} d\pmb{\eta_{p-1}} \left( \prod_{j=1}^{p-1} \varphi_1(\eta_j ) e^{-\frac 12 w_j \| \eta_j \|^2}\right), 
\end{align*}
which is a convergent series by previous discussion. Then by  dominated convergence and the  Riemann-Lebesgue's lemma, we have
\[
 \sum_{p\geq 2}  \Var\big( \Pi_p A_t(R) \big) =o(R^{2d-\beta}).
 \]
This tells us that the first chaos is indeed dominant and we have the desired Gaussian fluctuation \eqref{Nchaotic2}. This concludes the proof of Theorem \ref{Nchaotic}. \hfill $\square$

\medskip

\subsection{Proof of Theorem \ref{NCF}} Part (1): The proof of  the functional CLT for $\wh{A}_t(R)$ can be done exactly by  the same arguments from Sections 3.1, 3.2 and 3.3 except for using \eqref{new35} and \eqref{new36} instead of \eqref{fromh} and \eqref{stp1}. So we leave the details for interested readers and refer to the forthcoming work \cite{NSZ20} for similar situation when dealing with parabolic Anderson model driven by rough noise. 
 
 \medskip
 
Part (2): By results in part (2) of Theorem \ref{Nchaotic}, $R^{-d+\frac{\beta}{2}} \wh{A}_t(R)$ converges to the zero process in finite-dimensional distributions. So our proof consists in two parts: 
\begin{itemize}
\item[(i)] We prove $\Big\{  R^{-d+\frac{\beta}{2}} \Pi_1\big( A_t(R)\big) : t\in\R_+\Big\}  \xrightarrow[\rm law]{R\to\infty} \wt{\mathcal{G}}.$

\item[(ii)] We prove $\big\{ R^{-d + \frac{\beta}{2}}\wh{A}_t(R): t\geq0 \big\}$ converges in law (hence in  probability) to the zero process, as $R\to\infty$. This will follow from the tightness of   $\big\{ R^{-d + \frac{\beta}{2}}\wh{A}_\bullet(R): R>0 \big\}$.
 \end{itemize}
\noindent{\it Proof of \rm (i):}    It is clear that $R^{-d + \frac{\beta}{2}} \Pi_1\big( A_t(R)\big) =R^{-d + \frac{\beta}{2}}  \int_0^t \int_{\R^d} G_{t-r}(x-z) W(dr,dz)$, $t\in\R_+$ is a centered Gaussian process with
 \begin{align*}
&\quad  R^{-2d+\beta}\E\big[ \Pi_1\big( A_t(R)\big) \Pi_1\big( A_s(R)\big)  \big] \\
&=\int_0^t\int_0^s dudv\gamma_0(u-v) \int_{\R^d} d\xi \left[ \int_{B_1^2} dxdy e^{-\i(x-y)\cdot\xi} \right] c_{d,\beta}\|\xi\|^{\beta-d} e^{-\frac{(t-u+s-v)}{2R^2}\|\xi\|^2}
 \end{align*}
  by the same change of variables as in \eqref{changesame}. By monotone convergence, we have
\[
R^{-2d+\beta}\E\big[ \Pi_1\big( A_t(R)\big) \Pi_1\big( A_s(R)\big)  \big]\xrightarrow{R\to\infty} \int_0^t\int_0^s dudv \gamma_0(u-v)   \int_{B_1^2} dxdy \|x-y\|^{-\beta}.
\]
This implies easily the convergence in finite-dimensional distributions.  As in section 3.3, we let $s<t$ and write 
\[
 \Pi_1\big( A_t(R)\big)-  \Pi_1\big( A_s(R)\big)  = J_{1,1,R} + J_{2,1,R}
\]
with $J_{1,1,R} := \int_0^s\int_{\R^d} \left( \int_{B_R }d_1(s,t,x; s_1, y_1)dx\right) W(ds_1, dy_1)$ and 
\[
J_{2,1,R} := \int_0^t\int_{\R^d} \left( \int_{B_R }d_2(s,t,x; s_1, y_1)dx\right) W(ds_1, dy_1),
\]
where $d_1, d_2$ are introduced in \eqref{DD1}, \eqref{DD2} and
\[
\big\vert d_1(s,t,x; s_1, y_1)   \big\vert  \leq C (t-s)^{\alpha} (s-s_1)^{-\alpha} G(4t - 4s_1, x-y_1){\bf 1}_{[0,s)}(s_1).
\]
As before, we can write
\begin{align*}
&\big\|J_{1,1,R}\big\|_{L^2(\Omega)}^2 = \int_0^s\int_0^s ds_1ds_2 \gamma_0(s_1-s_2) \int_{\R^{2d}} dy_1dy_2 \| y_1- y_2\|^{-\beta}\int_{B_R^2}dx_1dx_2  \\
&\qquad \times d_1(s,t,x_1; s_1, y_1) d_1(s,t,x_2; s_2, y_2) \\
&\leq C(t-s)^{2\alpha} \int_0^s\int_0^s ds_1ds_2 \gamma_0(s_1-s_2)  (s-s_1)^{-\alpha} (s-s_2)^{-\alpha} \int_{\R^{2d}} dy_1dy_2 \| y_1- y_2\|^{-\beta} \\
&\qquad \times \int_{B_R^2}dx_1dx_2G(4t - 4s_1, x_1-y_1) G(4t - 4s_2, x_2-y_2) \\
&= C(t-s)^{2\alpha} \int_0^s\int_0^s ds_1ds_2 \gamma_0(s_1-s_2)  (s-s_1)^{-\alpha} (s-s_2)^{-\alpha} \int_{\R^{d}} d\xi c_{d,\beta} \| \xi\|^{\beta-d} \\
&\qquad\qquad \times \int_{B_R^2}dx_1dx_2 e^{-\i (x_1-x_2)\cdot \xi}  e^{ -(2t-2s_1+2t-2s_2)\|\xi\|^2   }    \\
&\leq C(t-s)^{2\alpha} \int_0^s\int_0^s ds_1ds_2 \gamma_0(s_1-s_2)  s_1^{-\alpha} s_2^{-\alpha} \int_{\R^{d}} d\xi c_{d,\beta} \| \xi\|^{\beta-d} \\
&\qquad\qquad \times \int_{B_R^2}dx_1dx_2 e^{-\i (x_1-x_2)\cdot \xi}.
\end{align*}
 Making the change of variables $(x_1,x_2, \xi )\to (Rx_1, Rx_2, \xi/R)$ yields
\begin{align*}
&\big\|J_{1,1,R}\big\|_{L^2(\Omega)}^2 \leq C(t-s)^{2\alpha} R^{2d-\beta} \int_0^s\int_0^s ds_1ds_2 \gamma_0(s_1-s_2)  s_1^{-\alpha} s_2^{-\alpha} \int_{\R^{d}} d\xi c_{d,\beta} \| \xi\|^{\beta-d} \\
&\qquad\qquad \times \left(\int_{B_1^2}dx_1dx_2 e^{-\i (x_1-x_2)\cdot \xi}\right)  = C(t-s)^{2\alpha} R^{2d-\beta} \Gamma_{s,\alpha} \int_{B_1^2}dxdy\|x-y\|^{-\beta},
\end{align*}
where $\Gamma_{s,\alpha}$ is given as in \eqref{GAMMASA}. Now let us estimate $\big\|J_{2,1,R}\big\|_{L^2(\Omega)}^2$:
\begin{align*}
&\big\|J_{2,1,R}\big\|_{L^2(\Omega)}^2 = \int_s^t\int_s^t ds_1ds_2 \gamma_0(s_1-s_2) \int_{\R^{2d}} dy_1dy_2 \| y_1- y_2\|^{-\beta}\int_{B_R^2}dx_1dx_2  \\
&\qquad \times G(t-s_1, x_1-y_1)  G(t-s_2, x_2-y_2)  \\
&= \int_s^t\int_s^t ds_1ds_2 \gamma_0(s_1-s_2) \int_{\R^{d}} d\xi c_{d,\beta} \| \xi\|^{\beta-d}\int_{B_R^2}dx_1dx_2 e^{-\i (x_1-x_2)\cdot\xi} e^{-\frac{(2t-s_1-s_2)}{2}\|\xi\|^2  }  \\
&\leq R^{2d-\beta}  \int_s^t\int_s^t ds_1ds_2 \gamma_0(s_1-s_2) \int_{\R^{d}} d\xi c_{d,\beta} \| \xi\|^{\beta-d}\int_{B_1^2}dx_1dx_2 e^{-\i (x_1-x_2)\cdot\xi}  \\
&\leq R^{2d-\beta} (t-s)  \left(\int_{B_1^2}dxdy\|x-y\|^{-\beta}\right) \left( \int_{-t}^t \gamma_0(s_1)ds_1\right).
\end{align*}
Hence given  $T\in(0,\infty)$, we have for  any $0<s<t\leq T$ and for any $k\in[2,\infty)$,
\[
\big\| \Pi_1\big( A_t(R)\big)-  \Pi_1\big( A_s(R)\big)  \big\|_{L^k(\Omega)} = c_k \big\| \Pi_1\big( A_t(R)\big)-  \Pi_1\big( A_s(R)\big)  \big\|_{L^2(\Omega)} \leq C (t-s)^\alpha,
\]
 where $c_k$ is the $L^k(\Omega)$-norm of $Z\sim N(0,1)$ and the constant $C$ does not depend on  $R$, $s$ or $t$. This gives us the desired tightness and hence leads to the functional CLT for $\big\{\Pi_1\big(A_t(R)\big): t\in\R_+\big\}$. 
 
 \medskip
 
 \noindent{\it Proof of \rm (ii):} Given $T\in(0,\infty)$, we consider any $0< s<t \leq T$ and as before, we write
 \[
 \Pi_p(A_t(R)) -  \Pi_p(A_s(R)) = J_{1,p,R}+ J_{2,p,R}.
 \]
 Then following the   arguments that led to \eqref{displ1}, we have 
 \begin{align*}
 &\big\| J_{2,p,R}\big\|_{L^2(\Omega)}^2   \leq  C^p R^d  \int_{\text{SIM}_p(t-s)} d\pmb{w_p} \int_{\R^{pd-d}} d\pmb{\eta_{p-1}} \left( \prod_{j=1}^{p-1} \| \eta_j - \eta_{j-1}\|^{\beta -d} e^{-\frac{w_j \| \eta_j \|^2}{2}}\right)  \notag   \\
  &\qquad\qquad  \times \int_{\R^d}d\eta_p   \|   \eta_p R^{-1} - \eta_{p-1} \|^{\beta-d}  \int_{B_1^2}dxdy  e^{-\i (x-y) \cdot   \eta_p}  e^{-w_p \| \eta_p \|^2/(2R^2)} \\
  &  \leq C^p R^{2d-\beta} \int_{\text{SIM}_p(t-s)} d\pmb{w_p} \int_{\R^{pd-d}} d\pmb{\eta_{p-1}} \left( \prod_{j=1}^{p-1} \| \eta_j - \eta_{j-1}\|^{\beta -d} e^{-\frac{w_j \| \eta_j \|^2}{2}}\right), ~\text{see \eqref{sei}}\\
  &\leq C^p R^{2d-\beta} (t-s) \int_{\text{SIM}_{p-1}(t-s)} d\pmb{w_{p-1}} \int_{\R^{pd-d}} d\pmb{\eta_{p-1}}  \prod_{j=1}^{p-1} \| \eta_j - \eta_{j-1}\|^{\beta -d} e^{-\frac{w_j \| \eta_j \|^2}{2}}.
  \end{align*}
   By using \eqref{ecuu1} under the Dalang's condition, we have
 \[
 \int_{\R^{pd-d}} d\pmb{\eta_{p-1}}  \prod_{j=1}^{p-1} \| \eta_j - \eta_{j-1}\|^{\beta -d} e^{-\frac{w_j \| \eta_j \|^2}{2}}\leq  \prod_{j=1}^{p-1}  \int_{\R^{d}} d\eta_j   \| \eta_j \|^{\beta -d} e^{-\frac{w_j \| \eta_j \|^2}{2}}
 \]
 so that by the same application of Lemma 3.3 in \cite{HHNT15} as in \eqref{fromh}, we deduce 
  \begin{align*}
 \big\| J_{2,p,R}\big\|_{L^2(\Omega)}^2     \leq C R^{2d-\beta} (t-s) (4C_N)^{p-1}.
    \end{align*}
    where $C_N > 0$ can be chosen arbitrarily small for large enough $N$, see  \eqref{CNN}.

 Now let us estimate $ \big\| J_{1,p,R}\big\|_{L^2(\Omega)}^2  $: Following the arguments around \eqref{see1}, \eqref{see11}, \eqref{see2}, \eqref{see3} and \eqref{see44}, we can write  
   \begin{align*}
 \big\| J_{1,p,R}\big\|_{L^2(\Omega)}^2   &  \leq C   (t-s)^{2\alpha} \frac{1}{p!} \int_{B_R^2}dxdy \sum_{\sigma,\pi\in\mathfrak{S}_p} \int_{\Delta_p(s)^2} d\pmb{s^\sigma_p} d\pmb{r^\pi_p} \frac{\prod_{j=1}^p\gamma_0(s_j -r_j)  }{(s - s_1^\sigma)^{\alpha} (s - r_1^\pi)^{\alpha}}   \notag  \\
&\quad\times \int_{\R^{pd}} \mu_1(d\pmb{\xi_p}) e^{-\i (x-y)\cdot \tau(\pmb{\xi_p})}   \exp\left(-\frac{1}{2} \Var \sum_{j=1}^p (X_s^1 - X_{s_j}^1)\cdot\xi_j\right),
\intertext{since $ \int_{B_R^2}dxdye^{-\i (x-y)\cdot \tau(\pmb{\xi_p})} $ is nonnegative;  }
&\leq \frac{ C   (t-s)^{2\alpha}  \Gamma_{s,\alpha} \Gamma_s^{p-1}  }{p}   \int_{B_R^2}dxdy \int_{\text{SIM}_{p-1}(s)}dw_2\cdots dw_p \\
&\quad\times \int_{\R^{pd}} \mu_1(d\pmb{\xi_p}) e^{-\i (x-y)\cdot \tau(\pmb{\xi_p})} \exp\left( -\frac{1}{2} \sum_{j=2}^p w_j \| \xi_2 + \cdots+\xi_j \|^2  \right).
    \end{align*}
 Then by the usual change of variables $\eta_j = \xi_1 + \cdots + \xi_j$ and $(x,y,\eta_p)\to (Rx,Ry, \frac{\eta_p}{R})$, we have
 \begin{align*}
 &\quad  \int_{B_R^2}dxdy    \int_{\R^{pd}} \mu_1(d\pmb{\xi_p}) e^{-\i (x-y)\cdot \tau(\pmb{\xi_p})} \exp\left( -\frac{1}{2} \sum_{j=2}^p w_j \| \xi_2 + \cdots+\xi_j \|^2  \right)\\
 &= \int_{B_R^2}dxdy    \int_{\R^{pd}}d\pmb{\eta_p} \| \eta_p - \eta_{p-1} \|^{\beta-d} e^{-\i (x-y)\cdot \eta_p} e^{-\frac{1}{2}\sum_{j=2}^p w_j \| \eta_j - \eta_1\|^2}    \prod_{j=1}^{p-1} \| \eta_j - \eta_{j-1} \|^{\beta-d}\\
 &=R^{2d-\beta}  \int_{\R^{pd -d}}d\pmb{\eta_{p-1}} e^{-\frac{1}{2}\sum_{j=2}^{p-1} w_j \| \eta_j - \eta_1\|^2} \left( \prod_{j=1}^{p-1} \| \eta_j - \eta_{j-1} \|^{\beta-d}  \right)   \\
 &\qquad\qquad \times \left(\int_{B_1^2}dxdy  \int_{\R^d} d\eta_p  \| \eta_p - R\eta_{p-1} \|^{\beta-d} e^{-\i (x-y)\cdot \eta_p}   e^{-\frac{w_p}{2} \| \eta_p R^{-1} - \eta_{p-1} \|^2} \right) \\
 &\leq \frac{\int_{B_1^2} dxdy \|x-y\|^{-\beta} }{c_{d,\beta} }  R^{2d-\beta}  \int_{\R^{pd -d}}d\pmb{\eta_{p-1}} e^{-\frac{1}{2}\sum_{j=2}^{p-1} w_j \| \eta_j - \eta_1\|^2}  \prod_{j=1}^{p-1} \| \eta_j - \eta_{j-1} \|^{\beta-d}   \\
 &\leq C R^{2d-\beta} \prod_{j=1}^{p-1}  \int_{\R^{d}}d\eta_j e^{-\frac{1}{2} w_j \| \eta_j\|^2}   \| \eta_j   \|^{\beta-d}  
 \end{align*}
 where the last inequality is a consequence of \eqref{ecuu1}. So an application of Lemma 3.3 from \cite{HHNT15} yields
 \[
  \big\| J_{1,p,R}\big\|_{L^2(\Omega)}^2  \leq C (t-s)^{2\alpha} (4C_N \Gamma_s)^{p-1}.
  \]
    Therefore, for large enough $N$, we deduce from the hypercontractivity property that for any $k\in[2,\infty)$
    \begin{align*}
&\quad  \big\| \wh{A}_t(R)   -\wh{A}_s(R) \big\|_{L^k(\Omega)} \leq \sum_{p\geq 2} (k-1)^{p/2} \Big(  \big\| J_{1,p,R}\big\|_{L^2(\Omega)} + \big\| J_{2,p,R}\big\|_{L^2(\Omega)} \Big) \\
& \leq C(t-s)^{\alpha} R^{2d-\beta}\sum_{p\geq 2}  \Big( \big[ 4(k-1) C_N\Gamma_s\big]^{p/2} +  \big[ 4(k-1) C_N\big]^{p/2} \Big) \leq C(t-s)^{\alpha} R^{2d-\beta}.
    \end{align*}
  This proves (ii), and hence concludes our proof.    \hfill $\square$

 \section{Proof of technical results  }\label{tech}

\begin{proof}[Proof of Proposition \ref{Maru}]
Recall the definition of $\Psi_p$, which is defined a.e. by the following change of variables:
 \[
 \int_{\R^{pd}} \| \tau(\pmb{\xi_p}) \|^{-d} J_{d/2}(R \| \tau(\pmb{\xi_p})  \|)^2 \vert \F f_p\vert^2(\pmb{\xi_p}) \mu(d\pmb{\xi_p}) = \int_{\R^d} dx \| x\|^{-d} J_{d/2}(R \| x \|)^2 \Psi_p(x)
 \]
 with $\Psi_p(x)$ almost everywhere equal to
 \[
 \int_{\R^{pd-d}}  \vert \F f_p\vert^2(\pmb{\xi_{p-1}}, x - \tau(\pmb{\xi_{p-1}})  ) \varphi(x -\tau(\pmb{\xi_{p-1}})) \prod_{j=1}^{p-1} \varphi(\xi_j) d\pmb{\xi_{p-1}} \,.\]
 We write 
 \[
\sigma_{p,R}^2 R^{-d} = \omega_d p! (2\pi)^d \int_{\R^d} \ell_R(x) \Psi_p(x) dx \geq  \omega_d p! (2\pi)^d \int_{\{  \|x\|\leq R^{-1}\}} R^{d}\ell_1(Rx) \Psi_p(x) dx 
 \]
 and for $y=Rx\in B_1$, we have 
 \begin{align}\label{sharpe}
 (2\pi)^d\omega_d \ell_1(y) = \left(  \int_{B_1} e^{-\i y \cdot u} du \right)^2 = \left(  \int_{B_1} \cos( y \cdot u) du \right)^2 \in\big[ \cos(1)^2 \omega^2_d, \omega^2_d \big]\,.
 \end{align}
As a consequence,
\begin{align*}
\sigma_{p,R}^2 R^{-d}&\geq  p! \omega^2_d  \cos(1)^2 R^d \int_{\|x\|\leq R^{-1}}  \Psi_p(x) dx \\
&= p! \omega^2_d  \cos(1)^2 R^d  \int_{  \{    \| \tau(\pmb{\xi_p} ) \| \leq R^{-1} \} }  \vert  \F f_p\vert^2(\pmb{\xi_p})   \mu(d\pmb{\xi_p}) =p! \omega^2_d  \cos(1)^2 R^d  \wh{\Psi}_p(R^{-1}).
\end{align*}
This gives us 
\[
\liminf_{R\to+\infty}  \sigma_{p,R}^2 R^{-d} \geq \omega_d \cos(1)^2p! \liminf_{R\to+\infty}  R^{d} \wh{ \Psi}_p(R^{-1}) > 0 \,.
\]

 For the upper bound, we proceed as follows: 
\begin{align*}
 \sigma_{p,R}^2  R^{-d} &= \omega_d p! (2\pi)^d \int_{\R^d} \ell_R(x) \Psi_p(x) dx \\
&=  \omega_d p! (2\pi)^d \int_{\|x\|\leq R^{-1}} R^{d}\ell_1(Rx) \Psi_p(x) dx + \omega_d p! (2\pi)^d \int_{\|x\|> R^{-1}}  \ell_R(x) \Psi_p(x) dx \,.
\end{align*}
It follows from \eqref{sharpe} that 
\[
 (2\pi)^d \int_{\|x\|\leq R^{-1}} R^{d}\ell_1(Rx) \Psi_p(x) dx  \leq  \omega_d  R^d \int_{\|x\|\leq R^{-1}} \Psi_p(x) dx = \omega_d R^d \wh{\Psi}_p(R^{-1}) \,.
 \]
 By Lemma \ref{Bessel},  there exists some absolute constant $C$ such that $\ell_R(x)\leq  C (R/n)^d n^{-1}$ for $n \leq R\| x \| < n+1$. Therefore,
 \begin{align*}
&\quad  \int_{\|x\|> R^{-1}}  \ell_R(x) \Psi_p(x) dx =  \sum_{n=1}^\infty \int_{nR^{-1} \leq \|x\| < (n+1)R^{-1}}  \ell_R(x) \Psi_p(x) dx\\
 &\leq C \sum_{n=1}^\infty \int_{nR^{-1} \leq \|x\| < (n+1)R^{-1}} (R/n)^d n^{-1} \Psi_p(x) dx  \\
 &= CR^d \sum_{n=1}^\infty n^{-d-1} \Big( \wh{\Psi}_p(\frac{n+1}{R}) -\wh{\Psi}_p(\frac{n}{R}) \Big) \\
 &= C R^d\sum_{n=2}^\infty  \wh{\Psi}_p(n/R) \big[ (n-1)^{-d-1} - n^{-d-1} \big] \leq C R^d\sum_{n=2}^\infty  \wh{\Psi}_p(n/R) n^{-1} (n-1)^{-d-1}\\
 &=  C R^d\sum_{2\leq n \leq R^\delta + 1}  \wh{\Psi}_p(n/R) n^{-1} (n-1)^{-d-1} + C R^d\sum_{ n > R^\delta + 1}  \wh{\Psi}_p(n/R) n^{-1} (n-1)^{-d-1},
  \end{align*}
  where $\delta = d/(d+1)$.
  This implies
  \begin{align*}
  \int_{\|x\|> R^{-1}}  \ell_R(x) \Psi_p(x) dx
 &\leq C \left( \sup_{h\leq R^{-1} + R^{\delta-1}} \frac{ \wh{\Psi}_p (h) }{h^{d}} \right) \left(\sum_{2\leq n \leq R^\delta + 1} \frac{n^{d-1}}{(n-1)^{d+1}} \right) \\
 & \qquad + C\wh{\Psi}_p(\infty) \sum_{ n > R^\delta + 1}  \frac{  n^{-1} R^d}{(n-1)^{d+1}} \\
 &\leq C\left( \sup_{h\leq R^{-1} + R^{\delta-1}} \wh{\Psi}_p(h) h^{-d} \right) + C   \,.
 \end{align*}
Therefore,  
\[
\limsup_{R\to+\infty} \sigma_{p,R}^2 R^{-d} \leq C   + C \limsup_{R\to+\infty}     \wh{\Psi}_p(h) h^{-d} < \infty \,.
\]
This finishes our proof. \qedhere

\end{proof}

  \begin{proof}[Proof of Lemma \ref{lem2.1}]
  Notice that  the condition $f_p \in L^1(\R^{pd})$ implies that $\F f_p$ is uniformly continuous and bounded. We fix a generic $z\in\R^d$, and we write
  \begin{align*}
  |   \Psi_p(x) -  \Psi_p(z) | &  \leq  \int_{\R^{pd-d}}  \Bigg\vert   \vert \F f_p\vert^2\big(\pmb{\xi_{p-1}}, x -\tau( \pmb{\xi_{p-1}}) \big)  
 \varphi\big(x -\tau(  \pmb{\xi_{p-1}})\big)   \\
 &  \qquad -  \vert \F f_p\vert^2\big(\pmb{\xi_{p-1}},  z-\tau( \pmb{\xi_{p-1}})\big)  
 \varphi\big( z-\tau(  \pmb{\xi_{p-1}})\big)   \Bigg\vert   ~
   \prod_{i=1} ^{p-1} \varphi(\xi_i) 
    d \pmb{\xi_{p-1}}  \\
    &  \leq  A_1(x) + A_2(x),
  \end{align*} 
  where
  \begin{align*}
   A_1(x) &:= \int_{\R^{pd-d}}  \Bigg\vert   \vert \F f_p\vert^2  \big(\pmb{\xi_{p-1}}, x -\tau( \pmb{\xi_{p-1}})\big) 
  -  \vert \F f_p\vert^2\big( \pmb{\xi_{p-1}}, z -\tau( \pmb{\xi_{p-1}})\big)   \Bigg\vert \\
&\qquad\qquad \times \varphi\big(x -\tau( \pmb{\xi_{p-1}}) \big) 
   \prod_{i=1} ^{p-1} \varphi(\xi_i) 
    d \pmb{\xi_{p-1}} 
    \end{align*}
    and
  \begin{align*}
  A_2(x) &:= \int_{\R^{pd-d}}  
  \vert \F f_p\vert^2\big(\pmb{\xi_{p-1}},  z-\tau( \pmb{\xi_{p-1}}) \big)   
  \Bigg\vert \varphi\big(x -\tau(\pmb{\xi_{p-1}})\big) - \varphi\big(z-\tau( \pmb{\xi_{p-1}})\big) \Bigg\vert  \\
 &\qquad\qquad\qquad \times   \prod_{i=1} ^{p-1} \varphi(\xi_i) 
    d \pmb{\xi_{p-1}}.
  \end{align*} 
\noindent{\it Estimation of   $A_1$}: We write
    \begin{align*}
    A_1(x)  & \leq  \sup_{ \pmb{\eta_{p-1}} \in \R^{pd-d} }   \Big|  \vert \F f_p\vert^2\big(\pmb{\eta_{p-1}}, x -\tau( \pmb{\eta_{p-1}}) \big) 
  - \vert \F f_p\vert^2\big(\pmb{\eta_{p-1}},  z-\tau( \pmb{\eta_{p-1}}) \big)  \Big| \\
  & \qquad\qquad\qquad \times   \int_{\R^{pd-d}} \varphi\big(x -\tau( \pmb{\xi_{p-1}}) \big) 
   \prod_{i=1} ^{p-1} \varphi(\xi_i)  d \pmb{\xi_{p-1}}.
    \end{align*}
    The first factor  tends to zero as $x\to 0$, due to the uniform continuity of $\F f_p$. We rewrite the second factor as the $p$-convolution $\varphi^{\ast p}(x)  $ and we deduce from \eqref{conv-bdd} that
\[
\big\| \varphi^{\ast p}  \big\| _\infty  \leq    \| \varphi \|_{L^q(\R^d)} ^p.
\]
 Thus,  we obtain that $A_1(x)\to 0$, as $x\to 0$.  Moreover, the previous computations  also lead to 
  \[
  A_1(x) \leq  \big\|  \vert \F f_p\vert^2 \big\| _\infty  \| \varphi \|_{L^q(\R^d)} ^p < \infty \,.
  \]

  \medskip

\noindent{\it Estimation of  $A_2$}: Using the boundedness of  $\F f_p$, we write 
\begin{align*}
  A_2(x) &\leq C \int_{\R^{pd-d}}      \Big\vert \varphi\big(x -\tau(\pmb{\xi_{p-1}})\big) - \varphi\big(z-\tau( \pmb{\xi_{p-1}})\big) \Big\vert   \prod_{i=1} ^{p-1} \varphi(\xi_i)  d \pmb{\xi_{p-1}} \\
  &=  C \int_{\R^d}  dy    \big\vert \varphi (x - y  ) - \varphi(z-y ) \big\vert   \left( \int_{\R^{pd-2d}}  \varphi\big( y - \tau(\pmb{\xi_{p-2}}) \big) \prod_{i=1} ^{p-2} \varphi(\xi_i)   d\pmb{\xi_{p-2}} \right) \\
  &= C\int_{\R^d}     \big\vert \varphi (x - y  ) - \varphi(z-y ) \big\vert  \varphi^{\ast p-1} (y) \, dy \\
  &\leq C\left(\int_{\R^d}     \big\vert \varphi (x - y  ) - \varphi(z-y ) \big\vert^q dy \right)^{1/q}  \|  \varphi^{\ast p-1} \| _{L^p(\R^d)},
    \end{align*}
where we made the change of variables $\pmb{\xi_{p-1}}\to (  \pmb{\xi_{p-2}}, y- \tau(\pmb{\xi_{p-2}})    )$ in the first equality. We know from the proof of \eqref{conv-bdd}  that $ \|  \varphi^{\ast p-1} \| _{L^p(\R^d)} \leq  \| \varphi \|_{L^q(\R^d)} ^{p-1}$, so
  \[
    A_2(x) \le C \| \varphi \|_{L^q(\R^d)} ^{p-1}   \left(\int_{\R^d}  \big\vert \varphi(x -y) - \varphi(z-y) \big\vert^q  dy\right)^{1/q} \xrightarrow{x\to z} 0\,.
  \]
  The above bound also indicates that $A_2$ is uniformly bounded.  
  
  \medskip 
  
 Hence we conclude our proof by combining the above two estimates.    \end{proof}

  \bigskip

  \begin{proof}[Proof of Lemma \ref{lem2.3}] Let us first prove the boundedness.  Since $f_p \in L^1(\R^{pd})$,  $\F f_p$ is uniformly bounded, so that
\begin{align*}
\big\vert \Psi_p^{(r,\delta)}(x,y) \big\vert \leq C  \varphi^{\ast p}(x) \varphi^{\ast p}(y) \leq C \| \varphi\| _{L^q(\R^d)}^{2p} \,, 
\end{align*}
 where the last inequality follows from \eqref{conv-bdd}.     Now let us show the continuity. To ease the presentation, we define 
\begin{align*}
&\mathbf{M}_{x,y}\equiv\mathbf{M}_{x,y}\big(\pmb{\xi_r}, \pmb{\wt{\xi}_{r-1}},\pmb{\eta_{p-r}} , \pmb{\wt{\eta}_{p-r-1}} \big)\\
&= \vert\F f_p\vert^2  \big(\pmb{\eta_{p-r}}, \pmb{\wt{\xi}_{r-1} },  x-  \tau ( \pmb{\wt{\xi}_{r-1}} ) - \tau( \pmb{\eta_{p-r}})  \big)   \vert \F f_p\vert^2\big(\pmb{\wt{\eta}_{p-r-1}}, y-  \tau(\pmb{\xi_r }) - \tau( \pmb{\wt{\eta}_{p-r-1}}),    \pmb{\xi_r} \big) .
\end{align*}
Suppose $x_n, y_n\in\R^d$ converge to $x$ and $y$ respectively, as $n\to+\infty$. Then
\begin{align*}
&\quad \big\vert \Psi_p^{(r,\delta)}(x,y) -  \Psi_p^{(r,\delta)}(x_n,y_n) \big\vert   \\
&\leq   \int_{\R^{2pd-2d}} d\pmb{\xi_r} d\pmb{\wt{\xi}_{r-1}} d\pmb{\eta_{p-r}} d\pmb{\wt{\eta}_{p-r-1}} {\bf 1}_{ \{  \|    \tau ( \pmb{\xi_r} )   + \tau( \pmb{\eta_{p-r}} )   \| < \delta \}       } \left(\prod_{i=1}^{r-1} \varphi(\xi_i)  \varphi(\wt{\xi}_i) \right) \varphi(\xi_r) \varphi(\eta_{p-r})      \\
  &\qquad  \times \left(\prod_{j=1}^{p-r-1} \varphi(\eta_j)  \varphi(\wt{\eta}_j) \right)\Bigg\vert \mathbf{M}_{x,y}  \varphi\big( y - \tau(\pmb{\xi_r}) -\tau(\pmb{\wt{\eta}_{p-r-1}}) \big)  \varphi\big( x - \tau(\pmb{\wt{\xi}_{r-1}}) -\tau(\pmb{\eta_{p-r}}) \big)  \\
  &\qquad -  \mathbf{M}_{x_n, y_n} \varphi\big( y_n - \tau(\pmb{\xi_r}) -\tau(\pmb{\wt{\eta}_{p-r-1}}) \big)  \varphi\big( x_n - \tau(\pmb{\wt{\xi}_{r-1}}) -\tau(\pmb{\eta_{p-r}}) \big)  \Bigg\vert \leq A_{1,n} +  A_{2,n}\,,
  \end{align*}
where  
\begin{align*}
A_{1,n}=&   \int_{\R^{2pd-2d}} d\pmb{\xi_{r}} d\pmb{\wt{\xi}_{r-1}} d\pmb{\eta_{p-r}} d\pmb{\wt{\eta}_{p-r-1}} {\bf 1}_{ \{  \|    \tau ( \pmb{\xi_r} )   + \tau( \pmb{\eta_{p-r}} )   \| < \delta \}       } \left(\prod_{i=1}^{r-1} \varphi(\xi_i)  \varphi(\wt{\xi}_i) \right) \varphi(\xi_r) \varphi(\eta_{p-r})      \\
  &   \times \left(\prod_{j=1}^{p-r-1} \varphi(\eta_j)  \varphi(\wt{\eta}_j) \right)  \varphi\big( y - \tau(\pmb{\xi_r}) -\tau(\pmb{\wt{\eta}_{p-r-1}}) \big)  \varphi\big( x - \tau(\pmb{\wt{\xi}_{r-1}}) -\tau(\pmb{\eta_{p-r}}) \big)  \\
  &\qquad \times \big\vert \mathbf{M}_{x,y} - \mathbf{M}_{x_n, y_n} \big\vert 
    \end{align*}
    and
\begin{align*}
A_{2,n}&=   \int_{\R^{2pd-2d}} d\pmb{\xi_{r}} d\pmb{\wt{\xi}_{r-1}} d\pmb{\eta_{p-r}} d\pmb{\wt{\eta}_{p-r-1}} {\bf 1}_{ \{  \|    \tau ( \pmb{\xi_r} )   + \tau( \pmb{\eta_{p-r}} )   \| < \delta \}       } \left(\prod_{i=1}^{r-1} \varphi(\xi_i)  \varphi(\wt{\xi}_i) \right) \varphi(\xi_r) \varphi(\eta_{p-r})      \\
  &   \times \left(\prod_{j=1}^{p-r-1} \varphi(\eta_j)     \varphi(\wt{\eta}_j) \right)  \mathbf{M}_{x_n, y_n} \Bigg\vert \varphi\big( y - \tau(\pmb{\xi_r}) -\tau(\pmb{\wt{\eta}_{p-r-1}}) \big)  \varphi\big( x - \tau(\pmb{\wt{\xi}_{r-1}}) -\tau(\pmb{\eta_{p-r}}) \big) \\
  &\qquad\qquad\qquad\qquad   - \varphi\big( y_n - \tau(\pmb{\xi_r}) -\tau(\pmb{\wt{\eta}_{p-r-1}}) \big)  \varphi\big( x_n - \tau(\pmb{\wt{\xi}_{r-1}}) -\tau(\pmb{\eta_{p-r}}) \big)  \Bigg\vert\,.
    \end{align*}
It follows immediately from the first part of our proof that 
\[
A_{1,n}\leq C \| \varphi\|_{L^q(\R^d)}^{2p}   \sup\Big\{   \vert  \mathbf{M}_{x_n,y_n} -\mathbf{M}_{x,y}  \vert:\pmb{\xi_r}, \pmb{\wt{\xi}_{r-1}},\pmb{\eta_{p-r}} , \pmb{\wt{\eta}_{p-r-1}} \Big\} \xrightarrow{n\to+\infty} 0\,,
\]
due to the uniform continuity of $\F f_p$.  Now, using   $\| \F f_p\| _\infty<\infty$, we write 
 \begin{align*}
A_{2,n} &\leq  C  \int_{\R^{2pd-2d}} d\pmb{\xi_{r}} d\pmb{\wt{\xi}_{r-1}} d\pmb{\eta_{p-r}} d\pmb{\wt{\eta}_{p-r-1}}   \left(\prod_{i=1}^{r-1} \varphi(\xi_i)  \varphi(\wt{\xi}_i) \right) \varphi(\xi_r) \varphi(\eta_{p-r})      \\
  &    \times \left(\prod_{j=1}^{p-r-1} \varphi(\eta_j)     \varphi(\wt{\eta}_j) \right)    \Bigg\vert \varphi\big( y - \tau(\pmb{\xi_r}) -\tau(\pmb{\wt{\eta}_{p-r-1}}) \big)  \varphi\big( x - \tau(\pmb{\wt{\xi}_{r-1}}) -\tau(\pmb{\eta_{p-r}}) \big) \\
  &       - \varphi\big( y_n - \tau(\pmb{\xi_r}) -\tau(\pmb{\wt{\eta}_{p-r-1}}) \big)  \varphi\big( x_n - \tau(\pmb{\wt{\xi}_{r-1}}) -\tau(\pmb{\eta_{p-r}}) \big)  \Bigg\vert \leq C ( A_{21,n} + A_{22,n}  ), 
    \end{align*}
with
\begin{align*}
&A_{21,n}: =\int_{\R^{2pd-2d}} d\pmb{\xi_{r}} d\pmb{\wt{\xi}_{r-1}} d\pmb{\eta_{p-r}} d\pmb{\wt{\eta}_{p-r-1}}   \left(\prod_{i=1}^{r-1} \varphi(\xi_i)  \varphi(\wt{\xi}_i) \right) \varphi(\xi_r) \varphi(\eta_{p-r})      \\
  & \quad  \times \left(\prod_{j=1}^{p-r-1} \varphi(\eta_j)     \varphi(\wt{\eta}_j) \right)    \Bigg\vert \varphi\big( y - \tau(\pmb{\xi_r}) -\tau(\pmb{\wt{\eta}_{p-r-1}}) \big) - \varphi\big( y_n - \tau(\pmb{\xi_r}) -\tau(\pmb{\wt{\eta}_{p-r-1}}) \big)   \Bigg\vert \\
  &\qquad \times \varphi\big( x - \tau(\pmb{\wt{\xi}_{r-1}}) -\tau(\pmb{\eta_{p-r}}) \big) \\
  &= \varphi^{\ast p}(x) \int_{\R^{pd-d}} d\pmb{\xi_{p-1}}     \left(\prod_{i=1}^{p-1} \varphi(\xi_i)\right) \Big\vert \varphi\big( y - \tau(\pmb{\xi_{p-1}})   \big) - \varphi\big( y_n -\tau(\pmb{\xi_{p-1}})  \big)   \Big\vert 
  \end{align*}
and smilarly,
\begin{align*}
A_{22,n}:= \varphi^{\ast p}(y_n) \int_{\R^{pd-d}} d\pmb{\xi_{p-1}}     \left(\prod_{i=1}^{p-1} \varphi(\xi_i)\right) \Big\vert \varphi\big( x - \tau(\pmb{\xi_{p-1}})   \big) - \varphi\big( x_n -\tau(\pmb{\xi_{p-1}})  \big)   \Big\vert \,.
\end{align*}
Put $\varphi_y(x) = \varphi(x-y)$, so we can rewrite 
\[
 \int_{\R^{pd-d}} d\pmb{\xi_{p-1}}     \left(\prod_{i=1}^{p-1} \varphi(\xi_i)\right) \Big\vert \varphi\big( x - \tau(\pmb{\xi_{p-1}})   \big) - \varphi\big( x_n -\tau(\pmb{\xi_{p-1}})  \big)   \Big\vert  
 \]
as $\big( \varphi^{\ast p-1} \ast \vert \varphi_{-x} -\varphi_{-x_n }\vert \big) (0) $, which is bounded by 
\[
\big\| \varphi^{\ast p-1} \big\| _{L^p(\R^d)}  \| \varphi_{-x} -\varphi_{-x_n }\|_{L^q(\R^d)} \leq  \big\| \varphi\big\| _{L^q(\R^d)}^{ p-1 } \| \varphi_{-x} -\varphi_{-x_n }\|_{L^q(\R^d)} \xrightarrow{n\to+\infty} 0\,,
\]
that is, $A_{22,n}\to 0$, as $n\to+\infty$. The same arguments also imply that $A_{21,n}\to 0$, as $n\to+\infty$. This concludes our proof.\qedhere

\end{proof}

\begin{lemma} Let $\varphi_1$ be given as in  Theorem \ref{thmSHE}. Then   for any $x,y\in \R^d$ and $s>0$, we have 
\begin{equation} \label{ecu1}
\int_{\R^d} e^{-s \| \eta\|^2}  \varphi_1(\eta -x) \varphi_1(y-\eta) d\eta \le  \int_{\R^d} e^{-s \| \eta\|^2}  \varphi_1^2(\eta)d\eta
\end{equation}
and
\begin{align}\label{ecuu1}
\int_{\R^d} e^{-s \| \eta\|^2}  \varphi_1(\eta -x) d\eta \le  \int_{\R^d} e^{-s \| \eta\|^2}  \varphi_1(\eta)d\eta.
\end{align}

\end{lemma}

\begin{proof} It suffices to prove it for  $x=y$, as the general case follows from the Cauchy-Schwarz inequality and symmetry of $\varphi_1$.

Put $h(\eta)=e^{-s \| \eta\|^2} $, then its Fourier transform $\F h$ is a nonnegative function.  Then,  we write, using Plancherel's identity and the fact $\varphi_1^2 = \frac{1}{(2\pi)^{2d}} \F (\gamma_1\ast \gamma_1)$
\begin{align*}
& \int_{\R^d} h(\eta)  \varphi_1(\eta -x)^2  d\eta = \int_{\R^d} h(\eta+x) \frac{1}{(2\pi)^{2d}} \F (\gamma_1\ast \gamma_1)(\eta)  d\eta  \\
&=  \int_{\R^d} (\F h)(a) e^{\i ax}  \frac{1}{(2\pi)^{2d}}  (\gamma_1\ast \gamma_1)(a)  da     \quad (\text{$\gamma_1$ is also nonnegative})  \\
&\leq  \int_{\R^d} (\F h)(a)  \frac{1}{(2\pi)^{2d}}  (\gamma_1\ast \gamma_1)(a)  da   = \int_{\R^d} h(\eta)  \varphi_1(\eta )^2  d\eta\,,
\end{align*}
which proves \eqref{ecu1}. The same argument also leads easily  to \eqref{ecuu1}. 
\end{proof}

 {\noindent\bf Acknowledgement:} We thank the anonymous referee for suggesting us to add a functional version of Theorem \ref{Nchaotic}, which is the content of Theorem \ref{NCF}.

\end{document}